\documentclass[11pt,reqno]{amsart}

\usepackage[colorlinks,linkcolor={blue},
citecolor={blue},urlcolor={blue}]{hyperref}

\usepackage{mathrsfs}
\usepackage{bbm}
\usepackage{amsfonts}

\usepackage{}
\usepackage{pifont}
\usepackage[shortlabels]{enumitem}
\usepackage{empheq}
\usepackage{amsfonts}
\usepackage{amsmath, amsrefs, amsthm, amssymb}
\usepackage[active]{srcltx}		
\usepackage{pdfsync}					
\usepackage{graphicx}
\usepackage [latin1]{inputenc}
\usepackage{bbm}
\usepackage{amsfonts}
\usepackage{amsthm}
\usepackage{graphicx}
\usepackage{framed}
\usepackage{multicol,multienum}
\usepackage{graphicx} 
\usepackage{booktabs} 
\usepackage{mathrsfs}
\usepackage{tcolorbox}
\usepackage{color}
\usepackage{xcolor}

\newtheorem{theorem}{Theorem}[section]
\newtheorem{proposition}[theorem]{Proposition}
\newtheorem{lemma}[theorem]{Lemma}
\newtheorem{corollary}[theorem]{Corollary}

\theoremstyle{definition}

\newtheorem{assumption}[theorem]{Assumption}

\newtheorem{remark}[theorem]{Remark}

\numberwithin{equation}{section}

\begin{document}

\title [Stochastic Chemotaxis-Fluid model in $\mathbb{R}^3$]{Global strong solution for the stochastic tamed Chemotaxis-Navier-Stokes system in $\mathbb{R}^3$}

\author{Fan Xu}
\address{School of Mathematics and Statistics, Hubei Key Laboratory of Engineering Modeling  and Scientific Computing, Huazhong University of Science and Technology,  Wuhan 430074, Hubei, P.R. China.}
\email{d202280019@hust.edu.cn (F. Xu)}

\author{Lei Zhang}
\address{School of Mathematics and Statistics, Hubei Key Laboratory of Engineering Modeling  and Scientific Computing, Huazhong University of Science and Technology,  Wuhan 430074, Hubei, P.R. China.}
\email{lei\_zhang@hust.edu.cn (L. Zhang)}

\author{Bin Liu}
\address{School of Mathematics and Statistics, Hubei Key Laboratory of Engineering Modeling  and Scientific Computing, Huazhong University of Science and Technology,  Wuhan 430074, Hubei, P.R. China.}
\email{binliu@mail.hust.edu.cn (B. Liu)}


\date{\today}

\begin{abstract}
In this work, we consider the 3D Cauchy problem for a coupled system arising in biomathematics, consisting of a chemotaxis model with a cubic logistic source and the stochastic tamed Navier-Stokes equations (STCNS, for short). Our main goal is to establish the existence and uniqueness of a global strong solution (strong in both the probabilistic and PDE senses) for the 3D STCNS system with large initial data. To achieve this, we first introduce a triple approximation scheme by using the Friedrichs mollifier, frequency truncation operators, and cut-off functions. This scheme enables the construction of sufficiently smooth approximate solutions and facilitates the effective application of the entropy-energy method. Then, based on a newly derived stochastic version of the entropy-energy inequality, we further establish some a priori higher-order energy estimates, which together with the stochastic compactness method, allow us to construct the strong solution for the STCNS system.
\end{abstract}

\maketitle
\section{Introduction}
\subsection{Statement of the problem}
Chemotaxis is a fundamental biological mechanism observed across a wide range of organisms, including bacteria, unicellular eukaryotes, and multicellular entities. It refers to the ability of these organisms to sense and respond to variations in chemical concentration gradients in their environment \cite{keller1970,keller1971}. Functionally, chemotaxis acts as a navigational strategy, guiding organisms toward favorable conditions or away from harmful environments. In 2005, Tuval et al. introduced the Chemotaxis-Navier-Stokes (CNS) model to capture the dynamic behavior of bacteria suspended in fluid media \cite{tuval2005}. This model describes the interplay between bacterial density, chemical concentration, and fluid dynamics, accounting for the impact of bacterial density on fluid velocity through gravitational effects. To reflect more realistic conditions, subsequent studies have extended the CNS system by incorporating random external forces into the model \cite{hausenblas2024existence,zhai20202d,zhang2022global}.

In the deterministic setting, many authors have thoroughly investigated the impact of fluid viscosity on chemotaxis phenomena within the framework of partial differential equations. The relevant results primarily involve the global well-posedness \cite{chae2013existence,ding2022generalized,duan2010global,jeong2022well,liu2011coupled,lorz2010coupled,winkler2012global,winkler2016global,winkler2017far,winkler2022does,zhang2014global,zhang2021global} and asymptotic behavior \cite{chae2014global,17di2010chemotaxis,duan2017global,jiang2015global}, providing important insights into the intricate dynamics involved.  It is worth emphasizing that in the three-dimensional whole space, Liu and Lorz \cite{liu2011coupled} proved the existence of global solutions for the Chemotaxis-Stokes system with nonlinear diffusion. Later, Kang, Lee, and Winkler \cite{28Kyungkeun} demonstrated the existence of global weak solutions for the CNS system in $\mathbb{R}^3$. Additionally, Zhang and Zheng \cite{zhang2021global} established the existence and uniqueness of global axisymmetric weak solutions for the 3D CNS system. Furthermore, the first and third authors of this paper proved the existence and uniqueness of a global large-data smooth solution for the 3D tamed Chemotaxis-Navier-Stokes system in $\mathbb{R}^3$ \cite{FanXu2024}.

Compared with the deterministic case, available studies on stochastic versions of the CNS system remain limited. In 2020, Zhai and Zhang \cite{zhai20202d} first established the existence of unique global mild and weak solutions to the 2D stochastic Chemotaxis-Navier-Stokes (SCNS) system in a bounded convex domain. Later, in a 3D general bounded domain, Zhang and Liu \cite{zhang2022global} showed that the SCNS system driven by L\'{e}vy-type noise admits at least one global martingale weak solution. Moreover, Hausenblas et al. \cite{hausenblas2024existence} proved the existence and uniqueness of global pathwise weak solutions to the SCNS system in a 2D bounded domain with an additional noise term in the chemical concentration equation. Recently, Zhang and Liu \cite{zhang2025keller} demonstrated the global well-posedness of the SCNS system driven by multiplicative Gaussian noise in the whole space $\mathbb{R}^2$.

Inspired by the aforementioned works in both 2D unbounded domains and 3D bounded domains, a natural mathematical question arises: \emph{Does the stochastically perturbed CNS system admit a unique global solution in the physical space $\mathbb{R}^3$?} The main purpose of this study is to provide a partial affirmative answer to this question. To the best of our knowledge, the global solvability of the SCNS system in $\mathbb{R}^3$ remains unaddressed in the literature. In fact, fully resolving this problem remains a long-term challenge. One major difficulty stems from the open problem of whether the 3D (stochastic) Navier-Stokes equations with large initial data admit a unique global weak/strong solution. Another challenge arises from the potential blow-up induced by the cross-diffusion term in the cell density equation under the three-dimensional setting.

In this paper, we investigate a suitably regularized version of the SCNS system that retains its essential biological and physical features. More precisely, we are interested in the mathematical analysis of the following Cauchy problem for the chemotaxis model coupled with the 3D stochastic tamed Navier-Stokes equations (STCNS):
\begin{equation}\label{CNS}
\left\{
\begin{aligned}
&\mathrm{d}n+u\cdot\nabla n\mathrm{d}t=\Delta n\mathrm{d}t-\nabla\cdot(n\nabla c)\mathrm{d}t+L(n)\mathrm{d}t,& & \textrm{in}~ \mathbb{R}_+\times \mathbb{R}^3,\\
&\mathrm{d}c+u\cdot\nabla c\mathrm{d}t=\Delta c\mathrm{d}t-nc\mathrm{d}t,&&\textrm{in}~ \mathbb{R}_+\times\mathbb{R}^3,\\
&\mathrm{d}u+(u\cdot\nabla)u\mathrm{d}t=\Delta u\mathrm{d}t+\nabla P\mathrm{d}t-\mathbf{g}(|u|^2)u\mathrm{d}t+n\nabla\phi \mathrm{d}t+G(u)\mathrm{d}W(t),&&\textrm{in}~\mathbb{R}_+\times\mathbb{R}^3,\cr
&\nabla\cdot u=0,&&\textrm{in}~ \mathbb{R}_+\times\mathbb{R}^3,\\
&n|_{t=0}=n_0,~c|_{t=0}=c_0,~u|_{t=0}=u_0,&&\textrm{in}~\mathbb{R}^3,
\end{aligned}
\right.
\end{equation}
where $\mathbb{R}_+ := [0,\infty)$, and the unknown functions $n(t,x)$, $c(t,x)$, and $u(t,x)$ represent the cell density, the chemical concentration, and the fluid velocity field, respectively. The scalar function $P(t,x)$ stands for the pressure, and $\phi(x)$ represents the gravitational potential. The force $L:\mathbb{R}\rightarrow\mathbb{R}$ is a smooth function generalizing the logistic source, modeling cell proliferation and apoptosis. In the following, we consider the cubic logistic term in the form of
\begin{equation}\label{gggg}
\begin{split}
L(n):=n(1-n)(n-a),\quad \textrm{for some}~  a\in(0,\frac{1}{2}),
 \end{split}
\end{equation}
which quantitatively reflects the experimental observations of Mimura and Tsujikawa \cite{mimura1996aggregating}. Note that the Navier-Stokes equations in \eqref{CNS} are influenced by a tamed term $\mathbf{g}(|u|^2)u$ (see \eqref{con1} below for more details). This tamed term is a useful feedback mechanism for complex biological fluid environments, originally introduced by R\"{o}ckner and Zhang \cite{rockner2009stochastic,rockner2009tamed} to describe the nonlinear self-regulation of fluid flows in high-speed regimes, which suppresses nonlinear growth and reduces instability. In \eqref{CNS}, besides the deterministic force $n\nabla\phi $ caused by the cells via the time-independent potential $\phi(x)$, the system is also perturbed by a random external force in the form of $G(u)\mathrm{d}W(t)$, which accounts for unpredictable measurement noise and modeling errors in approximating physical phenomena. Here, $\{W(t)\}_{t\geq 0}$  is an $\mathbf{U}$-valued cylindrical Wiener process defined on a fixed stochastic basis $(\Omega,\mathcal {F},\{\mathcal {F}_t\}_{t>0},\mathbb{P})$ with complete right-continuous filtration, and it is formally given by the expansion
$$
W(t)=\sum_{j\geq 1}W_j(t )e_j,
$$
where $\{W_j(t)\}_{j\geq1}$ is a family of mutually independent $\mathcal {F}_t$-adapted real-valued standard Wiener processes, and $\{e_j\}_{j\geq1}$ is a complete orthonormal basis in the separable Hilbert space $\mathbf{U}$ \cite{da2014stochastic}. To make sense of $W(t)$, one can consider a larger auxiliary space $\mathbf{U}_0$ via
$
 \mathbf{U}_0:=\{v=\sum_{j\geq 1}\alpha_j e_j;~\sum_{j\geq 1} \alpha_j^2/j^2  <\infty\} \supset \mathbf{U},
$
which is endowed with the norm $\|f\|_{\mathbf{U}_0}^2=\sum_{j\geq 1} \alpha_j^2/j^2$, for any $f=\sum_{j\geq 1} \alpha_j e_j$. In particular, the embedding $\mathbf{U} \hookrightarrow \mathbf{U}_0$ is Hilbert-Schmidt \cite{da2014stochastic}, ensuring that the trajectories of $W$ lie in $\mathcal{C}([0,T];\mathbf{U}_0)$ almost surely.

The key novelties of this work are summarized as follows: First, it seems that the classical Faedo-Galerkin method and energy estimates used in \cite{11brzezniak2013existence,rockner2009stochastic} may not be straightforwardly applied to the coupled system considered in the present paper, and we shall introduce an approximation system by borrowing some ideas from our recent works \cite{zhang2025ejp,zhang2025keller}, which enable us to establish the existence and uniqueness of the global smooth approximation solutions, and in particular enable us to establish a crucial stochastic entropy-energy inequality, which is necessary for constructing the exact strong solution by using the stochastic compactness method. Second, in addition to proving the existence of solutions in the whole three-dimensional space, we also establish the uniqueness of the strong solution, which has not yet been addressed even in the bounded 3D setting (cf. \cite{zhang2022global}). Moreover, the unboundedness of the domain poses additional challenges in applying the compactness argument, preventing us from utilizing the approximation system introduced in \cite{zhang2022global}. Third, we show that the tamed term in the stochastic Navier-Stokes equations and the cubic logistic term in the cell density equation ensure the regularity and uniqueness of the solutions. An interesting remaining question is to explore whether the main result remains valid for weaker logistic term, such as $L(n)=n-n^\alpha$ with $0<\alpha <3$. Note that this type of logistic source also has a biologically meaningful interpretation (cf. \cite{ding2022generalized}), and we shall provide partial answer to this problem in our forthcoming work.

\subsection{Notations}\label{nnnn}
In this paper, we adopt the following conventions: The inequality $A \lesssim_{a,b,\cdots} B$ means that there exists a positive constant $C$ depending only on $a, b, \cdots$ such that $A \leq C B$, while $A \asymp_{a,b,\cdots} B$ indicates that there exist two positive constants $C \leq C'$ depending only on $a, b, \cdots$ such that $C B \leq A \leq C' B$.

Let $X$ and $Y$ be two Banach spaces. We denote by $\mathcal{L}(X; Y)$ the space of all bounded linear operators from $X$ to $Y$. The symbol $\langle \cdot, \cdot \rangle_{X', X}$ stands for the standard duality pairing, where $X' := \mathcal{L}(X; \mathbb{R})$ is the dual space of $X$. In particular, if $X$ is a Hilbert space, then the symbol $(\cdot, \cdot)_X$ denotes the scalar product. If $X$ and $Y$ are separable Hilbert spaces, then the symbol $\mathcal{L}_2(X; Y)$ denotes the Hilbert space of all Hilbert-Schmidt operators from $X$ to $Y$.

Let $\mathcal{A}=(\mathcal{A},\Sigma,\mu)$ be a measure space, where $\Sigma$ is a $\sigma$-algebra on $\mathcal{A}$, and $\mu$ is a non-negative measure. Let $\mathcal{B}$ be a Banach space over $\mathbb{R}$. For $p\in[1,\infty)$: The space $L^p(\mathcal{A};\mathcal{B})$ is defined as the set of all (equivalence classes of) strongly measurable functions $f:\mathcal{A}\mapsto \mathcal{B}$ such that $\|f\|_{\mathcal{B}}^p$ is integrable over $\mathcal{A}$ with respect to $\mu$. For $p=\infty$: The space $L^\infty(\mathcal{A};\mathcal{B})$ is defined as the set
of all (equivalence classes of) strongly measurable functions $f:\mathcal{A}\mapsto \mathcal{B}$ that are essentially bounded on $\mathcal{A}$. Let $\textrm{Lip}([0,T];\mathcal{B})$ denote the space of Lipschitz continuous functions from the interval $[0,T]$ to the Banach space $\mathcal{B}$.

Let $\mathcal{V}:=\{f\in \mathcal{C}_0^{\infty}(\mathbb{R}^3;\mathbb{R}^3):\nabla\cdot f=0\}$, where $\mathcal{C}_0^{\infty}(\mathbb{R}^3;\mathbb{R}^3)$ denotes the space of all $\mathbb{R}^3$-valued functions of class $\mathcal{C}^{\infty}$ with compact supports. Let $\Lambda^s:=(I-\Delta)^{\frac{s}{2}}$ denote the Bessel operator \cite{bahouri2011fourier}. Let $H^s(\mathbb{R}^3)$ be a Hilbert space endowed with the norm
\begin{equation*}
\begin{split}
\|f\|_{H^s}=\|\Lambda^sf\|_{L^2}=\left[\int_{\mathbb{R}^3}(1+|\xi|^2)^s|\hat{f}(\xi)|^2\textrm{d}\xi\right]^{\frac{1}{2}},
 \end{split}
 \end{equation*}
where $\hat{f}$ denotes the Fourier transform of a tempered distribution $f$. Let
\begin{equation*}
\begin{split}
\mathbb{H}^s:=\textrm{the closure of}~\mathcal{V}~ \textrm{in}~H^s(\mathbb{R}^3;\mathbb{R}^3).
\end{split}
\end{equation*}

Let $Q_w$ denote the Hilbert space $Q$ endowed with the weak topology. Let $\{\mathcal{O}_d\}_{d\in\mathbb{N}}$ be a sequence of bounded open subsets of $\mathbb{R}^3$ with regular boundaries $\partial\mathcal{O}_d$ such that $\mathcal{O}_d\subset\mathcal{O}_{d+1}$ and $\bigcup_{d=1}^{\infty}\mathcal{O}_d=\mathbb{R}^3$.  Let $H^k(\mathcal{O}_d)$ denote the space of restrictions of functions defined on $\mathbb{R}^3$ to subsets $\mathcal{O}_d$, i.e.
$$H^k(\mathcal{O}_d):=\{f|_{\mathcal{O}_d}; f\in H^k(\mathbb{R}^3)\},\quad k\geq0.$$
We will use the following spaces in the subsequent content, analogous to those considered in \cite{11brzezniak2013existence}:
\begin{equation*} \begin{split}
\mathcal{C}([0,T];Q_w):=& \textrm{the~space~of~weakly~continuous~functions}~f:[0,T]\rightarrow Q\\
 & \textrm{with~the~weakest~topology~}\mathcal{T}_1\textrm{~such~that~for~all~}g\in Q \textrm{~the mappings}\\
 &\mathcal{C}([0,T];Q_w)\ni f\mapsto(f(\cdot),g)_{Q}\in \mathcal{C}([0,T];\mathbb{R})~\textrm{are~continuous}.\\
L^2_w(0,T;Q):=& \textrm{the~space}~L^2(0,T;Q)~\textrm{endowed~with~the~weak~topology~}\mathcal{T}_2.\\
L^2(0,T;H^k_{loc}(\mathbb{R}^3)):=& \textrm{the~space~of~measurable~functions}~f:[0,T]\rightarrow~H^k(\mathbb{R}^3)~\textrm{such~that~for~all}~d\in\mathbb{N}\\
 & p_{T,d}(f):=\|f\|_{L^2(0,T;H^k(\mathcal{O}_d))}:=\left(\int_0^T\|f\|_{H^k(\mathcal{O}_d)}^2\mathrm{d}t\right)^{\frac{1}{2}}<\infty,\\
 & \textrm{with~the~topology}~\mathcal{T}_3\textrm{~generated~by~the~seminorms}~(p_{T,d})_{d\in\mathbb{N}}.
 \end{split} \end{equation*}

\subsection{Main result}

In order to state the main theorem precisely, we begin by introducing the basic assumptions needed for the subsequent argument. Motivated by \cite{brzezniak2020stochastic,rockner2009tamed}, let us define the taming function $\mathbf{g}(\cdot)$ in the following manner.

\begin{assumption} \label{as1}
The function $\mathbf{g}(\cdot):\mathbb{R}_+\mapsto\mathbb{R}_+$ is smooth and satisfies, for some fixed $\mathbf{N}\in\mathbb{N}$
 \begin{equation}\label{con1}
\left\{
\begin{aligned}
&\mathbf{g}(r)=0,&&\textrm{if}~r\leq \mathbf{N},\cr
&\mathbf{g}(r)=r-\mathbf{N},&&\textrm{if}~r\geq \mathbf{N}+1,\cr
&\mathbf{g}'(r)\in[0,2],&&\textrm{if}~r\in[\mathbf{N},\mathbf{N}+1],\cr
&|\mathbf{g}^{(l)}(r)|\leq C_l,&&\textrm{if}~r\in[\mathbf{N},\mathbf{N}+1],\quad l=2,~3,~4,\cdots.
\end{aligned}
\right.
\end{equation}
\end{assumption}

Note that the function $\mathbf{g}(\cdot)$ defined in this way satisfies for any $r,~r_1\geq0$
\begin{equation} \begin{split}\label{con2}
|\mathbf{g}(r)|\leq r\quad\textrm{and}\quad|\mathbf{g}^{(l)}(r)-\mathbf{g}^{(l)}(r_1)|\leq C_{l+1}|r-r_1|.
\end{split} \end{equation}
Let $\mathbf{g}_1(r):=r-\mathbf{g}(r)$. Then, for any $r\geq0$, the following properties hold:
 \begin{equation}\label{connew}
\left\{
\begin{aligned}
&\mathbf{g}_1'(r)=1,&&\textrm{if}~r\leq \mathbf{N},\cr
&\mathbf{g}_1'(r)=0,&&\textrm{if}~r\geq \mathbf{N}+1,\cr
&|\mathbf{g}_1^{(l)}(r)|=|-\mathbf{g}^{(l)}(r)|\leq C_l,&&\textrm{if}~r\in[\mathbf{N},\mathbf{N}+1],\quad l=2,~3,~4,\cdots.
\end{aligned}
\right.
\end{equation}
In particular, for any $r\geq0$ and integers $m_1\geq1$, $m_2\geq0$, it follows from \eqref{con2}-\eqref{connew} that
\begin{equation} \begin{split}\label{connew1}
|\mathbf{g}_1(r)|\leq 2r,\quad |\mathbf{g}_1^{(m_1)}(r)r^{m_2}|\leq C_\mathbf{N}.
\end{split} \end{equation}

\begin{assumption} \label{as2}
Suppose that the initial data $(n_0,c_0,u_0)$ and the potential $\phi$ satisfy
\begin{equation*}
\begin{split}
&n_0 \in L^1(\mathbb{R}^3)\cap H^1(\mathbb{R}^3),~n_0>0;~c_0 \in L^1(\mathbb{R}^3)\cap H^2(\mathbb{R}^3),~|\nabla \sqrt{c_0}|\in L^2(\mathbb{R}^3),~c_0>0;\\
&u_0 \in \mathbb{H}^1;~\nabla\phi\in L^{\infty}(\mathbb{R}^3).
 \end{split}
\end{equation*}
\end{assumption}

\begin{assumption} \label{as3}
Concerning the stochastic term, we make the following assumptions:
\begin{itemize}
\item [1)] For all $s\in \mathbb{R}$, there is a constant $C>0$ such that
\begin{equation*}
\begin{split}
  \|G(u)\|_{\mathcal{L}_2(\mathbf{U};\mathbb{H}^s)}^2\leq C (1+\|u\|_{\mathbb{H}^s}^2),\quad\forall u \in \mathbb{H}^s.
 \end{split}
\end{equation*}
\item [2)] For each $M>0$, there is a constant $C_M>0$ such that
\begin{equation*}
\begin{split}
&\|G(u_1)-G(u_2)\|_{\mathcal{L}_2(\mathbf{U};\mathbb{H}^s)}\leq C_M  \|u_1-u_2\|_{\mathbb{H}^s} ,\\
\end{split}
\end{equation*}
for all $u_1,u_2 \in B_M:=\{u \in \mathbb{H}^s;~\|u\|_{\mathbb{H}^s}\leq M\}$, where $\mathcal{L}_2(\mathbf{U};\mathbb{H}^s)$ consists of all Hilbert-Schmidt operators from $\mathbf{U}$ to $\mathbb{H}^s$ equipped with the norm
$
\|G(u)\|_{\mathcal{L}_2(\mathbf{U};\mathbb{H}^s)}^2=\sum_{j\geq 1}\|G_j(u)\|_{\mathbb{H}^s}^2$,   $G_j(u)= G(u)e_j$.
\end{itemize}
\end{assumption}

Our main result can now be stated by the following theorem.

\begin{theorem}\label{th1}
Let $(\Omega,\mathcal {F},\{\mathcal {F}_t\}_{t>0},\mathbb{P})$ be the fixed stochastic basis as before, and suppose that the Assumptions \ref{as1}-\ref{as3} hold. Then, system \eqref{CNS} has a unique global strong solution $(n,c,u)$ in the following sense:
\begin{itemize}
\item [a)] For every $T>0$, there hold $\mathbb{P}$-a.s.
\begin{equation*}
\begin{split}
&n \in  \mathcal{C}([0,T];H^1(\mathbb{R}^3))\cap L^2(0,T;H^2(\mathbb{R}^3)),\\
&c \in \mathcal{C}([0,T];H^2(\mathbb{R}^3))\cap L^2(0,T;H^3(\mathbb{R}^3)),\\
&u \in  \mathcal{C}([0,T]; \mathbb{H}^1)\cap L^2(0,T;\mathbb{H}^2).
 \end{split}
\end{equation*}

\item [b)] For any $t \in (0,T]$, the integral identities
\begin{equation*}
\begin{split}
(n(t),\varphi)_{L^2}&=(n_0,\varphi)_{L^2}+\int_0^t(\Delta n(r)-u(r)\cdot\nabla n(r)-\nabla\cdot (n(r)\nabla c(r))+L(n(r)),\varphi)_{L^2}\mathrm{d}r,\\
(c(t),\varphi)_{L^2}&=(c_0,\varphi)_{L^2}+\int_0^t(\Delta c(r)-u(r)\cdot\nabla c(r)-n(r)c(r),\varphi)_{L^2}\mathrm{d}r,\\
(u(t),\psi)_{L^2}&=(u_0,\psi)_{L^2}+\int_0^t(\Delta u(r)-(u(r)\cdot\nabla)u(r)-\mathbf{g}(|u(r)|^2)u(r)+n(r) \nabla\phi,\psi)_{L^2}\mathrm{d}r\\
&\quad+\int_0^t (G(u(r))\mathrm{d}W(r),\psi)_{L^2}
 \end{split}
\end{equation*}
hold $\mathbb{P}$-a.s., for every $\varphi\in L^2(\mathbb{R}^3)$ and $\psi\in \mathbb{H}^0$.
\end{itemize}
\end{theorem}

\begin{remark}
A prototypical example satisfying the Assumption \ref{as3} takes the form of $B_0 (u) \mathrm{d} W(t)$, where $B_0$ is a \emph{linear} bounded operator from $\mathbb{H}^s$ to $\mathcal{L}_2(\mathbf{U};\mathbb{H}^s)$, i.e., $B_0\in \mathcal {L}(\mathbb{H}^s; \mathcal{L}_2(\mathbf{U};\mathbb{H}^s))$. Since $\{e_j\}_{j\geq1}$ is a complete orthogonal basis in $\mathbf{U}$, one can define the linear operators $\{B_j\}_{j\geq 1}$ by $B_j u:=B_0(u) e_j$, $j\geq 1$. Then, we have $
B_0 (u) h= \sum_{j=1}^\infty (h,e_j)_\mathbf{U} B_j u$,  $\forall h \in \mathbf{U}$. In this case, the stochastic term $
B_0 (u) \mathrm{d} W(t) $ can be written as $\sum_{j\geq 1} B_j u(t) \mathrm{d} W_j(t)$.
\end{remark}

\begin{remark}
To the best of our knowledge, Theorem \ref{th1} seems to be the first result addressing the global solvability of the SCNS system in the whole space $\mathbb{R}^3$.  In comparison with the existing results, Theorem \ref{th1} indicates that, under the appropriate regularizing effect of the logistic source \eqref{gggg} and the tamed term \eqref{con1}, it is possible to establish the existence and uniqueness of global pathwise solutions for the STCNS system \eqref{CNS} in a 3D unbounded domain. This main result improves the deterministic works by Liu and Lorz \cite{liu2011coupled}, Winkler \cite{winkler2012global}, and Xu and Liu \cite{FanXu2024}, and especially extends Theorem 1.2 in \cite{zhang2022global} by guaranteeing the uniqueness of the solution, while also overcoming the additional difficulties arising from the unboundedness of the domain.
\end{remark}

\begin{remark}
Let us briefly expound the main ideas behind the proof of Theorem \ref{th1}. In the first step, inspired by the deterministic work \cite{zhang2014global} and by the stochastic works \cite{zhang2025ejp,zhang2025keller}, we apply the Friedrichs mollifier, frequency truncation operators, and cut-off functions to the original system \eqref{CNS}, obtaining a sufficiently regularized approximation system \eqref{mod-22}. Using the standard theory of SDEs in Hilbert spaces \cite{kallianpur1995stochastic}, we establish the existence of a unique global smooth solution for the approximation system \eqref{mod-22}. Next, we perform refined uniform high-order energy estimates (independent of the frequency truncation operators) for the three equations in system \eqref{mod-22} (cf. Lemma \ref{lem2}) and verify a strong convergence result in the probabilistic sense (cf. Lemma \ref{lem2.3}). This allows us to prove the existence of a unique sufficiently smooth solution for the approximation system containing only the Friedrichs mollifier and cut-off functions (cf. Lemma \ref{lem2.4}). By further exploiting the structural properties of the approximation system \eqref{Mod-1} and employing a stopping time technique, we can eliminate the dependence on the cut-off functions, proving that the approximation system \eqref{Mod-1} admits a unique smooth solution (cf. Lemma \ref{lem2.5}). It is worth noting that the global smooth solution of system \eqref{Mod-1} is sufficient to support subsequent entropy-energy estimates (independent of the Friedrichs mollifier).

In the second step, based on the structural properties of system \eqref{CNS}, we first derive a new stochastic entropy-energy estimate (independent of the Friedrichs mollifier) (cf. Lemma \ref{lem3}). However, this uniform bounded estimate alone is insufficient to guarantee the existence of a global strong solution. To prove the existence of a unique global pathwise strong solution for system \eqref{CNS}, obtaining higher-order uniform bounds seems necessary. However, due to the influence of noise, the desired uniform bounds cannot be directly obtained through high-order energy estimates. Fortunately, through a refined iterative process, we show that the $\|n^{\epsilon}\|_{L^{\infty}(0,T;H^1)\cap L^2(0,T;H^2)}^2$-term and $\|c^{\epsilon}\|_{L^{\infty}(0,T;H^2)\cap L^2(0,T;H^3)}^2$-term are controlled by exponential functions of $\|u^{\epsilon}\|_{L^{\infty}(0,T;H^1)\cap L^2(0,T;H^2)}^2$, which is sufficient to ensure that the probability laws of the approximating solutions exhibit tightness in the appropriate phase space (cf. Lemma \ref{lem4}).

In the third step, inspired by the work of Mikulevicius and Rozovskii \cite{mikulevicius2005global} and Brze\'{z}niak and Motyl et al. \cite{12brzezniak2017note,11brzezniak2013existence}, we select an appropriate phase space \eqref{work} based on the uniform estimates provided by Lemma \ref{lem4} and Corollary \ref{1cor}. Utilizing compactness and tightness criteria, we verify that the probability measures of the approximate solutions are tight in the phase space \eqref{work} (cf. Proposition \ref{3pro}). Next, by applying the Jakubowski-Skorokhod theorem, one can construct a new probability space $(\bar{\Omega}, \bar{\mathcal{F}}, \bar{\mathbb{P}})$ on which a sequence $(\bar{n}^k,\bar{c}^k,\bar{u}^k,\bar{W}^k)$ is defined. This sequence has the same law as $(n^{\epsilon_k}, c^{\epsilon_k}, u^{\epsilon_k}, W^{\epsilon_k})$ and, more importantly, converges almost surely to an element $(n_*, c_*, u_*, W_*)$ in the aforementioned phase space. Then, through a careful identification of the limit, one can verify that the process $(n_*, c_*, u_*, W_*)$ defined on $(\bar{\Omega}, \bar{\mathcal{F}}, \bar{\mathbb{P}})$ is indeed a global martingale strong solution to system \eqref{CNS}. Finally, by verifying the uniqueness of the pathwise strong solution, one can invoke the Yamada-Watanabe theorem to demonstrate the existence of a unique global pathwise strong solution for system \eqref{CNS}.

\end{remark}

\subsection{Structure of the paper}
In Section \ref{sec2}, we establish the existence and uniqueness of a global smooth solution to the regularized problem \eqref{Mod-1}. In Section \ref{sec3-1}, we derive uniform bounds of low regularity via stochastic entropy-energy estimates. Section \ref{sec3-2} is devoted to improving the space-time regularity of the solutions through further energy estimates. In Section \ref{sec3.3}, we introduce compactness criteria and prove the tightness of the laws induced by the approximate solutions in the designated phase space. Finally, the proof of Theorem \ref{th1} is given in Section \ref{sec3.4}.

\section{Solvability of the modified system}\label{sec2}
This section is devoted to constructing an approximate system for \eqref{CNS}. Let $\rho^\epsilon$ be a standard Friedrichs mollifier (cf. \cite{majda2002vorticity}). We then consider the following regularized system:
\begin{equation}\label{Mod-1}
\left\{
\begin{aligned}
&\mathrm{d}n^\epsilon +u^\epsilon\cdot \nabla n^\epsilon\mathrm{d}t  = \Delta n ^\epsilon\mathrm{d}t -  \nabla\cdot \left(n^\epsilon\nabla (c^\epsilon*\rho^\epsilon)\right)\mathrm{d}t+L(n^{\epsilon})\mathrm{d}t, \\
&\mathrm{d}c^\epsilon+ u^\epsilon\cdot \nabla c^\epsilon\mathrm{d}t =\Delta c^\epsilon\mathrm{d}t-c^\epsilon(n^\epsilon*\rho^\epsilon)\mathrm{d}t,\\
&\mathrm{d}u^\epsilon+   \textbf{P} (u^\epsilon\cdot \nabla) u^\epsilon\mathrm{d}t =-A u^\epsilon\mathrm{d}t-\textbf{P}[\mathbf{g}(|u^{\epsilon}|^2)u^{\epsilon}]\mathrm{d}t+  \textbf{P}[ (n^\epsilon\nabla \phi)*\rho^\epsilon]\mathrm{d}t+\textbf{P}G(u^\epsilon) \mathrm{d} W(t),\\
&\nabla\cdot u ^\epsilon =0 , \\
&n^\epsilon|_{t=0}=n_0^\epsilon,~c|_{t=0}=c_0^\epsilon,~ u|_{t=0}=u_0^\epsilon,
\end{aligned}
\right.
\end{equation}
 with smooth initial conditions
$$
n_0^\epsilon= n_0*\rho^\epsilon, \quad c_0^\epsilon= c_0*\rho^\epsilon,\quad u_0^\epsilon= u_0*\rho^\epsilon.
$$
Here, $A:=- \textbf{P}  \Delta$ is the Stokes operator, where $ \textbf{P} $ denotes the
Helmholtz projection from $L^2(\mathbb{R}^3;\mathbb{R}^3)$ into $\mathbb{H}^0$ \cite{temam2001navier}. It is well known that the operator $\textbf{P}$ commutes with the Bessel operators. Moreover, the following equality holds:
$$
(\textbf{P}f,g)_{\mathbb{H}^0}=(f,g)_{L^2},\quad f\in L^2(\mathbb{R}^3;\mathbb{R}^3),\quad g\in \mathbb{H}^0.
$$

We shall prove that the regularized system \eqref{Mod-1} has a unique sufficiently smooth solution. To achieve this, we adopt a method that combines frequency truncations (cf. \cite{majda2002vorticity}) with smooth cut-off functions. Let $J_k$ be the frequency truncation operators defined by
$$
\widehat{J_k f} (\xi):= \textbf{1}_{\mathcal {C}_k}(\xi) \widehat{f}(\xi),\quad\mathcal {C}_k:=\{\xi\in \mathbb{R}^3;~ k^{-1}\leq |\xi|\leq k\},\quad k\in \mathbb{N}.
$$
Obviously, the operator $\textbf{P}$ and the truncation operator $J_k$ satisfy the commutation relation. For any $R>0$, choose a smooth cut-off function $\theta_R : [0,\infty)\rightarrow [0,1]$ such that
$\theta_R (x)\equiv 1$ if $0\leq x \leq R$ and $\theta_R (x)\equiv 0$ if $x > 2R$. Then the further modified system is given by
\begin{equation}\label{mod-22}
\left\{
\begin{aligned}
&\mathrm{d}n^{k,\epsilon,R} +\theta_R(\|(u^{k,\epsilon,R},n^{k,\epsilon,R})\|_{W^{1,\infty}})J_k(J_ku^{k,\epsilon,R}\cdot \nabla J_kn^{k,\epsilon,R})\mathrm{d} t=\Delta J_k^2n ^{k,\epsilon,R}\mathrm{d} t \\
&\quad-\theta_R(\|(n^{k,\epsilon,R},c^{k,\epsilon,R})\|_{W^{1,\infty}})J_k(\nabla\cdot(J_kn^{k,\epsilon,R}\nabla (c^{k,\epsilon,R}*\rho^\epsilon)))\mathrm{d} t\\
&\quad+\theta_R(\|n^{k,\epsilon,R}\|_{W^{1,\infty}})[-aJ_k^2n ^{k,\epsilon,R}+(1+a)J_k(J_kn ^{k,\epsilon,R})^2-J_k(J_kn ^{k,\epsilon,R})^3]\mathrm{d} t, \\
&\mathrm{d}c^{k,\epsilon,R}+ \theta_R(\|(u^{k,\epsilon,R},n^{k,\epsilon,R})\|_{W^{1,\infty}})J_k(J_ku^{k,\epsilon,R}\cdot \nabla J_kc^{k,\epsilon,R})\mathrm{d}t=\Delta J_k^2c^{k,\epsilon,R}\mathrm{d}t\\
&\quad-\theta_R(\|(n^{k,\epsilon,R},c^{k,\epsilon,R})\|_{W^{1,\infty}})J_k(J_kc^{k,\epsilon,R}(n^{k,\epsilon,R}*\rho^{\epsilon}))\mathrm{d}t,\\
&\mathrm{d}u^{k,\epsilon,R}+\theta_R(\|u^{k,\epsilon,R}\|_{W^{1,\infty}}) J_k\textbf{P}[ (J_ku^{k,\epsilon,R}\cdot \nabla) J_ku^{k,\epsilon,R}]\mathrm{d}t=-A J_k^2u^{k,\epsilon,R}\mathrm{d}t\\
&\quad-\theta_R(\|u^{k,\epsilon,R}\|_{W^{1,\infty}})J_k\textbf{P}[\mathbf{g}(|J_ku^{k,\epsilon,R}|^2)
J_ku^{k,\epsilon,R}]\mathrm{d}t\\
&\quad+\theta_R(\|n^{k,\epsilon,R}\|_{W^{1,\infty}})J_k\textbf{P}[ (n^{k,\epsilon,R}\nabla \phi)*\rho^\epsilon]\mathrm{d}t+\theta_R(\|u^{k,\epsilon,R}\|_{W^{1,\infty}})\textbf{P}G(u^{k,\epsilon,R}) \mathrm{d} W(t),\\
&\nabla\cdot u ^{k,\epsilon,R} =0 , \\
&n^{k,\epsilon,R}|_{t=0}=n_0^\epsilon,~c^{k,\epsilon,R}|_{t=0}=c_0^\epsilon,~ u^{k,\epsilon,R}|_{t=0}=u_0^\epsilon.
\end{aligned}
\right.
\end{equation}
Notice that the modified  system \eqref{mod-22} can be formulated as an infinite-dimensional SDE in the Hilbert space $\textbf{H}^s(\mathbb{R}^3):= H^s(\mathbb{R}^3)\times H^s(\mathbb{R}^3) \times \mathbb{H}^s $:
\begin{equation}\label{Mod-22}
\left\{
\begin{aligned}
&\mathrm{d}\textbf{y}^{k,\epsilon,R}= \textbf{F} ^{k,\epsilon,R} (\textbf{y}^{k,\epsilon,R})\mathrm{d}t+\textbf{L}^{R}(\textbf{y}^{k,\epsilon,R})\mathrm{d}\mathbf{W}(t),\\
& \textbf{y}^{k,\epsilon,R}(0) =\textbf{y}^{k,\epsilon,R}_0= (n_0^\epsilon,c_0^\epsilon,u_0^\epsilon)^{\top},
\end{aligned}
\right.
\end{equation}
where $\textbf{y}^{k,\epsilon,R}:=(n^{k,\epsilon,R} , c^{k,\epsilon,R},
u^{k,\epsilon,R})^{\top}$, $\textbf{F}^{k,\epsilon,R}(\cdot):=(\textbf{F}_1^{k,\epsilon,R} (\cdot),\textbf{F}_2^{k,\epsilon,R} (\cdot),\textbf{F}_3^{k,\epsilon,R}(\cdot))^{\top}$, and the functionals $\textbf{F}_i^{k,\epsilon,R}(\cdot)$ are given by
\begin{equation*}
\begin{split}
 \textbf{F}_1^{k,\epsilon,R}(\textbf{y}^{k,\epsilon,R} ) &= \Delta J_k^2n ^{k,\epsilon,R}-\theta_R(\|(n^{k,\epsilon,R},c^{k,\epsilon,R})\|_{W^{1,\infty}})J_k(\nabla\cdot(J_kn^{k,\epsilon,R}\nabla (c^{k,\epsilon,R}*\rho^\epsilon)))\\
&\quad+\theta_R(\|n^{k,\epsilon,R}\|_{W^{1,\infty}})[-aJ_k^2n ^{k,\epsilon,R}+(1+a)J_k(J_kn ^{k,\epsilon,R})^2-J_k(J_kn ^{k,\epsilon,R})^3]\\
&\quad-\theta_R(\|(u^{k,\epsilon,R},n^{k,\epsilon,R})\|_{W^{1,\infty}})J_k(J_ku^{k,\epsilon,R}\cdot \nabla J_kn^{k,\epsilon,R}),\\
\textbf{F}_2^{k,\epsilon,R}(\textbf{y}^{k,\epsilon,R} )&=\Delta J_k^2c^{k,\epsilon,R}-\theta_R(\|(n^{k,\epsilon,R},c^{k,\epsilon,R})\|_{W^{1,\infty}})J_k(J_kc^{k,\epsilon,R}(n^{k,\epsilon,R}*\rho^{\epsilon}))\\
&\quad-\theta_R(\|(u^{k,\epsilon,R},n^{k,\epsilon,R})\|_{W^{1,\infty}})J_k(J_ku^{k,\epsilon,R}\cdot \nabla J_kc^{k,\epsilon,R}),\\
\textbf{F}_3^{k,\epsilon,R}(\textbf{y}^{k,\epsilon,R} )&=-A J_k^2u^{k,\epsilon,R}-\theta_R(\|u^{k,\epsilon,R}\|_{W^{1,\infty}})J_k\textbf{P}[\mathbf{g}(|J_ku^{k,\epsilon,R}|^2)J_ku^{k,\epsilon,R}]\\
&\quad+\theta_R(\|n^{k,\epsilon,R}\|_{W^{1,\infty}})J_k\textbf{P}[ (n^{k,\epsilon,R}\nabla \phi)*\rho^\epsilon]-\theta_R(\|u^{k,\epsilon,R}\|_{W^{1,\infty}}) J_k\textbf{P} [(J_ku^{k,\epsilon,R}\cdot \nabla) J_ku^{k,\epsilon,R}].
\end{split}
\end{equation*}
The stochastic term in \eqref{Mod-22} is represented  by
$$
\textbf{L}^{ R}(\textbf{y}^{k,\epsilon,R}) \mathrm{d}\mathbf{W}=\begin{pmatrix}
0 & 0 & 0\\
  0& 0 & 0\\
0 & 0&  \theta_R\left(\|u^{k,\epsilon,R}\|_{W^{1,\infty}}\right)    \textbf{P}  G(u^{k,\epsilon,R})
\end{pmatrix}\begin{pmatrix}
0 \\
  0 \\
\mathrm{d}W
\end{pmatrix}.
$$

\begin{lemma}\label{lem1} Suppose that $s\geq10$, $\epsilon>0$, $R>1$, $k\in \mathbb{N}$, and the Assumptions \ref{as1}-\ref{as3} hold. Then, for any $T>0$, the approximate system \eqref{Mod-22} has a unique solution $(n^{k,\epsilon,R},c^{k,\epsilon,R},u^{k,\epsilon,R})$ in $\mathcal {C}([0,T];\textbf{H}^s(\mathbb{R}^3))$, $\mathbb{P}$-a.s.
\end{lemma}

\begin{proof}[\emph{\textbf{Proof}}]
Since supp $\widehat{J_k f} \subset\{\xi\in \mathbb{R}^3;~|\xi |\leq k\}$ for $k\geq1$, it is not hard to derive from the Bernstein inequality \cite{bahouri2011fourier} and the Moser-type estimate \cite{miao2012littlewood} that $\textbf{F}^{k,\epsilon,R}(\cdot):\textbf{H}^s(\mathbb{R}^3)\mapsto \textbf{H}^s(\mathbb{R}^3) $ and $\textbf{L}^{R}(\cdot): \textbf{H}^s(\mathbb{R}^3)\mapsto \mathcal{L}_2(\mathbf{U};\textbf{H}^s(\mathbb{R}^3))$ are locally Lipschitz continuous functionals, that is for a given $M>0$
\begin{equation*}
\begin{split}
\|\textbf{F}^{k,\epsilon,R}(\textbf{y}_1^{k,\epsilon,R}) -\textbf{F}^{k,\epsilon,R}(\textbf{y}_2^{k,\epsilon,R}) \|_{\textbf{H}^s} &\lesssim_{a,\phi,M,k,\epsilon,R}  \|\textbf{y}_1^{k,\epsilon,R}-\textbf{y}_2^{k,\epsilon,R}\|_{\textbf{H}^s},\\
\|\textbf{L}^{R}(\textbf{y}_1^{k,\epsilon,R})-\textbf{L}^{R}(\textbf{y}_2^{k,\epsilon,R}) \|_{\mathcal{L}_2(\mathbf{U};\textbf{H}^s)} &\lesssim_{M,R}  \|\textbf{y}_1^{k,\epsilon,R}-\textbf{y}_2^{k,\epsilon,R}\|_{\textbf{H}^s},
\end{split}
\end{equation*}
where $\|\textbf{y}_1^{k,\epsilon,R}\|_{\textbf{H}^s}\leq M$ and $\|\textbf{y}_2^{k,\epsilon,R}\|_{\textbf{H}^s}\leq M$. Moreover, there holds
\begin{equation*}
\begin{split}
 \|\textbf{F}^{k,\epsilon,R}(\textbf{y}^{k,\epsilon,R}) \|_{\textbf{H}^s} \lesssim_k  1+\|\textbf{y}^{k,\epsilon,R}\|_{\textbf{H}^s} ,\quad\|\textbf{L}^{R}(\textbf{y}^{k,\epsilon,R}) \|_{\mathcal{L}_2(\mathbf{U};\textbf{H}^s)} \lesssim 1+\|\textbf{y}^{k,\epsilon,R}\|_{\mathbf{H}^s}.
\end{split}
\end{equation*}
Thus, by using the well-known theory for SDEs in Hilbert spaces (cf. \cite{kallianpur1995stochastic,prevot2007concise}) and the cancelation property (cf. \cite{majda2002vorticity}), the system \eqref{Mod-22} admits a unique global   solution $\textbf{y} ^{k,\epsilon,R}\in\mathcal {C}([0,T];\textbf{H}^s(\mathbb{R}^3))$, $\mathbb{P}$-a.s. The proof of Lemma \ref{lem1} is completed.
\end{proof}

Next, we shall establish the uniform $\textbf{H}^s$-estimates for $\textbf{y}^{k,\epsilon,R}$ (independent of $k$).

\begin{lemma}\label{lem2} Let $s=10$, $\epsilon>0$, $R>1$ and $T>0$. Let $(n^{k,\epsilon,R},c^{k,\epsilon,R},u^{k,\epsilon,R}) \in \mathcal {C}([0,T]; \textbf{H}^{10}(\mathbb{R}^3))$ be the unique solution to \eqref{Mod-22} guaranteed by Lemma \ref{lem1}. Then we have
$$
\sup_{k \in \mathbb{N}}\mathbb{E}\left\|(n^{k,\epsilon,R},c^{k,\epsilon,R},u^{k,\epsilon,R})\right\| _{\mathcal {C}([0,T]; \textbf{H}^{10} )}^p \lesssim_{R,p,\phi,\epsilon,a,n_0,c_0,u_0,T} 1,\quad\textrm{for all}~p\geq2.
$$
Moreover, for any $p>2$ and  $\alpha \in (0,\frac{1}{2}-\frac{1}{p})$, there holds
\begin{equation*}
\begin{split}
\sup_{k \in \mathbb{N}}\left(\mathbb{E}\left\|(n^{k,\epsilon,R}, c^{k,\epsilon,R})\right\|^p_{   \emph{\textrm{Lip}}([0,T]; H^{8}\times H^8 ) }+ \mathbb{E}\left\|u^{k,\epsilon,R}\right\|^p_{ \mathcal {C}^{ \alpha}([0,T];  H^{8} )  }  \right) \lesssim _{R,p,\phi,\epsilon,a,n_0,c_0,u_0,T}1.
\end{split}
\end{equation*}
\end{lemma}

\begin{proof}[\emph{\textbf{Proof}}] To simplify the notations, we use the notation $(n,c,u)$ instead of  $(n^{k,\epsilon,R},c^{k,\epsilon,R},u^{k,\epsilon,R})$ in the proof. At first we shall establish the $\|(n,c,u)\|_{\textbf{H}^{10}}^2$-estimate. Applying the Bessel operator $\Lambda^{10}$ to both sides of the first equation of \eqref{mod-22}, and then taking the scalar product with  $\Lambda^{10} n$ over $\mathbb{R}^3$, we have
\begin{equation}\label{lem2-1}
\begin{split}
&\frac{1}{2}\frac{\mathrm{d}}{\mathrm{d}t}\|n\|_{H^{10}}^2+\|\nabla J_kn\|_{H^{10}}^2+a\theta_R(\|n\|_{W^{1,\infty}})\|J_kn\|_{H^{10}}^2\\
&=\theta_R(\|(u,n)\|_{W^{1,\infty}})(\Lambda^{10}(J_ku J_kn),\nabla\Lambda^{10}J_kn)_{L^2}\\
&\quad+\theta_R(\|(n,c)\|_{W^{1,\infty}})(\Lambda^{10}(J_kn\nabla (c*\rho^\epsilon)),\nabla\Lambda^{10}J_kn)_{L^2}\\
&\quad-\theta_R(\|n\|_{W^{1,\infty}})(\Lambda^{10}((J_kn)^3),\Lambda^{10}J_kn)_{L^2}\\
&\quad+(a+1)\theta_R(\|n\|_{W^{1,\infty}})(\Lambda^{10}((J_kn)^2),\Lambda^{10}J_kn)_{L^2}\\
&:=\sum_{j=1}^4A_j.
\end{split}
\end{equation}
According to the Moser estimate, we infer that
\begin{equation}\label{a2-1}
\begin{split}
A_1&\leq C\theta_R(\|(u,n)\|_{W^{1,\infty}})(\|J_kn\|_{L^{\infty}}\|J_ku\|_{H^{10}}+\|J_ku\|_{L^{\infty}}\|J_kn\|_{H^{10}})\|\nabla J_kn\|_{H^{10}}\\
&\leq \frac{1}{4}\|\nabla J_kn\|_{H^{10}}^2+C_{R}(\|J_ku\|_{H^{10}}^2+\|J_kn\|_{H^{10}}^2).
\end{split}
\end{equation}
Similarly, we have
\begin{equation}\label{a2-2}
\begin{split}
A_2&\leq C\theta_R(\|(n,c)\|_{W^{1,\infty}})(\|J_kn\|_{L^{\infty}}\|\nabla (c*\rho^\epsilon)\|_{H^{10}}+\|\nabla (c*\rho^\epsilon)\|_{L^{\infty}}\|J_kn\|_{H^{10}})\|\nabla J_kn\|_{H^{10}}\\
&\leq C\theta_R(\|(n,c)\|_{W^{1,\infty}})(\frac{1}{\epsilon}\|J_kn\|_{L^{\infty}}\|c\|_{H^{10}}+\|\nabla c\|_{L^{\infty}}\|J_kn\|_{H^{10}})\|\nabla J_kn\|_{H^{10}}\\
&\leq \frac{1}{4}\|\nabla J_kn\|_{H^{10}}^2+C_{R,\epsilon}(\|c\|_{H^{10}}^2+\|J_kn\|_{H^{10}}^2),
\end{split}
\end{equation}
and
\begin{equation}\label{a2-3}
\begin{split}
&|A_3|+|A_4|\leq C_a\theta_R(\|n\|_{W^{1,\infty}})(1+\|J_kn\|_{L^{\infty}}^2)\|J_kn\|_{H^{10}}\|J_kn\|_{H^{10}}\leq C_{R,a}\|J_kn\|_{H^{10}}^2.
\end{split}
\end{equation}
Thus, by plugging \eqref{a2-1}, \eqref{a2-2} and \eqref{a2-3} into \eqref{lem2-1}, we infer that
\begin{equation}\label{lem2-2}
\begin{split}
&\frac{\mathrm{d}}{\mathrm{d}t}\|n\|_{H^{10}}^2+\|\nabla J_kn\|_{H^{10}}^2\lesssim_{R,\epsilon,a}\|n\|_{H^{10}}^2+\|c\|_{H^{10}}^2+\|u\|_{H^{10}}^2.
\end{split}
\end{equation}
In a similar manner to \eqref{lem2-2}, one can derive that
\begin{equation}\label{lem2-3}
\begin{split}
&\frac{\mathrm{d}}{\mathrm{d}t}\|c\|_{H^{10}}^2+\|\nabla J_kc\|_{H^{10}}^2\lesssim_{R}\|n\|_{H^{10}}^2+\|c\|_{H^{10}}^2+\|u\|_{H^{10}}^2.
\end{split}
\end{equation}
Define $\Phi(u):= \|\Lambda^{10} u\|_{L^2}^2$. By performing direct calculations, we have
\begin{equation*}
\begin{split}
&\Phi'(u)[f]=2(\Lambda^{10} u,\Lambda^{10} f)_{L^2}, \quad \Phi''(u)[f,g]=2(\Lambda^{10} g,\Lambda^{10} f)_{L^2},\quad \forall u,~f,~g\in \mathbb{H}^{10}.
\end{split}
\end{equation*}
By the above identities, it follows that
\begin{equation}\label{ff0}
\begin{split}
\|\Phi'(u_1)-\Phi'(u_2)\|_{\mathcal{L}(\mathbb{H}^{10};\mathbb{R})}&= \sup_{\|f\|_{\mathbb{H}^{10}}=1}   |\Phi'(u_1)[f]-\Phi'(u_2)[f]|\\
&=2|(\Lambda^{10} (u_1-u_2),\Lambda^{10} f)_{L^2}|\leq 2\|u_1-u_2\|_{\mathbb{H}^{10}},
\end{split}
\end{equation}
which implies the uniformly continuity of $\Phi' $ on bounded subsets of $\mathbb{H}^{10}$. Similar arguments also lead to the uniform continuity of $\Phi''$. Therefore, the conditions in \cite[Theorem 4.32]{da2014stochastic} hold. By applying the It\^{o} formula to $\Phi(u)$, we obtain
\begin{equation}\label{lem2-4}
\begin{split}
& \| \Lambda^{10} u(t)\|_{L^2}^2 +2\int_0^t\|\nabla\Lambda^{10} J_ku(r)\|_{L^2 }^2 \mathrm{d}r\\
&=\| \Lambda^{10} u_0\|_{L^2}^2+ \int_0^t(B_1(r) +B_2(r) +B_3(r)+B_4(r) )  \mathrm{d}r+\sum_{j\geq 1}\int_0^tB_5^j(r)  \mathrm{d}  W_j(r),
\end{split}
\end{equation}
where
\begin{equation*}
\begin{split}
&B_1 =  \theta_R(\|u\|_{W^{1,\infty}}) ^2 \| \Lambda^{10} \textbf{P}  G(u)\|_{\mathcal{L}_2(\mathbf{U};L^2)}^2, \\
&B_2 =-2\theta_R(\|u\|_{W^{1,\infty}}) (\Lambda^{10}(\mathbf{g}(|J_ku|^2)J_ku),\Lambda^{10} J_k u)_{L^2},\\
&B_3 =2\theta_R(\|n\|_{W^{1,\infty}})(\Lambda^{10}((J_kn\nabla \phi)*\rho^{\epsilon}),\Lambda^{10} J_k u)_{L^2},\\
&B_4 =- 2  \theta_R  (\|u\|_{W^{1,\infty}}) ( \Lambda^{10} (J_k u\cdot \nabla J_k u ),\Lambda^{10} J_k u)_{L^2},\\
& B_5 ^j=2   \theta_R(\|u\|_{W^{1,\infty}})(\Lambda^{10} G _j(u),\Lambda^{10} u)_{L^2},\quad j\geq 1.
\end{split}
\end{equation*}
By using the Assumption \ref{as3}, we have
\begin{equation}\label{2.1}
\begin{split}
& |B_1 | \lesssim  1+ \|u\|_{H^{10}}^2 .
\end{split}
\end{equation}
By applying the Paralinearization theorem (cf. \cite[Theorem 2.87]{bahouri2011fourier}) in $H^{10}(\mathbb{R}^3)$ with respect to the smooth function $\mathbf{g}(\cdot)$ and using the Moser estimate as well as the properties \eqref{con1} and \eqref{con2}, we have
\begin{equation}\label{2.2}
\begin{split}
 |B_2| &\lesssim \theta_R(\|u\|_{W^{1,\infty}})\|\mathbf{g}(|J_ku|^2)J_ku\|_{H^{10}}\| J_k u\|_{H^{10}}\\
&\lesssim \theta_R(\|u\|_{W^{1,\infty}})(\|\mathbf{g}(|J_ku|^2)\|_{L^{\infty}}\|J_ku\|_{H^{10}}+\|\mathbf{g}(|J_ku|^2)\|_{H^{10}}\|J_ku\|_{L^{\infty}})\|J_k u\|_{H^{10}}\\
&\lesssim \theta_R(\|u\|_{W^{1,\infty}})(\|J_ku\|_{L^{\infty}}^2\|J_ku\|_{H^{10}}+C_{\|u\|_{L^{\infty}}^2}\||J_ku|^2\|_{H^{10}}\|J_ku\|_{L^{\infty}})\|J_k u\|_{H^{10}}\\
&\lesssim \theta_R(\|u\|_{W^{1,\infty}})C_{\|u\|_{L^{\infty}}^2}\|J_ku\|_{L^{\infty}}^2\|J_k u\|_{H^{10}}^2\\
&\lesssim_R\|J_k u\|_{H^{10}}^2.
\end{split}
\end{equation}
Similarly, by using the Moser estimate, we gain
\begin{equation}\label{2.3}
\begin{split}
& |B_3| \lesssim \theta_R(\|n\|_{W^{1,\infty}})\|(J_kn\nabla \phi)*\rho^{\epsilon}\|_{H^{10}}\| J_k u\|_{H^{10}}\lesssim_{\epsilon,\phi}\|n\|_{H^{10}}^2+\|u\|_{H^{10}}^2.
\end{split}
\end{equation}
By integrating by parts, we have
\begin{equation}\label{2.4}
\begin{split}
B_4&=2  \theta_R  (\|u\|_{W^{1,\infty}}) ( \Lambda^{10} (J_k u\otimes J_k u),\nabla \Lambda^{10} J_k u)_{L^2}\\
&\leq C\theta_R  (\|u\|_{W^{1,\infty}})\|J_k u\otimes J_k u\|_{H^{10}}\|\nabla \Lambda^{10} J_k u\|_{L^2}\\
&\leq\frac{1}{2}\|\nabla \Lambda^{10} J_k u\|_{L^2}^2+C_{R}\|u\|_{H^{10}}^2.
\end{split}
\end{equation}
For the stochastic term involving $B_5 ^j$, we get by using the Burkholder-Davis-Gundy (BDG) inequality (cf. \cite[Theorem 4.36]{da2014stochastic}) that
\begin{equation}\label{2.5}
\begin{split}
&\mathbb{E}\sup_{r\in [0,t]}\left|\sum_{j\geq 1}\int_0^rB_5^j(s)  \mathrm{d}  W_j(s)\right|\\
&\leq C\mathbb{E}\left(\int_0^ t\theta_R\left(\|u(r)\|_{W^{1,\infty}}\right)^2(\sum_{j\geq 1}(\Lambda^{10} G _j(u(r)),\Lambda^{10} u(r))_{L^2}^2) \mathrm{d}r\right)^{1/2} \\
&\leq C\mathbb{E}\left[\sup_{r\in [0,t]}\|\Lambda^{10} u(r)\|_{L^2}\left( \int_0^t \theta_R \left(\|u(r)\|_{W^{1,\infty}}\right)^2 (\sum_{j\geq 1}\|\Lambda^{10} G_j(u(r))\|_{L^2}^2) \mathrm{d}r\right)^{1/2}\right] \\
&\leq \frac{1}{2} \mathbb{E}\sup_{r\in [0,t]}\|\Lambda^{10} u(r)\|_{L^2}^2+ C  \mathbb{E} \int_0^ t \left(1+\|u(r)\|_{H^{10}}^2\right) \mathrm{d}r.
\end{split}
\end{equation}
Therefore, by taking the supremum over  $[0,t]$ in \eqref{lem2-4}, we infer from \eqref{2.1}-\eqref{2.5} that
\begin{equation*}
\begin{split}
 \mathbb{E}\sup_{r\in [0,t]}\| u(r)\|_{H^{10}}^2+ \mathbb{E}\int_0^t\|\nabla  J_ku(r)\|_{H^{10}}^2 \mathrm{d}r \lesssim_{R,\phi,\epsilon} \| u_0\|_{H^{10}}^2+  \mathbb{E} \int _0^t(1 + \|n(r)\|_{H^{10}}^2+ \|u(r)\|_{H^{10}}^2) \mathrm{d}r,
\end{split}
\end{equation*}
which together with \eqref{lem2-2} and \eqref{lem2-3} implies that
\begin{equation*}
\begin{split}
\mathbb{E}\sup_{r\in [0,t]}\|(n,c,u)(r)\|^2_{\textbf{H}^{10}} \lesssim_{R,\phi,a,\epsilon} \|(n_0,c_0,u_0)\|^2_{\textbf{H}^{10}}+t+  \mathbb{E}\int_0^t\|(n,c,u)(r)\|^2_{\textbf{H}^{10}}\mathrm{d}r.
\end{split}
\end{equation*}
By using the Gronwall lemma, we get
$$
\mathbb{E}\sup_{r\in [0,T]}\|(n,c,u)(r)\|^2_{\textbf{H}^{10}} \lesssim_{R,\phi,a,\epsilon,T,n_0,c_0,u_0}1,\quad\forall T > 0.
$$

Next, we shall establish the $\|(n,c,u)\|_{\textbf{H}^{10}}^p$-estimate ($p\geq4$). By applying the chain rule to $ \| n\|^p_{H^{10}}= (\| n\|^2_{H^{10}})^{p/2}$, we have
\begin{equation}\label{22.1}
\begin{split}
&\| n(t)\|^p_{H^{10}}+\int_0^t\big(p\| n(r)\|_{H^{10}}^{p-2}\|\nabla J_kn(r)\|^2_{H^{10}}+ap\theta_R(\|n(r)\|_{W^{1,\infty}})\| n(r)\|_{H^{10}}^{p-2}\|J_kn(r)\|^2_{L^2}\big) \mathrm{d}r\\
&= \| n_0\|^p_{H^{10}} +p\int_0^t\|n(r)\|_{H^{10}}^{p-2} (A_1(r)+A_2(r)+A_3(r)+A_4(r)) \mathrm{d}r,
\end{split}
\end{equation}
where the terms $A_1$-$A_4$ are provided in \eqref{lem2-1}.  By using the inequalities \eqref{a2-1}, \eqref{a2-2} and \eqref{a2-3}, we infer from \eqref{22.1} that
\begin{equation}\label{22.2}
\begin{split}
& \| n(t)\|^p_{H^{10}}+\frac{p}{2}\int_0^t\| n(r)\|_{H^{10}}^{p-2}\|\nabla J_kn(r)\|^2_{H^{10}} \mathrm{d}r\\
&\lesssim_{p,R,\epsilon,a}  \| n_0\|^p_{H^{10}} +  \int_0^t(\| n(r)\|_{H^{10}}^{p} +\| c(r)\|_{H^{10}}^{p}+\| u(r)\|_{H^{10}}^{p})\mathrm{d}r.
\end{split}
\end{equation}
Similarly, it is not hard to derive that
\begin{equation}\label{22.3}
\begin{split}
& \| c(t)\|^p_{H^{10}}+\frac{p}{2}\int_0^t\| c(r)\|_{H^{10}}^{p-2}\|\nabla J_k c(r)\|^2_{H^{10}} \mathrm{d}r\\
&\lesssim_{p,R} \| c_0\|^p_{H^{10}}+\int_0^t(\| n(r)\|_{H^{10}}^{p} +\| c(r)\|_{H^{10}}^{p}+\| u(r)\|_{H^{10}}^{p}) \mathrm{d}r.
\end{split}
\end{equation}
Consider the functional $\Phi_1(u):=\|\Lambda^{10} u\|_{L^2}^{p}=(\|\Lambda^{10} u\|_{L^2}^2)^{\frac{p}{2}}$. For any $u,~f,~g\in \mathbb{H}^{10}$, we have
\begin{equation*}
\begin{split}
\Phi'_1(u)[f]&=p\|\Lambda^{10}u\|_{L^2}^{p-2}(\Lambda^{10} u,\Lambda^{10} f)_{L^2}, \\
\Phi''_1(u)[f,g]&=p(p-2)\|\Lambda^{10}u\|_{L^2}^{p-4}(\Lambda^{10} u,\Lambda^{10} f)_{L^2}(\Lambda^{10} u,\Lambda^{10} g)_{L^2}\\
&\quad+p\|\Lambda^{10}u\|_{L^2}^{p-2}(\Lambda^{10} g,\Lambda^{10} f)_{L^2}^2 .
\end{split}
\end{equation*}
Let $B(M)\subseteq \mathbb{H}^{10}$ be a bounded ball centered at $0$ with radius $M>0$, it follows that
\begin{equation}\label{fff1}
\begin{split}
\|\Phi'_1(u_1)-\Phi'_1(u_2)\|_{\mathcal{L}(\mathbb{H}^{10};\mathbb{R})}&=\sup_{\|f\|_{H^{10}}=1}|\Phi'_1(u_1)[f]-\Phi'_1(u_2)[f]|\\
&\lesssim_p\big|\|u_1\|_{H^{10}}^{p-2}-\|u_2\|_{H^{10}}^{p-2}\big|\|u_1\|_{H^{10}}
+\|u_2\|_{H^{10}}^{p-2}\|u_1-u_2\|_{H^{10}}\\
&\leq C_{p,M}\|u_1-u_2\|_{\mathbb{H}^{10}},\quad \forall u_1,~u_2 \in B(M),
\end{split}
\end{equation}
and similarly
\begin{equation}\label{ffff1}
\begin{split}
&\|\Phi''_1(u_1)-\Phi''_1(u_2)\|_{\mathcal{L}(\mathbb{H}^{10}\times\mathbb{H}^{10};\mathbb{R} )} \leq C_{p,M}\|u_1-u_2\|_{\mathbb{H}^{10}},\quad \forall u_1,~u_2 \in B(M),
\end{split}
\end{equation}
which imply the uniform continuity of mappings $\Phi'_1$ and $\Phi''_1$ on bounded subsets of $\mathbb{H}^{10}$. Due to \eqref{fff1} and \eqref{ffff1}, one can apply  It\^{o}'s formula to $\|\Lambda^{10} u(t)\|^p_{L^2}$ to obtain
\begin{equation}\label{22.4}
\begin{split}
&\|\Lambda^{10} u(t)\|^p_{L^2}+ p\int_0^t\|\Lambda^{10} u(r)\| _{L^2} ^{p-2} \|\nabla \Lambda^{10} J_ku(r)\|_{L^2 }^2\mathrm{d}r\\
&= \|\Lambda^{10} u_0\|^p_{L^2} + p\int_0^t\|\Lambda^{10} u(r)\| _{L^2} ^{p-2} (\sum_{i=1}^4B_i(r)) \mathrm{d}r\\
&\quad+\frac{p(p-2)}{2}\sum_{j \geq 1} \int_0^t\theta_R\left(\|u(r)\|_{W^{1,\infty}}\right)\|\Lambda^{10} u(r)\|_{L^2} ^{p-4} \left(\Lambda^{10} u(r),   \Lambda^{10} G(u(r))e_j\right)_{L^2}^2 \mathrm{d}r\\
&\quad+p\sum_{j \geq 1}\int_0^t \theta_R\left(\|u(r)\|_{W^{1,\infty}}\right)\|\Lambda^{10} u(r)\| _{L^2} ^{p-2}  \left(\Lambda^{10} u(r),   \Lambda^{10}   G(u(r))e_j\right)_{L^2} \mathrm{d}  W_j(r) \\
&:= \|\Lambda^{10} u_0\|^p_{L^2}+ D_1(t)+D_2(t)+D_3(t),
\end{split}
\end{equation}
where $B_i $, $i=1,~2,~3,~4$ are defined in \eqref{lem2-4}. According to \eqref{2.1}-\eqref{2.4}, we have
\begin{equation}\label{22.5}
\begin{split}
|D_1(t)|&\leq\frac{p}{2}\int_0^t\|\Lambda^{10} u(r)\| _{L^2} ^{p-2} \|\nabla \Lambda^{10} J_ku(r)\|_{L^2 }^2\mathrm{d}r\\
&\quad+C_{R,\epsilon,\phi}\int_0^t\|\Lambda^{10} u(r)\| _{L^2} ^{p-2}(1+\|n(r)\|_{H^{10}}^2+\|u(r)\|_{H^{10}}^2)\mathrm{d}r\\
&\leq \frac{p}{2}\int_0^t\|\Lambda^{10} u(r)\| _{L^2} ^{p-2} \|\nabla \Lambda^{10} J_ku(r)\|_{L^2 }^2\mathrm{d}r+C_{R,p,\epsilon,\phi}\int_0^t(1+\|n(r)\|_{H^{10}}^p+\|u(r)\|_{H^{10}}^p)\mathrm{d}r.
\end{split}
\end{equation}
For $D_2$, by using Young's inequality, it follows that
\begin{equation}\label{22.6}
\begin{split}
|D_2(t)|&\lesssim_p \int_0^t\theta_R\left(\|u(r)\|_{W^{1,\infty}}\right)\|\Lambda^{10} u(r) \|_{L^2} ^{p-2}  (\sum_{j \geq 1}\| \Lambda^{10} G(u(r))e_j \|_{L^2} ^2) \mathrm{d}r\\
&\lesssim_p \int_0^t \|\Lambda^{10} u(r)\|_{L^2} ^{p-2}  (1+ \|\Lambda^{10} u(r)\|_{L^2}^2) \mathrm{d}r \\
&\lesssim_p \int_0^t (1+ \| u(r)\|_{H^{10}}^p) \mathrm{d}r.
\end{split}
\end{equation}
For $D_3$, by using the BDG inequality, we obtain
\begin{equation}\label{22.7}
\begin{split}
\mathbb{E}\sup_{r \in [0,t]}|D_3(r)|
&\lesssim_{p}\mathbb{E}\left(\sup_{r\in [0,t]}\|\Lambda^{10} u(r)\|^p_{L^2}\int_0^t\|\Lambda^{10} u(r)\| _{L^2} ^{ p-2} (1+\|\Lambda^{10} u\|^2_{L^2})\mathrm{d}r\right)^{1/2}\\
&\leq  \frac{1}{2}\mathbb{E} \sup_{r\in [0,t]}\|\Lambda^{10} u(r)\|^p_{L^2}  + C_{p} \mathbb{E}\int_0^t (1+ \|\Lambda^{10} u(r)\|^p_{L^2}) \mathrm{d}r.
\end{split}
\end{equation}
By taking the supremum over $[0,t]$ on both sides of \eqref{22.4}, we get from \eqref{22.5}, \eqref{22.6} and \eqref{22.7} that
\begin{equation*}
\begin{split}
 &\mathbb{E}\sup_{r\in [0,t]}\|\Lambda^{10} u(r)\|^p_{L^2}+   \mathbb{E}\int_0^t\|\Lambda^{10} u(r)\| _{L^2} ^{p-2} \|\nabla\Lambda^{10} J_ku(r)\|_{L^2 }^2\mathrm{d}r \\
 &\lesssim_{R,p,\epsilon,\phi} \| u_0\|^p_{H^{10}} + \mathbb{E}\int_0^t(1+\|n(r)\|_{H^{10}}^p+ \|u(r)\|_{H^{10}}^p)\mathrm{d}r,
\end{split}
\end{equation*}
which together with \eqref{22.2} and \eqref{22.3} yields that
\begin{equation*}
\begin{split}
  \mathbb{E}\sup_{r\in [0,t]}\|(n,c,u)(r)\|^p_{\textbf{H}^{10}}\lesssim_{R,p,\phi,a,\epsilon}\|(n_0,c_0,u_0)\|^p_{\textbf{H}^{10}}+t+ \mathbb{E}\int_0^t\|(n,c,u)(r)\|^p_{\textbf{H}^{10}}\mathrm{d}r.
\end{split}
\end{equation*}
An application of Gronwall lemma to the last inequality leads to
\begin{equation}\label{22.8}
\begin{split}
 \mathbb{E}\sup_{t \in [0,T]}\|(n,c,u)(t)\|^p_{\textbf{H}^{10}} \lesssim_{R,\phi,a,\epsilon,T,n_0,c_0,u_0} 1,\quad\forall T>0,
\end{split}
\end{equation}
which implies that $(n,c,u)\in L^p(\Omega; \mathcal {C}([0,T];\textbf{H}^{10}(\mathbb{R}^3) ))$, for any $T>0$ and $p>2$.

Finally, we shall demonstrate the H\"{o}lder regularity of solutions $(n,c,u)$. According to the estimate \eqref{22.8} and the equations associated to the unknowns $n$ and $c$, it is standard to verify that
\begin{equation*}
\begin{split}
(n,c)\in L^p\left(\Omega; \textrm{Lip}([0,T]; H^{8}(\mathbb{R}^3)\times H^{8}(\mathbb{R}^3) )\right).
\end{split}
\end{equation*}
In terms of the fluid equations in \eqref{mod-22}, we have
\begin{equation}\label{22.10}
\begin{split}
&\|u(t)-u(\tau)\|_{H^{8 }}\\
&\leq\left\|\int_\tau^t \Delta J_k^2u(r)\mathrm{d}r \right\|_{H^{8 }}+ \left\|\int_\tau^t\theta_R (\|u(r) \|_{W^{1,\infty}} ) \textbf{P}J_k(\mathbf{g}(|J_ku(r)|^2)J_ku(r))\mathrm{d}r\right\|_{H^{8}} \\
&\quad+ \left\|\int_\tau^t\theta_R (\|n (r)\|_{W^{1,\infty}} ) \textbf{P} J_k ((n(r)\nabla \phi)*\rho^{\epsilon}  )\mathrm{d}r\right\|_{H^{8}}\\
&\quad+\left\|\int_\tau^t\theta_R (\|u(r) \|_{W^{1,\infty}} ) \textbf{P} J_k  (J_k u(r) \cdot \nabla J_k u(r)  ) \mathrm{d}  r\right\|_{H^{8}}\\
&\quad+\left\|\int_\tau^t\theta_R (\|u(r)\|_{W^{1,\infty}} )    \textbf{P}  G(u(r)) \mathrm{d}  W(r)\right\|_{H^{8}} \\
&:=E_1+E_2+E_3+E_4+E_5 .
\end{split}
\end{equation}
By \eqref{22.8}, it is clear that
\begin{equation}\label{22.11}
\begin{split}
&\mathbb{E}\left(E_1+E_2+E_3+E_4\right)^p\\
&\lesssim _{p,R,\phi,\epsilon}\mathbb{E}  \sup_{r\in [\tau,t]}\left(1+\|n(r)\|_{H^{10}}^p+\|u(r)\|_{H^{10}}^p \right) |t-\tau|^p\\
&\lesssim _{p,R,\phi,a,\epsilon,T,n_0,c_0,u_0} |t-\tau|^p.
\end{split}
\end{equation}
For $E_5$, by using the BDG inequality, we deduce that
\begin{equation*}
\begin{split}
\mathbb{E} \left\|\int_\tau^t\theta_R(\|u(r)\|_{W^{1,\infty}}) \textbf{P}  G(u (r)) \mathrm{d}  W(r)\right\|_{H^{8}}^p &\lesssim_{p} \mathbb{E}\left( \int_\tau^t\|G(u(r) )\|_{\mathcal{L}_2(\mathbf{U};H^8)}^2\mathrm{d}r\right)^{\frac{p}{2}}\\
&\lesssim _{p,R,\phi,a,\epsilon,T,n_0,c_0,u_0} |t-\tau|^\frac{p}{2},
\end{split}
\end{equation*}
which combined with \eqref{22.11} leads to
$$
\mathbb{E} \|u(t)-u(\tau)\|_{H^{8}} ^p \lesssim_{p,R,\phi,a,\epsilon,T,n_0,c_0,u_0}  |t-\tau|^{\frac{p}{2}}.
$$
Thus by means of the Kolmogorov continuity theorem (cf. \cite[Theorem 3.3]{da2014stochastic}), we infer that the solution $u$ has an undistinguishable  version in $  \mathcal {C}^{\alpha}([0,T];H^{8}(\mathbb{R}^3))$, for all  $  0\leq {\alpha}<\frac{1}{2}-\frac{1}{p}$, $\mathbb{P}$-a.s. The proof of Lemma \ref{lem2} is completed.
\end{proof}

Next, we shall prove that the following convergence result holds for any $h>0$:
\begin{equation}\label{22.13}
\begin{split}
\lim_{k\rightarrow\infty}\sup_{l\geq k} \mathbb{P} \left\{\sup_{t\in [0,T]} \|\textbf{y}^{k,\epsilon,R}(t)-\textbf{y}^{l,\epsilon,R}(t)\|_{\textbf{H}^{7} }> h \right\}=0.
\end{split}
\end{equation}
To this end, for any fixed $R>0$ and $\epsilon \in (0,1)$, let $\textbf{y}^{k,\epsilon,R}$ and $\textbf{y}^{l,\epsilon,R}$ be solutions to \eqref{Mod-22} associated with the frequency truncations $J_k$ and $J_l$, respectively. Define $\textbf{y}^{k,l,\epsilon,R}:=\textbf{y}^{k,\epsilon,R}-\textbf{y}^{l,\epsilon,R}$, it follows from \eqref{Mod-22} that
\begin{equation}\label{2..1}
\left\{
\begin{aligned}
& \mathrm{d}  \textbf{y}^{k,l,\epsilon,R}= [\textbf{F} ^{k,\epsilon,R} (\textbf{y}^{k,\epsilon,R}  )-\textbf{F} ^{l,\epsilon,R} (\textbf{y}^{l,\epsilon,R})]  \mathrm{d} t
+[\textbf{L}^R(\textbf{y}^{k,\epsilon,R}  )-\textbf{L}^R(\textbf{y}^{l,\epsilon,R}  )]\mathrm{d}\mathbf{W}(t),\\
& \textbf{y}^{k,l,\epsilon,R}(0) =0.
\end{aligned}
\right.
\end{equation}
For convenience,  in the following discussion we will simply write $\textbf{y}^{k,l}$, $\textbf{y}^{k}$, $\textbf{y}$, $\textbf{F} ^{k}(\cdot)$ and $\textbf{L}(\cdot)$ instead of $\textbf{y}^{k,l,\epsilon,R}$, $\textbf{y}^{k,\epsilon,R}$, $\textbf{y}^{\epsilon,R}$, $\textbf{F} ^{k,\epsilon,R}(\cdot)$ and $\textbf{L}^{R} (\cdot)$, respectively. Then we have
\begin{equation*}
\begin{split}
& \textbf{F} ^{k}_1 (\textbf{y}^{k}  )-\textbf{F} ^{l}_1 (\textbf{y}^{l}  )\\
&=\Delta J_k^2  n^{k,l}+(J_k^2-J_l^2)\Delta n^l+\theta_R(\|(u^k,n^k)\|_{W^{1,\infty}})[J_l(J_l u^l\cdot \nabla J_l n^l)- J_k(J_k u^k\cdot \nabla J_k n^k)]   \\
&\quad+[\theta_R(\|(u^l,n^l)\|_{W^{1,\infty}}) -\theta_R(\|(u^k,n^k)\|_{W^{1,\infty}}) ] J_l(J_l u^l\cdot \nabla J_l n^l)\\
&\quad+[\theta_R(\|(n^l,c^l)\|_{W^{1,\infty}})-\theta_R(\|(n^k,c^k)\|_{W^{1,\infty}})] \nabla\cdot J_k(J_kn^k(\nabla c^k*\rho^{\epsilon})) \\
&\quad+\theta_R(\|(n^l,c^l)\|_{W^{1,\infty}})[\nabla\cdot J_l(J_ln^l(\nabla c^l*\rho^{\epsilon}))- \nabla\cdot J_k(J_kn^k(\nabla c^k*\rho^{\epsilon}))]\\
&\quad+\theta_R(\|n^k\|_{W^{1,\infty}})[-a(J_k n^k-J_l n^l)+(1+a)(J_k (J_kn^k)^2-J_l(J_l n^l)^2)-(J_k (J_kn^k)^3-J_l(J_l n^l)^3)]\\
&\quad+[\theta_R(\|n^k\|_{W^{1,\infty}})-\theta_R(\|n^l\|_{W^{1,\infty}})][-aJ_l n^l+(1+a)J_l(J_l n^l)^2-J_l(J_l n^l)^3]\\
&:=\Delta J_k^2  n^{k,l}+ I_1+\cdots + I_7,
 \end{split}
\end{equation*}
\begin{equation*}
\begin{split}
&\textbf{F} ^{k}_2 (\textbf{y}^{k}  )-\textbf{F} ^{l}_2 (\textbf{y}^{l}  )\\
&=\Delta J_k^2  c^{k,l}+(J_k^2-J_l^2)\Delta c^l  +[\theta_R(\|(u^l,c^l)\|_{W^{1,\infty}})-\theta_R(\|(u^k,c^k)\|_{W^{1,\infty}}) ] J_l(J_l u^l\cdot \nabla J_k c^l) \\
&\quad+\theta_R(\|(u^k,c^k)\|_{W^{1,\infty}})[J_l(J_l u^l\cdot \nabla J_l c^l)- J_k(J_k u^k\cdot \nabla J_k c^k)]\\
&\quad+[\theta_R(\|(n^l,c^l)\|_{W^{1,\infty}})-\theta_R(\|(n^k,c^k)\|_{W^{1,\infty}})]J_k( J_kc^k(J_kn^k*\rho^{\epsilon})) \\
&\quad+\theta_R(\|(n^l,c^l)\|_{W^{1,\infty}})[J_l(J_lc^l(J_ln^l*\rho^{\epsilon}))-J_k(J_kc^k(J_kn^k*\rho^{\epsilon}))]\\
&:=\Delta J_k^2  c^{k,l}+ K_1+\cdots + K_5,
 \end{split}
\end{equation*}
and
\begin{equation*}
\begin{split}
&\textbf{F} ^{k}_3 (\textbf{y}^{k}  )-\textbf{F} ^{l}_3 (\textbf{y}^{l}  )\\
&=\Delta J_k^2  u^{k,l}+(J_k^2-J_l^2)\Delta u^l  +[\theta_R(\| u^l \|_{W^{1,\infty}}) -\theta_R(\| u^k \|_{W^{1,\infty}}) ] \textbf{P}J_l((J_l u^l\cdot \nabla) J_l u^l)\\
&\quad+ \theta_R(\| u^k \|_{W^{1,\infty}}) \textbf{P} [  J_l((J_l u^l\cdot \nabla) J_l u^l)- J_k((J_k u^k\cdot \nabla) J_k u^k)]      \\
&\quad+ \textbf{P} [J_k((J_kn^k\nabla \phi)*\rho^{\epsilon})- J_l((J_ln^l\nabla \phi)*\rho^{\epsilon})]\\
&\quad+[\theta_R(\| u^l \|_{W^{1,\infty}}) -\theta_R(\| u^k \|_{W^{1,\infty}}) ]\textbf{P}[J_l(\mathbf{g}(|J_l u^l|^2)J_l u^l)]\\
&\quad+ \theta_R(\| u^k \|_{W^{1,\infty}}) \textbf{P} [  J_l(\mathbf{g}(|J_l u^l|^2)J_l u^l)- J_k(\mathbf{g}(|J_k u^k|^2)J_k u^k)]\\
&:=\Delta J_k^2  u^{k,l}+ L_1+\cdots + L_6.
 \end{split}
\end{equation*}
Concerning the stochastic integral, we have
\begin{equation*}
\begin{split}
&\int_0^t(\textbf{y}^{k,l}(r),(\textbf{L}(\textbf{y}^{k}(r)  )-\textbf{L}(\textbf{y}^{l}(r)  ))\mathrm{d}  \mathbf{W}(r))_{\textbf{H} ^{7}}\\
&=\sum_{j\geq 1}\left(\int_0^t(\theta_R (\|u^{k} (r)\|_{W^{1,\infty}} )-\theta_R (\|u^{l} (r)\|_{W^{1,\infty}} )) (u^{k,l}(r),G_j(u^{k}(r) ))_{H^{7}}\mathrm{d}  W_j(r)\right.\\
&\left.\quad +\int_0^t\theta_R (\|u^{l}(r) \|_{W^{1,\infty}} )  (u^{k,l}(r),G_j(u^{k}(r) )-  G_j(u^{l}(r)  ))_{H^{7}}\mathrm{d}  W_j(r)\right)\\
&:= \sum_{j\geq 1}\left(\int_0^t R_j(r)\mathrm{d}  W_j(r)+\int_0^t S_j(r)\mathrm{d}  W_j(r)\right).
 \end{split}
\end{equation*}
Now let us estimate the terms $I_i,~K_i$ and $L_i$ in the above equations one by one.
\begin{lemma}\label{lem2.1}
For the integrals associated to $I_i,~K_i$ and $L_i$, we have
\begin{subequations}
\begin{align}
& |(n^{k,l},I_1)_{H^{7}}| \lesssim \|n^{k,l}\|_{H^{7}}^2+\max\{\frac{1}{k^2},\frac{1}{l^2}\} \|  n^l\|_{H^{10}}^2 ,\label{2.30a}\\
&|(n^{k,l},I_2)_{H^{7}}| \lesssim  \|n^{k,l}\|_{H^{7}} ^2+  \max\{\frac{1}{k^{2}} ,\frac{1}{l^{2}} \}\big(1+\| u^k\|_{H^{10}}^4+\| u^l\|_{H^{10}}^4+\| n^l\|_{H^{10}}^4+\| n^k\|_{H^{10}} ^4\big),\label{2.30b}\\
&|(n^{k,l},I_3)_{H^{7}}| \lesssim\|u^{k,l}\|_{H^{7}}^2+\|n^{k,l}\|_{H^{7}}^2+ \frac{1}{l^2} \left( 1+\| u^l\|_{H^{10}}^8+\| n^l\|_{H^{10}}^8+\| n^k\|_{H^{10}} ^8\right),\label{2.30c}\\
&|(n^{k,l},I_4)_{H^{7}}|\lesssim\|n^{k,l}\|_{H^{7}}^2+\|c^{k,l}\|_{H^{7}}^2+  \frac{1}{k^2} (1+\|n^k\|_{H^{10}} ^8+\|n^l\|_{H^{10}} ^8+\|c^k \|_{H^{10}}^8),\label{2.30d}\\
&|(n^{k,l},I_5)_{H^{7}}| \lesssim\|n^{k,l}\|_{H^{7}} ^2   +  \max\{\frac{1}{k^2},\frac{1}{l^2}\} (\|n^k \|_{H^{10}}^4+\|n^l \|_{H^{10}}^4+\|c^l \|_{H^{10}}^4+\|c^k\|_{H^{10}}^4),\label{2.30e}\\
&|(n^{k,l},I_6)_{H^{7}}|+|(n^{k,l},I_7)_{H^{7}}|\lesssim_a\|n^{k,l}\|_{H^{7}} ^2+\max\{\frac{1}{k^2},\frac{1}{l^2}\} (1+\|n^k \|_{H^{10}}^8+\|n^l \|_{H^{10}}^8),\label{2.30e2}\\
& |(c^{k,l},K_1)_{H^{7}}| \lesssim \|c^{k,l}\|_{H^{7}}^2+\max\{\frac{1}{k^2},\frac{1}{l^2}\} \|c^l\|_{H^{10}}^2 \label{2.31a},\\
&|(c^{k,l},K_2)_{H^{7}}| \lesssim \|(u^{k,l},c^{k,l})\|_{H^{7}}^2 + \frac{1}{ l^2}(\|u^l\|_{H^{10}}^4+\|c^l\|_{H^{10}}^4),\label{2.31b}\\
&|(c^{k,l},K_3)_{H^{7}}| \lesssim \|(u^{k,l},c^{k,l})\|_{H^{7}}^2 + \frac{1}{k^2}(1+\|c^k\|_{H^{10}}^8+\|c^l\|_{H^{10}}^8+\|u^l\|_{H^{10}}^8),\label{2.31c}\\
&|(c^{k,l},K_4)_{H^{7}}| \lesssim \|(n^{k,l},c^{k,l})\|_{H^{7}}^2+\frac{1}{k^2} (1+\|n^k\|_{H^{10}}^8 +\| c^{k} \|_{H^{10}}^8+\| c^{l} \|_{H^{10}}^8), \label{2.31d}\\
&|(c^{k,l},K_5)_{H^{7}}| \lesssim \| c^{k,l} \|_{H^{7}}^2 + \max\{\frac{1}{k^2},\frac{1}{l^2}\}(\| n^l \|_{H^{10}}^4+\| c^l \|_{H^{10}}^4+\| n^k \|_{H^{10}}^4+\| c^k\|_{H^{10}}^4),\label{2.31e}\\
& |(u^{k,l},L_1)_{H^{7}}| \lesssim \|u^{k,l}\|_{H^{7}}^2+\max\{\frac{1}{k^2},\frac{1}{l^2}\} \|  u^l\|_{H^{10}}^2 ,\label{2.32a}\\
&|(u^{k,l},L_2)_{H^{7}}| \lesssim  \|u^{k,l}\|_{H^{7}} ^2+  \frac{1}{l^{2}}\| u^l\|_{H^{10}}^4,\label{2.32b}\\
&|(u^{k,l},L_3)_{H^{7}}| \lesssim  \|u^{k,l}\|_{H^{7}}^2+ \max\{\frac{1}{k^2},\frac{1}{l^2}\}(\|u^{k}\|_{H^{10}}^4+\|u^{l}\|_{H^{10}}^4 ),\label{2.32c}\\
&|(u^{k,l},L_4)_{H^{7}}|\lesssim_{\phi,\epsilon}  \|u^{k,l}\|_{H^{7}}^2+ \max\{\frac{1}{k^2},\frac{1}{l^2}\} (\|n^k \|_{H^{10}}^2+\|n^l \|_{H^{10}}^2),\label{2.32d}\\
&|(u^{k,l},L_5)_{H^{7}}|\lesssim\|u^{k,l}\|_{H^{7}}^2+\frac{1}{l^2}(1+\|u^{k}\|_{H^{10}}^{32}+\|u^{l}\|_{H^{10}}^{32}),\label{2.32d1}\\
&|(u^{k,l},L_6)_{H^{7}}|\lesssim\|u^{k,l}\|_{H^{7}}^2+\frac{1}{l^2}(1+\|u^{k}\|_{H^{10}}^{32}+\|u^{l}\|_{H^{10}}^{32}).\label{2.32d2}
\end{align}
\end{subequations}
\end{lemma}

\begin{proof}[\emph{\textbf{Proof}}]
Since for any  integer $k>0$, $\widehat{J_k f}$ is supported on the torus away from $0$, we get from the Bernstein inequality (cf. \cite{miao2012littlewood}) that
\begin{equation*}
\begin{split}
\|J_l f^l- J_k f^k\|_{H^{s-d}}\lesssim\left\{
\begin{aligned}
& \max\{\frac{1}{k^d},\frac{1}{l^d}\} ( \|f^k\|_{H^{s}}+\|f^l\|_{H^{s}}),\\
&   \max\{\frac{1}{k^d},\frac{1}{l^d}\} \|f^l\|_{H^{s}}+ \|f^{k,l}\|_{H^{s-d}} ,
\end{aligned}
\right.\quad  \textrm{for all} ~s \in \mathbb{R},~d\in \mathbb{N}.
 \end{split}
\end{equation*}
For \eqref{2.30a}, we have
\begin{equation*}
\begin{split}
 |(n^{k,l},I_1)_{H^{7}}|&\leq \|n^{k,l}\|_{H^{7}}\|(J_k +J_l )(J_k -J_l )\Delta n^l\|_{H^{7}}\\
 & \lesssim \max\{\frac{1}{k},\frac{1}{l}\} \|n^{k,l}\|_{H^{7}}\|  n^l\|_{H^{10}} \lesssim \|n^{k,l}\|_{H^{7}}^2+\max\{\frac{1}{k^2},\frac{1}{l^2}\} \|  n^l\|_{H^{10}}^2.
 \end{split}
\end{equation*}
For \eqref{2.30b}, by the Moser estimate and the fact that $H^{7}(\mathbb{R}^3)$ is a Banach algebra, we have
\begin{equation*}
\begin{split}
 |(n^{k,l},I_2)_{H^{7}}|&\leq\Bigl[|(n^{k,l},(J_l-J_k)(J_l u^l\cdot \nabla J_l n^l))_{H^{7}}|\\
 &\quad+|(n^{k,l},J_k[ (J_l u^l-J_ku^k)\cdot \nabla J_l n^l +  J_k u^k\cdot \nabla( J_l n^l- J_k n^k) ] )_{H^7}|\Bigl]\\
 & \lesssim  \|n^{k,l}\|_{H^{7}} \left(\frac{1}{l^5}\| u^l\|_{H^{10}}\| n^l\|_{H^{10}}+ \max\{\frac{1}{k^3l^2},\frac{1}{l^5} \}(\|u^k\|_{H^{10}}+\|u^l\|_{H^{10}})\| n^l \|_{H^{10}}\right.\\
 &\left.\quad+ \max\{\frac{1}{k^3l^2},\frac{1}{k^5}\} \|  u^k\|_{H^{10}} (\|n^k\|_{H^{10}}+\|n^l\|_{H^{10}})\right)\\
 &\lesssim  \|n^{k,l}\|_{H^{7}} ^2+  \max\{\frac{1}{l^{2}}, \frac{1}{k^{2}} \}\left( 1+\| u^k\|_{H^{10}}^4+\| u^l\|_{H^{10}}^4+\| n^l\|_{H^{10}}^4+\| n^k\|_{H^{10}} ^4\right).
 \end{split}
\end{equation*}
For \eqref{2.30c}, by using the Mean value theorem, the embedding $H^{3}(\mathbb{R}^3)\hookrightarrow W^{1,\infty}(\mathbb{R}^3)$ and the boundedness of $\theta'(\cdot)$, we have
\begin{equation*}
\begin{split}
 |(n^{k,l},I_3)_{H^{7}}|&\lesssim \|(u^{k,l},n^{k,l})\|_{H^{7}}(\|n^{k}\|_{H^{7}}+\|n^{l}\|_{H^{7}})\|J_l u^l\|_{H^{7}}\|\nabla J_l n^l\|_{H^{7}}\\
 &\lesssim (\|u^{k,l}\|_{H^{7}}+\|n^{k,l}\|_{H^{7}})(\|n^{k}\|_{H^{10}}+\|n^{l}\|_{H^{10}})\frac{1}{l^3}\| u^l\|_{H^{10}}\frac{1}{l^2}\| n^l\|_{H^{10}}\\
 &\lesssim \|u^{k,l}\|_{H^{7}}^2+\|n^{k,l}\|_{H^{7}}^2+ \frac{1}{l^2} \left(1+ \| u^l\|_{H^{10}}^8+\| n^l\|_{H^{10}}^8+\| n^k\|_{H^{10}} ^8\right).
 \end{split}
\end{equation*}
Similarly, we have
\begin{equation*}
\begin{split}
|(n^{k,l},I_4)_{H^{7}}|
&\lesssim\|(n^{k,l},c^{k,l})\|_{H^{7}}(\| n^k \|_{H^{10}}+\| n^l\|_{H^{10}})\|J_kn^k\|_{H^{8}} \| c^k*\rho^{\epsilon} \|_{H^{9}} \\
&\lesssim\|n^{k,l}\|_{H^{7}}^2+\|c^{k,l}\|_{H^{7}}^2+  \frac{1}{k^2} (1+\|n^k\|_{H^{10}} ^8+\|n^l\|_{H^{10}} ^8+\|c^k \|_{H^{10}}^8),
  \end{split}
\end{equation*}
\begin{equation*}
\begin{split}
|(n^{k,l},I_5)_{H^{7}}|&\lesssim \|n^{k,l}\|_{H^{7}} \left(\|(J_l-J_k) J_ln^l\nabla (c_l*\rho^{\epsilon})\|_{H^{8}}+\|(J_l n^l-J_k n^k) \nabla  (c_k*\rho^{\epsilon})\|_{H^{8}}\right.\\
 &\left.\quad+\| J_k n^k [\nabla (c_l*\rho^{\epsilon})-\nabla (c_k*\rho^{\epsilon})]\|_{H^{8}}\right)\\
 &\lesssim\|n^{k,l}\|_{H^{7}} \left( \frac{1}{l^2} \| n^l \|_{H^{10}}\|c^l\|_{H^{10}}  +  \max\{\frac{1}{k^2},\frac{1}{l^2}\} ( \|n^k\|_{H^{10}}+\|n^l\|_{H^{10}})\|c^k\|_{H^{10}}\right.\\
 &\left.\quad+   \max\{\frac{1}{k^2},\frac{1}{l^2}\} \| n^k\|_{H^{10}}(\|c^l \|_{H^{10}}+\|c^k\|_{H^{10}})\right)\\
 &\lesssim\|n^{k,l}\|_{H^{7}} ^2   +  \max\{\frac{1}{k^2},\frac{1}{l^2}\} (\|n^k \|_{H^{10}}^4+\|n^l \|_{H^{10}}^4+\|c^l \|_{H^{10}}^4+\|c^k\|_{H^{10}}^4),
  \end{split}
\end{equation*}
and
\begin{equation*}
\begin{split}
&|(n^{k,l},I_6)_{H^{7}}|+|(n^{k,l},I_7)_{H^{7}}|\lesssim_a\|n^{k,l}\|_{H^{7}} ^2+\max\{\frac{1}{k^2},\frac{1}{l^2}\} (1+\|n^k \|_{H^{10}}^8+\|n^l \|_{H^{10}}^8).
  \end{split}
\end{equation*}
Moreover, the proofs of inequalities \eqref{2.31a} through \eqref{2.32d} are similar, and we omit the proof process. To obtain the estimate \eqref{2.32d1}, we need to handle the term $\|\mathbf{g}(|u|^2)\|_{H^7}^2$. Actually, given a multi-index $\alpha$ with $|\alpha|=s\geq2$, it is not difficult to see that $D^{\alpha}\mathbf{g}(|u|^2)$ is a sum of the form
\begin{equation}\label{GGG}
\begin{split}
&\mathbf{g}^{(l)}(|u|^2)\prod_{j=1}^lD^{\beta_j}(|u|^2),
\end{split}
\end{equation}
where $l\in\{1,\cdots,s\}$ and the $\beta_j$'s are multi-indices such that $\alpha=\beta_1+\cdots+\beta_l$ and $|\beta_j|\geq1$. Let $p_j=\frac{2s}{|\beta_j|}$, so that $\sum_{j=1}^l\frac{1}{p_j}=\frac{1}{2}$. By using the H\"{o}lder inequality and Gagliardo-Nirenberg inequality (cf. \cite{gagliardo1959ulteriori,gilbarg1977elliptic}):
\begin{equation}\label{GGN1}
\begin{split}
&\|D^{\beta_j}(|u|^2)\|_{L^{p_j}}\lesssim\||u|^2\|_{H^s}^\frac{|\beta_j|}{s}\||u|^2\|_{L^{\infty}}^{1-\frac{|\beta_j|}{s}}.
\end{split}
\end{equation}
It follows  from the Moser estimate that
\begin{equation}\label{GGN2}
\begin{split}
\left\|\prod_{j=1}^lD^{\beta_j}(|u|^2)\right\|_{L^{2}}\leq\prod_{j=1}^l\|D^{\beta_j}(|u|^2)\|_{L^{p_j}}\lesssim\|u\cdot u\|_{H^s}\|u\cdot u\|_{L^{\infty}}^{l-1}\lesssim_s\|u\|_{L^{\infty}}^{2l-1}\|u\|_{H^s}.
\end{split}
\end{equation}
Moreover, by using \eqref{con1} and \eqref{con2}, we get from \eqref{GGG} and \eqref{GGN2} that for any integer $s\geq2$
\begin{equation}\label{Gu1}
\begin{split}
&\|\mathbf{g}(|u|^2)\|_{H^s}^2\lesssim_s\|\mathbf{g}(|u|^2)\|_{L^2}^2+\sum_{|\alpha|=s}\|D^{\alpha}\mathbf{g}(|u|^2)\|_{L^2}^2\lesssim_{s}(1+\|u\|_{H^2}^{4s-2})\|u\|_{H^s}^2,\quad \forall u\in H^s(\mathbb{R}^3).
\end{split}
\end{equation}
Thus by using the Mean value theorem and inequality \eqref{Gu1}, we have
\begin{equation*}
\begin{split}
|(u^{k,l},L_5)_{H^{7}}|&\lesssim \|u^{k,l}\|_{H^{7}}(\|u^{k}\|_{H^{7}}+\|u^{l}\|_{H^{7}})\|\mathbf{g}(|J_l u^l|^2)\|_{H^{7}}\|J_l u^l\|_{H^{7}}\\
&\lesssim\|u^{k,l}\|_{H^{7}}^2+(\|u^{k}\|_{H^{7}}^2+\|u^{l}\|_{H^{7}}^2)(\|J_l u^l\|_{H^7}^2+\|J_l u^l\|_{H^7}^{28})\|J_l u^l\|_{H^{7}}^2\\
&\lesssim\|u^{k,l}\|_{H^{7}}^2+\frac{1}{l^2}(1+\|u^{k}\|_{H^{10}}^{32}+\|u^{l}\|_{H^{10}}^{32}).
 \end{split}
\end{equation*}
Similarly, by \eqref{Gu1}, we have
\begin{equation*}
\begin{split}
|(u^{k,l},L_6)_{H^{7}}|&\lesssim \|u^{k,l}\|_{H^{7}}(\|\mathbf{g}(|J_l u^l|^2)\|_{H^{7}}\|J_l u^l\|_{H^{7}}+\|\mathbf{g}(|J_k u^k|^2)\|_{H^{7}}\|J_k u^k\|_{H^{7}})\\
&\lesssim\|u^{k,l}\|_{H^{7}}^2+\frac{1}{l^2}(1+\|u^{k}\|_{H^{10}}^{32}+\|u^{l}\|_{H^{10}}^{32}).
 \end{split}
\end{equation*}
The proof of Lemma \ref{lem2.1} is thus completed.
\end{proof}

Now we are ready to prove that, for any fixed $\epsilon, R$, the sequence $\{\textbf{y}^{k,\epsilon,R}\}_{k\geq1}$ (up to a subsequence) converges strongly in $  \mathcal {C}([0,T];\mathbf{H}^{7}(\mathbb{R}^3))$, $\mathbb{P}$-a.s.

\begin{lemma}\label{lem2.3}
For any $R>1$, $0<\epsilon<1$ and $T>0$, there exists a subsequence of $\{\textbf{y} ^{k,\epsilon,R} \}_{k\geq1}$ (still denoted by itself) and a $\mathcal {F}_t$-progressively measurable stochastic  process $\textbf{y}^{\epsilon,R} \in L^p(\Omega; L^\infty(0,T;\textbf{H}^{10}(\mathbb{R}^3)))$, such that
\begin{equation}
\begin{split}
\mathbb{P} \left\{\textbf{y}^{k,\epsilon,R}\rightarrow \textbf{y}^{\epsilon,R}~ \textrm{strongly in}~ \mathcal {C}([0,T];\textbf{H}^{7}(\mathbb{R}^3)) ~ \textrm{as}~ k\rightarrow\infty \right\}=1.
 \end{split}
\end{equation}
\end{lemma}

\begin{proof}[\emph{\textbf{Proof}}] We shall use the same notations as we did in the proof of Lemma \ref{lem2.1}. Similar to the argument in \eqref{ff0}-\eqref{lem2-4}, we can apply It\^{o}'s formula to  $\|\Lambda^{7}\textbf{y}^{k,l}\|_{L^2}^2$ associated to the system \eqref{2..1}, and we have
\begin{equation}\label{2.31}
\begin{split}
& \|\textbf{y}^{k,l}(t)\|_{\textbf{H}^{7}}^2 + 2\int_0^t\|\nabla J_k \textbf{y}^{k,l}(r)\|_{\textbf{H}^{7}}^2\mathrm{d}r\\
&= 2 \sum_{1\leq i \leq 7}\int_0^t(n^{k,l}(r),I_i(r))_{H^{7}} \mathrm{d}r+2 \sum_{1\leq i \leq 5}\int_0^t(c^{k,l}(r),K_i(r))_{H^{7}} \mathrm{d}r\\
&\quad+2 \sum_{1\leq i \leq 6}\int_0^t(u^{k,l}(r),L_i(r))_{H^{7}} \mathrm{d}r+\sum_{j\geq 1}\int_0^t|R_j(r)+S_j(r)| ^2  \mathrm{d}r \\
&\quad+2\sum_{j\geq 1}\left(\int_0^t R_j(r)\mathrm{d}  W_j(r)+\int_0^t S_j(r)\mathrm{d}  W_j(r)\right)\\
&:=M_1(t)+\cdots+M_5(t).
 \end{split}
\end{equation}
According to the estimates provided by Lemma \ref{lem2.1}, we infer that
\begin{equation}\label{2.32}
\begin{split}
&|M_1(t)|+|M_2(t)|+|M_3(t)|\\
&\lesssim_{\phi,\epsilon,a}  \int_0^t\|\textbf{y}^{k,l}(r)\|_{\textbf{H}^{7}}^2 \mathrm{d}r+\max\{\frac{1}{k^2},\frac{1}{l^2}\}\int_0^t(1+\|\textbf{y}^{k }(r)\|_{\textbf{H}^{10}}^{32}+\|\textbf{y}^{l}(r)\|_{\textbf{H}^{10}}^{32}) \mathrm{d}r.
 \end{split}
\end{equation}
For $M_4$, it follows from the Mean value theorem and the Assumption  \ref{as3} that
\begin{equation*}
\begin{split}
\sum_{j\geq 1} |R_j|^2+|S_j|^2&\lesssim \sum_{j\geq 1}|\theta_R (\|u^{k} \|_{W^{1,\infty}} )-\theta_R (\|u^{l} \|_{W^{1,\infty}} )| ^2 (\|u^k\|_{H^{7}}^2+\|u^l\|_{H^{7}}^2)\|G_j(u^{k} )\|_{H^{7}}^2\\
&\quad+\sum_{j\geq 1} \|u^{k,l}\|_{H^{7}}^2(\|G_j(u^{k} )\|_{H^{7}}^2+\|G_j(u^{l} )\|_{H^{7}}^2)\\
&\lesssim\|u^{k,l}\|_{H^{7}}^2 (1+\|u^k\|_{H^{10}}^4+\|u^l\|_{H^{10}}^4),
 \end{split}
\end{equation*}
which means that
\begin{equation}\label{2.33}
\begin{split}
&|M_4(t)|\lesssim \int_0^t(1+\|\textbf{y}^{k }(r)\|_{\textbf{H}^{10 }}^4+\|\textbf{y}^{ l}(r)\|_{\textbf{H}^{10}}^4)\|\textbf{y}^{k,l}(r)\|_{\textbf{H}^{7}}^2 \mathrm{d}r.
 \end{split}
\end{equation}
For all $k,~l \in \mathbb{N}$ and $\mathbf{R}\geq 1$, let us define
\begin{equation}\label{stop1}
\begin{split}
\textbf{t}^{k,l,T,\mathbf{R}}:=\textbf{t}^{k,T,\mathbf{R}}\wedge\textbf{t}^{l,T,\mathbf{R}},\quad\textbf{t}^{k,T,\mathbf{R}} :=T \wedge \inf\{t\geq 0;~ \|\textbf{y}^{k }(t)\|_{\textbf{H}^{10}}\geq \mathbf{R}\}.
 \end{split}
\end{equation}
By raising the $p$-th power on both sides of \eqref{2.31}, we deduce from the estimates \eqref{2.32}-\eqref{2.33} that
\begin{equation} \label{2.34}
\begin{split}
&\mathbb{E}\sup_{t\in[0,\textbf{t}^{k,l,T,\mathbf{R}}]}\|\textbf{y}^{k,l}(t)\|_{\textbf{H}^{7}}^{2p}\\
&\lesssim_{\phi,\epsilon,a,p,T} \mathbb{E}\int_0^{\textbf{t}^{k,l,T,\mathbf{R}}} (1+\|\textbf{y}^{k }(r)\|_{\textbf{H}^{10 }}^ {4p}+\|\textbf{y}^{ l}(r)\|_{\textbf{H}^{10 }}^{4p})\|\textbf{y}^{k,l}(r)\|_{\textbf{H}^{7}}^{2p} \mathrm{d}r\\
& \quad+\max\{\frac{1}{k^{2p}},\frac{1}{l^{2p}}\}\mathbb{E}\int_0^{\textbf{t}^{k,l,T,\mathbf{R}}}(1+\|\textbf{y}^{k }(r)\|_{\textbf{H}^{10 }}^{32p}+\|\textbf{y}^{l}(r)\|_{\textbf{H}^{10 }}^{32p})\mathrm{d}r+  \mathbb{E}\sup_{t\in[0,\textbf{t}^{k,l,T,\mathbf{R}}]}\left|M_5(t)\right|^p\\
&\lesssim_{\phi,\epsilon,a,p,T} (1+\mathbf{R}^ {4p})\mathbb{E}\int_0^{\textbf{t}^{k,l,T,\mathbf{R}}} \|\textbf{y}^{k,l}(r)\|_{\textbf{H}^{7}}^{2p} \mathrm{d}r+\max\{\frac{1}{k^p},\frac{1}{l^p}\}(1+\mathbf{R}^{32p})T\\
&\quad+  \mathbb{E}\sup_{t\in[0,\textbf{t}^{k,l,T,\mathbf{R}}]}\left|M_5(t)\right|^p.
 \end{split}
\end{equation}
By using the BDG inequality and Assumption \ref{as3}, we have
\begin{equation}\label{2.35}
\begin{split}
&\mathbb{E}\sup_{t\in[0,\textbf{t}^{k,l,T,\mathbf{R}}]}\left|M_5(t)\right|^p\\
&\leq C\mathbb{E}\left(\int_0^{\textbf{t}^{k,l,T,\mathbf{R}}}\|u^{k,l}(r)\|_{H^7}^4\|G(u^{k}(r) )\|_{\mathcal{L}_2(\mathbf{U};H^{7})}^2+\|u^{k,l}(r)\|_{H^7}^4\mathrm{d}r\right)^{\frac{p}{2}}\\
&\leq \frac{1}{2}\mathbb{E}\sup_{t\in[0,\textbf{t}^{k,l,T,\mathbf{R}}]}\|\textbf{y}^{k,l}(t)\|_{\textbf{H}^{7}}^{2p}
 +C_{\mathbf{R},p,T}\mathbb{E} \int_0^{\textbf{t}^{k,l,T,\mathbf{R}}}(1+ \|\textbf{y}^{k}(r)\|_{\textbf{H}^{7}}^{2p})\|\textbf{y}^{k,l}(r)\|_{\textbf{H}^{7}}^{2p}\mathrm{d}r\\
&\leq\frac{1}{2}\mathbb{E}\sup_{t\in[0,\textbf{t}^{k,l,T,\mathbf{R}}]}\|\textbf{y}^{k,l}(t)\|_{\textbf{H}^{7}}^{2p}+C_{\mathbf{R},p,T}(1+\mathbf{R}^{2p})\mathbb{E} \int_0^{\textbf{t}^{k,l,T,\mathbf{R}}}\|\textbf{y}^{k,l}(r)\|_{\textbf{H}^{7}}^{2p}\mathrm{d}r.
\end{split}
\end{equation}
Plugging \eqref{2.35} into \eqref{2.34}, it follows that
\begin{equation*}
\begin{split}
&\mathbb{E}\sup_{t\in[0,\textbf{t}^{k,l,T,\mathbf{R}}]}\|\textbf{y}^{k,l}(t)\|_{\textbf{H}^{7}}^{2p}\lesssim_{\phi,\epsilon,a} \max\{\frac{1}{k^p},\frac{1}{l^p}\}(1+\mathbf{R}^{32p})T+C_{\mathbf{R},p,T}\mathbb{E}\int_0^{\textbf{t}^{k,l,T,\mathbf{R}}} \|\textbf{y}^{k,l}(r)\|_{\textbf{H}^{7}}^{2p} \mathrm{d}r.
 \end{split}
\end{equation*}
According to the Gronwall lemma, we have
\begin{equation*}
\begin{split}
\mathbb{E}\sup_{t\in[0,\textbf{t}^{k,l,T,\mathbf{R}}]}\|\textbf{y}^{k,l}(t)\|_{\textbf{H}^{7}}^{2p}\lesssim_{\phi,\epsilon,a} \max\{\frac{1}{k^p},\frac{1}{l^p}\}(1+\mathbf{R}^{32p})Te^{C_{\mathbf{R},p,T}T},
 \end{split}
\end{equation*}
which implies that
\begin{equation}\label{2.36}
\begin{split}
 \lim_{k\rightarrow\infty} \sup_{l\geq k}\mathbb{E}\sup_{t\in[0,\textbf{t}^{k,l,T,\mathbf{R}}]}\|\textbf{y}^{k}(t)-\textbf{y}^{l}(t)
 \|_{\textbf{H}^{7}}^{2p}=0,\quad\mathbf{R}\geq 1.
 \end{split}
\end{equation}
For any $h>0$, let us introduce the following events
$$
E^{k,l}(r):= \left\{\omega \in \Omega;~\sup_{t\in[0,r]}\|\textbf{y}^{k}(t,\omega)-\textbf{y}^{l}(t,\omega)
 \|_{\textbf{H}^{7}}^{2p}>h\right\},\quad\textrm{for all}~r> 0.
$$
According to the definition of the stopping times $\textbf{t}^{k,l,T,\mathbf{R}}$ in \eqref{stop1}, we see that for any $T>0$
\begin{equation*}
\begin{split}
 E^{k,l}(T)
& = \left(E^{k,l}(T)\cap \{\textbf{t}^{k,l,T,\mathbf{R}} =T\}\right) \cup \left(E^{k,l}(T)\cap\{\textbf{t}^{k,l,T,\mathbf{R}} <T\}\right) \\
& \subset E^{k,l}(\textbf{t}^{k,l,T,\mathbf{R}})\cup \{\textbf{t}^{k,T,\mathbf{R}} \leq T\}\cup\{\textbf{t}^{ l,T,\mathbf{R}} \leq T\}.
 \end{split}
\end{equation*}
By using the estimates in Lemma \ref{lem2} and the Chebyshev inequality, it follows that
\begin{equation*}
\begin{split}
\mathbb{P}\{E^{k,l}(T)\}&\leq \mathbb{P}\{E^{k,l}(\textbf{t}^{k,l,T,\mathbf{R}})\}+\mathbb{P}\{\textbf{t}^{k,T,\mathbf{R}} \leq T\}+\mathbb{P}\{\textbf{t}^{l,T,\mathbf{R}}\leq T\}\\
&\lesssim_{R,p,\phi,\epsilon,a,n_0,c_0,u_0,T} \frac{1}{\mathbf{R}^2}+ \mathbb{P}\left\{\sup_{t\in[0,\textbf{t}^{k,l,T,\mathbf{R}}]}\|\textbf{y}^{k}(t)-\textbf{y}^{l}(t)
\|_{\textbf{H}^{7}}^{2p}>h\right\}\\
&\lesssim_{R,p,\phi,\epsilon,a,n_0,c_0,u_0,T} \frac{1}{\mathbf{R}^2}+ \frac{1}{h}\mathbb{E}\sup_{t\in[0,\textbf{t}^{k,l,T,\mathbf{R}}]}\|\textbf{y}^{k}(t)-\textbf{y}^{l}(t)
\|_{\textbf{H}^{7}}^{2p}.
\end{split}
\end{equation*}
By using \eqref{2.36} and taking the limit $\mathbf{R}\rightarrow\infty$, the last inequality implies the result \eqref{22.13}. Thus $\textbf{y}^{k}$ converges in probability to an element $\textbf{y}$ in $\mathcal {C}([0,T];\textbf{H}^{7}(\mathbb{R}^3) )$, as $k\rightarrow\infty$. As a result, it follows from the Riesz theorem that there exists a subsequence of $\{\textbf{y}^{k}\}_{k\geq1}$, still denoted by itself, such that $\textbf{y}^{k}\rightarrow\textbf{y}$ strongly in $\mathcal {C}([0,T];\textbf{H}^{7}(\mathbb{R}^3) )$, $\mathbb{P}$-a.s. Moreover, according to the uniqueness of the limit in Hausdorff space  and the uniform bound in Lemma \ref{lem2}, it is not difficult to verify that $\textbf{y}\in L^p(\Omega;L^\infty(0,T;\textbf{H}^{10}(\mathbb{R}^3) )$. This completes the proof of Lemma \ref{lem2.3}.
\end{proof}

Based on the uniform bound provided by Lemma \ref{lem2} and the strong convergence result provided by Lemma \ref{lem2.3}, we are able to prove the existence and uniqueness of a global pathwise solution to the modified systsem \eqref{Mod-1} with cut-off functions.
\begin{lemma}\label{lem2.4}
For any fixed $R>1$, $ \epsilon \in (0,1)$ and $T>0$, the system \eqref{Mod-1} with cut-off functions has a unique solution $\textbf{y}^{\epsilon,R} \in L^p(\Omega;\mathcal {C}([0,T];\textbf{H}^{7}(\mathbb{R}^3)))\cap L^p(\Omega;L^{\infty}(0,T;\textbf{H}^{10}(\mathbb{R}^3)))$.
\end{lemma}

\begin{proof}[\emph{\textbf{Proof}}]
For simplicity, we use the notation $(n,c,u)$ instead of  $(n^{\epsilon,R},c^{\epsilon,R},u^{\epsilon,R})$ in the proof. By Lemma \ref{lem2.3} and the fact of $H^{3}(\mathbb{R}^3)\hookrightarrow W^{1,\infty}(\mathbb{R}^3)$, one can take the limit as $k\rightarrow \infty$ in \eqref{Mod-22} to conclude that the limit $\textbf{y}$ solves \eqref{Mod-1} with cut-off functions, and
$
\textbf{y} \in L^\infty(0,T;\textbf{H}^{10}(\mathbb{R}^3))\bigcap  \mathcal {C}([0,T];\textbf{H}^{7}(\mathbb{R}^3))$, $\mathbb{P}$-a.s.

To complete the proof of Lemma \ref{lem2.3}, we shall only prove the uniqueness of the solution. Assume that $(\textbf{y}_1,\textbf{t}_1)$ and  $(\textbf{y}_2,\textbf{t}_2)$ are two local solutions to \eqref{Mod-1} with cut-off functions with respect to the same initial data $\textbf{y}_0$.  Let $\mathcal{L}> 2\| \textbf{y}_0\|_{\textbf{H}^{7}}$. For any $T>0$, define $\textbf{t}^\mathcal{L}_T:= \textbf{t}^\mathcal{L} \wedge T$, where
\begin{equation*}
\begin{split}
\textbf{t}^\mathcal{L}:=\inf\{t> 0;~\| \textbf{y}_1(t)\|_{\textbf{H}^{7}}+\| \textbf{y}_2(t)\|_{\textbf{H}^{7}}\geq \mathcal{L}\}.
 \end{split}
\end{equation*}
Then $\mathbb{P}\{\textbf{t}^\mathcal{L}>0\}=1$ for $\mathcal{L}$ large enough. Invoking the Fatou lemma, we see from Lemma \ref{lem2} that
$$
\mathbb{E}\|\textbf{y}_i\| _{L^{\infty}(0,T; \textbf{H}^{7} )}^p\leq\mathbb{E}\|\textbf{y}_i\| _{L^{\infty}(0,T; \textbf{H}^{10} )}^p \lesssim_{R,p,\phi,\epsilon,a,n_0,c_0,u_0,T} 1,\quad p\geq2,\quad i=1,~2.
$$
Thus, we have
\begin{equation}\label{2.44}
\begin{split}
\mathbb{P}\left\{\liminf_{\mathcal{L}\rightarrow\infty}\textbf{t}^\mathcal{L} \geq \textbf{t}_1\wedge\textbf{t}_2\right\}=1.
 \end{split}
\end{equation}
Define $\textbf{y}^*=\textbf{y}_1-\textbf{y}_2$, then there holds
\begin{equation*}
\begin{split}
\textbf{y}^*(t)= \int_0^t[\textbf{F}(\textbf{y}_1(r))-\textbf{F}(\textbf{y}_2(r))] \mathrm{d} r+ \int_0^t[\textbf{L}(\textbf{y}_1(r))-\textbf{L}(\textbf{y}_2(r))]\mathrm{d}\mathbf{W}(r),
 \end{split}
\end{equation*}
where $\textbf{F}(\cdot)=(\textbf{F}_1(\cdot),\textbf{F} _2(\cdot),\textbf{F}_3(\cdot))^{\top}$,
$\textbf{F}_1 (\textbf{y} ) =\Delta  n -\theta_R(\|(u,n)\|_{W^{1,\infty}})u\cdot \nabla  n- \theta_R(\|(n,c)\|_{W^{1,\infty}})\nabla\cdot(n\nabla(c*\rho^{\epsilon}))+\theta_R(\|n\|_{W^{1,\infty}})(-an+(1+a)n^2-n^3)$, $\textbf{F}_2 (\textbf{y} )= \Delta  c-\theta_R(\|(u,c)\|_{W^{1,\infty}})  u\cdot \nabla  c  -\theta_R(\|(n,c)\|_{W^{1,\infty}})c(n*\rho^{\epsilon})$ and
$\textbf{F}_3 (\textbf{y} )= \Delta  u-\theta_R(\|u\|_{W^{1,\infty}}) \textbf{P} (u\cdot \nabla)  u + \theta_R(\|n\|_{W^{1,\infty}})\textbf{P} (n\nabla \phi)*\rho^{\epsilon}-\theta_R(\|u\|_{W^{1,\infty}})\textbf{P}\mathbf{g}(|u|^2)u$. The aforementioned integral equation can be regarded as holding in the Hilbert space $\textbf{H}^5(\mathbb{R}^3)$, and it is not difficult to verify the validity of It\^{o}'s formula in $\textbf{H}^5(\mathbb{R}^3)$. Thus we have
\begin{equation}\label{2.46}
\begin{split}
&\|\textbf{y}^*(t)\|_{\textbf{H}^{5}}^2 +2\int_0^t \|\nabla \textbf{y}^*(r)\|_{\textbf{H}^{5}}^2\mathrm{d}r \\
&\leq C\int_0^t\|\textbf{y}^*(r)\|_{\textbf{H}^{5}}^2\mathrm{d}r+2\int_0^t|(\textbf{y}^*(r), \textbf{F}(\textbf{y}_1(r)) -\textbf{F}(\textbf{y}_2(r)))_{\textbf{H}^{5}}|\mathrm{d}r \\
&\quad+2\sum_{j\geq 1}\left|\int_0^t(\textbf{y}^*(r),(\textbf{L}(\textbf{y}_1(r))-\textbf{L}(\textbf{y}_2(r)))e_j )_{\textbf{H}^{5}} \mathrm{d} W_j(r)\right|\\
&:=C\int_0^t\|\textbf{y}^*(r)\|_{\textbf{H}^{5}}^2\mathrm{d}r+2O_1(t)+2O_2(t),
\end{split}
\end{equation}
where $O_i=O_i^{(1)}+O_i^{(2)}+O_i^{(3)},~i=1,~2$. For $O_1(t)$, we only handle the third component $O_1^{(3)}(t)$,  and the estimates for the other two components are similar to the discussion in Lemma \ref{lem2}. Note that
\begin{equation}\label{2.47}
\begin{split}
O_1^{(3)}(t)
&\leq\int_0^t[|\left(u^*(r), \theta_R(\|u_1(r)\|_{W^{1,\infty}})(u_1(r)\cdot \nabla)  u_1(r)-\theta_R(\|u_2(r)\|_{W^{1,\infty}})(u_2(r)\cdot \nabla)  u_2(r) \right)_{H^{5}}|\\
&\quad+|\left(u^*(r),\theta_R(\|n_1(r)\|_{W^{1,\infty}})(n_1(r)\nabla\phi)*\rho^{\epsilon}-\theta_R(\|n_2(r)\|_{W^{1,\infty}})(n_2(r)\nabla\phi)*\rho^{\epsilon}\right) _{H^{5}}| \\
&\quad+|\left(u^*(r), \theta_R(\|u_1(r)\|_{W^{1,\infty}})\mathbf{g}(|u_1(r)|^2)u_1(r)-\theta_R(\|u_2(r)\|_{W^{1,\infty}})\mathbf{g}(|u_2(r)|^2)u_2(r) \right)_{H^{5}}|]\mathrm{d}r\\
&:=\int_0^t(S_1(r)+S_2(r)+S_3(r))\mathrm{d}r.
 \end{split}
\end{equation}
Since $H^5(\mathbb{R}^3)$ is a Banach algebra, it is not difficult to deduce that
\begin{equation}\label{S0}
\begin{split}
S_1+S_2\leq\frac{1}{2}\|\nabla u^*\|_{H^{5}}^2+C(\| u_1\|_{H^{7}}^2+\| u_2\|_{H^{7}}^2)\|u^*\|_{H^{5}}^2+C_{\phi,\epsilon} (\|u^*\| _{H^{5}} ^2+\| n^* \|_{H^{5} } ^2).
\end{split}
\end{equation}
For $S_3$, by using the triangle inequality and Mean value theorem, we have
\begin{equation}\label{S1}
\begin{split}
S_3&\lesssim\|u^*\|_{H^{5}}^2+\theta_R(\|u_1\|_{W^{1,\infty}})^2\|\mathbf{g}(|u_1|^2)u_1-\mathbf{g}(|u_2|^2)u_2\|_{H^5}^2\\
&\quad+|\theta_R(\|u_1\|_{W^{1,\infty}})-\theta_R(\|u_2\|_{W^{1,\infty}})|^2\|\mathbf{g}(|u_2|^2)u_2\|_{H^5}^2\\
&\lesssim\|\mathbf{g}(|u_1|^2)-\mathbf{g}(|u_2|^2)\|_{H^5}^2\|u_1\|_{H^{5}}^2+(1+\|\mathbf{g}(|u_2|^2)\|_{H^5}^2+\|\mathbf{g}(|u_2|^2)\|_{H^5}^2\|u_2\|_{H^{5}}^2)\|u^*\|_{H^{5}}^2.
\end{split}
\end{equation}
By using the Fa\`{a} di Bruno formula (cf. \cite[Formula (1.1)]{johnson2002curious}), given a  multi-index $\alpha$ with $|\alpha|=s\geq2$, it follows that $D^{\alpha}[\mathbf{g}(|u_1|^2)-\mathbf{g}(|u_2|^2)]$ is a sum of terms
$
\mathbf{g}^{(l)}(|u_1|^2)\prod_{j=1}^lD^{\beta_j}(|u_1|^2)-\mathbf{g}^{(l)}(|u_2|^2)\prod_{j=1}^lD^{\beta_j}(|u_2|^2),
$
where $l\in\{1,\cdots,s\}$ and the $\beta_j$'s are multi-indices such that $\alpha=\beta_1+\cdots+\beta_l$ and $|\beta_j|\geq1$. Each of the above terms can itself be decomposed as a sum of terms, where the first one is
\begin{equation*} \begin{split}
[\mathbf{g}^{(l)}(|u_1|^2)-\mathbf{g}^{(l)}(|u_2|^2)]\prod_{j=1}^lD^{\beta_j}(|u_1|^2).
\end{split} \end{equation*}
The other ones have the form $
\mathbf{g}^{(l)}(|u_2|^2)\prod_{j=1}^lD^{\beta_j}(|w_j|^2)$,
where all the $|w_j|^2$'s are equal to $|u_1|^2$ or $|u_2|^2$, except one which is equal to $|u_1|^2-|u_2|^2$. Let $p_j=\frac{2s}{|\beta_j|}$, so that $\sum_{j=1}^l\frac{1}{p_j}=\frac{1}{2}$. Then, by \eqref{GGN1}-\eqref{GGN2} and the Sobolev embedding $H^s(\mathbb{R}^3)\hookrightarrow L^{\infty}(\mathbb{R}^3)$, we deduce from \eqref{con1} that
\begin{equation}\label{aa2}
\begin{split}
\left\|\mathbf{g}^{(l)}(|u_2|^2)\prod_{j=1}^lD^{\beta_j}(|w_j|^2)\right\|_{L^2}
&\lesssim\prod_{j=1}^l\big\|D^{\beta_j}(|w_j|^2)\big\|_{L^{p_j}}\\
& \lesssim\||u_1|^2\|_{H^s}^k\||u_2|^2\|_{H^s}^{l-k-1}\||u_1|^2-|u_2|^2\|_{H^s}\\
& \lesssim(1+\|u_1\cdot u_1\|_{H^s}^l)(1+\|u_2\cdot u_2\|_{H^s}^l)\|(u_1+u_2)\cdot(u_1-u_2)\|_{H^s}\\
& \lesssim(1+\|u_1\|_{H^s}^{8l}+\|u_2\|_{H^s}^{8l})\|u_1-u_2\|_{H^s}.
\end{split} \end{equation}
Similarly, by using \eqref{con2}, it follows that
\begin{equation}\label{aa3}
\begin{split}
&\left\|[\mathbf{g}^{(l)}(|u_1|^2)-\mathbf{g}^{(l)}(|u_2|^2)]\prod_{j=1}^lD^{\beta_j}(|u_1|^2)\right\|_{L^2}\\
&\lesssim\|\mathbf{g}^{(l)}(|u_1|^2)-\mathbf{g}^{(l)}(|u_2|^2)\|_{L^{\infty}}\prod_{j=1}^l
\left\|D^{\beta_j}(|u_1|^2)\right\|_{L^{p_j}}\\
&\lesssim\|u_1+u_2\|_{L^{\infty}}\|u_1-u_2\|_{L^{\infty}}\|u_1\cdot u_1\|_{H^s}^l\\
&\lesssim(1+\|u_1\|_{H^s}^{4l}+\|u_2\|_{H^s}^{4l})\|u_1-u_2\|_{H^s}.
\end{split} \end{equation}
Thereby, for every integer $s\geq2$, we infer from \eqref{aa2}-\eqref{aa3} that
\begin{equation}\label{aa4}
\begin{split}
\|\mathbf{g}(|u_1|^2)-\mathbf{g}(|u_2|^2)\|_{H^s}^2&\lesssim_s\|\mathbf{g}(|u_1|^2)-\mathbf{g}(|u_2|^2)\|_{L^2}^2+\sum_{|\alpha|=s}\|D^{\alpha}[\mathbf{g}(|u_1|^2)-\mathbf{g}(|u_2|^2)]\|_{L^2}^2\\
&\lesssim_{s}(1+\|u_1\|_{H^s}^{16s}+\|u_2\|_{H^s}^{16s})\|u_1-u_2\|_{H^s}^2.
\end{split}
\end{equation}
As a result, by using \eqref{Gu1} and \eqref{aa4}, we see from \eqref{S1} that
\begin{equation}\label{aa5}
\begin{split}
&S_3\lesssim(1+\|u_1\|_{H^5}^{82}+\|u_2\|_{H^5}^{82})\|u^*\|_{H^{5}}^2.
\end{split}
\end{equation}
Plugging \eqref{S0} and \eqref{aa5} into \eqref{2.47}, we infer that
\begin{equation}\label{S4}
\begin{split}
O_1^{(3)}(t)
\leq\frac{1}{2}\int_0^t\|\nabla u^*(r)\|_{H^{5}}^2\mathrm{d}r+C_{\phi,\epsilon}\int_0^t(1+\|u_1(r)\| _{H^{7}} ^{82}+\|u_2(r)\| _{H^{7}} ^{82}) (\|u^*(r)\| _{H^{5}} ^2+\| n^*(r) \|_{H^{5} } ^2)\mathrm{d}r.
 \end{split}
\end{equation}
Similarly, it is not hard to derive that
\begin{equation*}
\begin{split}
&O_1^{(1)}(t)+O_2^{(1)}(t)\leq\frac{1}{2}\int_0^t\|\nabla\mathbf{y}^*(r)\|_{\textbf{H}^5}^2\mathrm{d}r+C_{a,\epsilon}\int_0^t(1+\|\mathbf{y}_1(r)\|_{\textbf{H}^{7}}^4+\|\mathbf{y}_2(r)\|_{\textbf{H}^{7}}^4)\|\mathbf{y}^*(r)\|_{\textbf{H}^5}^2\mathrm{d}r,
\end{split}
\end{equation*}
which together with \eqref{S4} implies that
\begin{equation}\label{2.51}
\begin{split}
&O_1 (t)\leq\frac{1}{2}\int_0^t\|\nabla\mathbf{y}^*(r)\|_{\textbf{H}^5}^2\mathrm{d}r+C_{\phi,a,\epsilon}\int_0^t (1+\|\mathbf{y}_1(r)\|_{\textbf{H}^{7}}^{82}+\|\mathbf{y}_2(r)\|_{\textbf{H}^{7}}^{82})\|\mathbf{y}^*(r)\|_{\textbf{H}^5}^2\mathrm{d}r.
\end{split}
\end{equation}
For $O_2(t)$, by applying the BDG inequality, we get
\begin{equation} \label{2.52}
\begin{split}
\mathbb{E}\sup_{r\in[0,\textbf{t}^\mathcal{L}_T\wedge t]}|O_2 (r)|&\lesssim \mathbb{E}\Bigg(\int_0^{\textbf{t}^\mathcal{L}_T\wedge t}\|\textbf{y}^*(r)\|_{\textbf{H}^{5}}^2(\sum_{j\geq 1}\|(\textbf{L}(\textbf{y}_1(r))-\textbf{L}(\textbf{y}_2(r)))e_j \|_{\textbf{H}^{5}}^2 )\mathrm{d}r\Bigg)^{1/2}\\
&\leq \frac{1}{2} \mathbb{E} \sup_{r\in[0,\textbf{t}^\mathcal{L}_T\wedge t]}\|\textbf{y}^*(r)\|_{\textbf{H}^{5}}^2+ C_{\mathcal{L}} \mathbb{E}\int_0^{\textbf{t}^\mathcal{L}_T\wedge t} \|\textbf{y}^*(r)\|_{\textbf{H}^{5}} ^2\mathrm{d}r .
\end{split}
\end{equation}
By using the estimates \eqref{2.51}-\eqref{2.52} and noting the definition of $\textbf{t}^\mathcal{L}_T$, we get from \eqref{2.46} that
\begin{equation*}
\begin{split}
\mathbb{E}\sup_{r\in[0,\textbf{t}^\mathcal{L}_T\wedge t]}\|\textbf{y}^*(r)\|_{\textbf{H}^{5}}^2  \lesssim_{\mathcal{L},\phi,a,\epsilon} \int_0^{ t} \mathbb{E}\sup_{r\in[0,\textbf{t}^\mathcal{L}_T\wedge \tau]} \|\textbf{y}^*(r)\|_{\textbf{H}^{5}} ^2\mathrm{d}  \tau,
\end{split}
\end{equation*}
which combined with the Gronwall lemma implies
$
\mathbb{E}\sup_{r\in[0,\textbf{t}^\mathcal{L}\wedge \textbf{t}_1\wedge\textbf{t}_2 ]}\|\textbf{y}^*(r)\|_{\textbf{H}^{5}}^2=0.
$
Taking the limit as $\mathcal{L}\rightarrow\infty$ and using \eqref{2.44}, we get from the Monotone convergence theorem that
$$
\mathbb{P}\{\textbf{y}_1(t,x)=\textbf{y}_2(t,x),\quad\forall (t,x) \in [0,\textbf{t}_1\wedge\textbf{t}_2]\times \mathbb{R}^3\}=1.
$$
The proof of Lemma \ref{lem2.4} is thus completed.
\end{proof}

Next, we will remove the cut-off functions to obtain the global solution of \eqref{Mod-1}. The specific result is presented as follows.

\begin{lemma}\label{lem2.5}
For any $0<\epsilon<1$ and $T>0$,   the regularized system \eqref{Mod-1} has a global unique pathwise solution $(n^{\epsilon} ,c^{\epsilon} ,u^{\epsilon} )$ such that
\begin{equation*}
\begin{split}
&n^{\epsilon} \in \mathcal {C}([0,T];H^{7}(\mathbb{R}^3))\cap \mathcal {C}^1([0,T];H^{5}(\mathbb{R}^3)),\\
&c^{\epsilon} \in \mathcal {C}([0,T];H^{7}(\mathbb{R}^3))\cap \mathcal {C}^1([0,T];H^{5}(\mathbb{R}^3)),\\
&u^{\epsilon} \in \mathcal {C}([0,T];\mathbb{H}^{7}),\quad \mathbb{P}\textrm{-a.s.}
\end{split}
\end{equation*}
\end{lemma}

\begin{proof}[\emph{\textbf{Proof}}]
The proof is divided into two steps. At first, we shall show the existence and uniqueness of local maximal pathwise solutions. For any given $R>1$, let $\textbf{y}^{\epsilon,R}(t)$ be the global pathwise solution to \eqref{Mod-1} with cut-offs (cf. Lemma \ref{lem2.4}). To remove the cut-off functions, we define
\begin{equation*}
\begin{split}
\textbf{t}^\epsilon :=\inf\{t>0;~ \sup_{r\in [0,t]}\|\textbf{y}^{\epsilon,R}(r)\|_{\textbf{H}^{7}}> \|\textbf{y}^{\epsilon,R}(0)\|_{\textbf{H}^{7}}+1\}.
\end{split}
\end{equation*}
For any $R>0$, we have $\mathbb{P}\{\textbf{t}^\epsilon > 0\}=1$. Let $C_\epsilon>0$ be a constant such that $\|\textbf{y}^{\epsilon,R}(0)\|_{\textbf{H}^{7}}\leq C_\epsilon$, and   $C_{\textrm{emb}}>0$ be the Sobolev embedding constant such that $\|\cdot\|_{W^{1,\infty}}\leq C_{\textrm{emb}} \|\cdot\|_{\textbf{H}^{7}} $. Then we have
\begin{equation*}
\begin{split}
 \|\textbf{y}^{\epsilon,R}(t)\|_{W^{1,\infty}} \leq C_{\textrm{emb}} \|\textbf{y}^{\epsilon,R}(t)\|_{\textbf{H}^{7}} \leq C_{\textrm{emb}} (\|\textbf{y}^{\epsilon,R}(0)\|_{\textbf{H}^{7}} +1) \leq C_{\textrm{emb}} (C_\epsilon +1),
\end{split}
\end{equation*}
for any $t \in [0,\textbf{t}^\epsilon ]$.   If $R > C_{\textrm{emb}} (C_\epsilon +1)$, then  $\theta_R(\|\textbf{y}^{\epsilon,R}\|_{W^{1,\infty}})\equiv 1$, $\mathbb{P}$-a.s., which means that $(\textbf{y}^{\epsilon,R}, \textbf{t}^\epsilon )=(\textbf{y}^{\epsilon}, \textbf{t}^\epsilon )$ is a local pathwise solution to \eqref{Mod-1}. Then, by using a standard method (cf. \cite{Breit}), one can extend the solution $\textbf{y}^{\epsilon}$ to a maximal time of existence $\widetilde{\textbf{t}^\epsilon}$.

Secondly, we shall prove that the maximal solution can be extended to infinity, that is $\mathbb{P}\{\widetilde{\textbf{t}^\epsilon}=\infty\}=1$. To this end, we will leverage the parabolic nature of system \eqref{Mod-1} and an iterative process to progressively derive suitable $\mathbf{H}^7$-energy estimates for solution $\textbf{y}^{\epsilon}$. According to the standard reasoning, it is easy to obtain the nonnegativity of $n^\epsilon$ and $c^\epsilon$ (cf. \cite{winkler2012global}). As a result, we infer from the second equation of \eqref{Mod-1} that
\begin{equation}\label{0-1}
\begin{split}
&\frac{1}{2}\frac{\mathrm{d}}{\mathrm{d}t}\|c^{\epsilon}\|_{L^2}^2+\|\nabla c^{\epsilon}\|_{L^2}^2=-\int_{\mathbb{R}^3}|c^{\epsilon}|^2(n^{\epsilon}*\rho^\epsilon)\mathrm{d}x\leq 0.
\end{split}
\end{equation}
Similarly, we have
\begin{equation}\label{0-2}
\begin{split}
&\frac{1}{2}\frac{\mathrm{d}}{\mathrm{d}t}\|n^{\epsilon}\|_{L^2}^2+\|\nabla n^{\epsilon}\|_{L^2}^2+a\|n^{\epsilon}\|_{L^2}^2+\|n^{\epsilon}\|_{L^4}^4\\
&=(a+1)\|n^{\epsilon}\|_{L^3}^3+\int_{\mathbb{R}^3}n^{\epsilon}\nabla (c^{\epsilon}*\rho^\epsilon)\cdot\nabla n^{\epsilon}\mathrm{d}x\\
&\leq\frac{1}{2}\|n^{\epsilon}\|_{L^4}^4+C_a\|n^{\epsilon}\|_{L^2}^2+\frac{1}{2}\|\nabla n^{\epsilon}\|_{L^2}^2+C_{\epsilon}\|c^{\epsilon}\|_{L^2}^2\|n^{\epsilon}\|_{L^2}^2.
\end{split}
\end{equation}
By using Gronwall's lemma, we see from \eqref{0-1}-\eqref{0-2} that for every $t\in [0,T]$ and $\mathbb{P}$-a.s.
\begin{equation}\label{0-3}
\begin{split}
&\|n^{\epsilon}(t)\|_{L^2}^2+\|c^{\epsilon}(t)\|_{L^2}^2+\int_0^t(\|\nabla n^{\epsilon}(r)\|_{L^2}^2+\|\nabla c^{\epsilon}(r)\|_{L^2}^2)\mathrm{d}r\lesssim_{a,\epsilon,T,n_0,c_0}1.
\end{split}
\end{equation}
Moreover, by using the It\^{o} formula to $\|u^{\epsilon}\|_{L^2}^2$ and applying the BDG inequality, we have
\begin{equation*}
\begin{split}
\mathbb{E}\sup_{r\in [0,t]}\| u^{\epsilon}(r)\|_{L^2}^2+ \mathbb{E}\int_0^t\|\nabla  u^{\epsilon}(r)\|_{L^2}^2 \mathrm{d}r \lesssim_{\phi} \| u_0\|_{L^2}^2+  \mathbb{E} \int _0^t(1 + \|n^{\epsilon}(r)\|_{L^2}^2+ \|u^{\epsilon}(r)\|_{L^2}^2)\mathrm{d}r,
\end{split}
\end{equation*}
which combined with \eqref{0-3} implies that
\begin{equation}\label{0-4}
\begin{split}
\mathbb{E}\sup_{t\in [0,T]}\|(n^\epsilon,c^\epsilon,u^\epsilon)(t) \|_{\textbf{L}^2}^2 +\mathbb{E} \int_0^T  \|(n^\epsilon,c^\epsilon,u^\epsilon)(t) \|_{\textbf{H}^1}^2  \mathrm{d} t \lesssim _{a,\epsilon,\phi,T,\mathbf{y}_0} 1 .
\end{split}
\end{equation}
Next, we shall establish the $\textbf{H}^1$-estimate of solution $\textbf{y}^{\epsilon}$. Let us define
$$\tau^{L_0}:= \inf\{t>0;~ \sup_{r\in [0,t]}\|( n^\epsilon, c^\epsilon, u^\epsilon)(r)\|_{\textbf{L}^2}^2>L_0~ \textrm{or}~  \int_0^t \|( n^\epsilon, c^\epsilon, u^\epsilon)(r)\|_{\textbf{H}^1}^2 \mathrm{d}r>L_0\}.
$$
Then by \eqref{0-4}, we have $\tau^{L_0}\rightarrow\infty$ as $L_0\rightarrow\infty$, $\mathbb{P}$-a.s. Applying the It\^{o} formula to $\|\nabla u^{\epsilon}\|_{L^2}^2$, we have
\begin{equation}\label{0-5}
\begin{split}
&\frac{1}{2}\|\nabla u^\epsilon(t)\|_{L^2}^2+\int_0^t\|\Delta u^\epsilon(r)\|_{L^2}^2\mathrm{d}r\\
&=\frac{1}{2}\|\nabla u^\epsilon(0)\|_{L^2}^2+\int_0^t[(\mathbf{g}(|u^{\epsilon}(r)|^2)u^{\epsilon}(r),\Delta u^{\epsilon}(r))_{L^2}+((u^{\epsilon}(r)\cdot \nabla) u^{\epsilon}(r),\Delta u^{\epsilon}(r))_{L^2}\\
&\quad-((n^{\epsilon}(r)\nabla \phi)*\rho^\epsilon,\Delta u^{\epsilon}(r))_{L^2}]\mathrm{d}r+\frac{1}{2}\int_0^t\|\nabla \textbf{P}  G(u^{\epsilon}(r))\|_{\mathcal{L}_2(\mathbf{U};L^2)}^2\mathrm{d}r\\
&\quad+\sum_{j\geq1}\int_0^t(\nabla G _j(u^{\epsilon}(r)),\nabla u^{\epsilon}(r))_{L^2}\mathrm{d}W_j(r)\\
&\leq\|u^\epsilon(0)\|_{H^1}^2+\frac{1}{2}\int_0^t\|\Delta u^\epsilon(r)\|_{L^2}^2\mathrm{d}r+\int_0^t[(\mathbf{g}(|u^{\epsilon}(r)|^2)u^{\epsilon}(r),\Delta u^{\epsilon}(r))_{L^2}+\||u^{\epsilon}(r)||\nabla u^{\epsilon}(r)|\|_{L^2}^2\\
&\quad+\|\nabla\phi\|_{L^{\infty}}^2\|n^{\epsilon}(r)\|_{L^2}^2]\mathrm{d}r+C\int_0^t(1+\|u^{\epsilon}(r)\|_{H^1}^2)\mathrm{d}r+\sum_{j\geq1}\int_0^t(\nabla G _j(u^{\epsilon}(r)),\nabla u^{\epsilon}(r))_{L^2}\mathrm{d}W_j(r).
\end{split}
\end{equation}
By using the conditions \eqref{con1} and \eqref{connew1}, we infer that
\begin{equation}\label{0-6}
\begin{split}
(\mathbf{g}(|u|^2)u,\Delta u)_{L^2}&=\int_{\mathbb{R}^3}|u(x)|^2u(x)\cdot\Delta u(x)\mathrm{d}x-\int_{\mathbb{R}^3}\mathbf{g}_1(|u(x)|^2)u(x)\cdot\Delta u(x)\mathrm{d}x\\
&=-3\int_{\mathbb{R}^3}|u(x)|^2|\nabla u(x)|^2\mathrm{d}x+2\sum_{i=1}^3\int_{\mathbb{R}^3}\mathbf{g}_1'(|u(x)|^2)\langle u(x),\partial_iu(x)\rangle_{\mathbb{R}^3}^2\mathrm{d}x\\
&\quad+\int_{\mathbb{R}^3}\mathbf{g}_1(|u(x)|^2)|\nabla u(x)|^2\mathrm{d}x\\
&\leq-\int_{\mathbb{R}^3}|u(x)|^2|\nabla u(x)|^2\mathrm{d}x+C_{\mathbf{N}}\|\nabla u(x)\|_{L^2}^2.
\end{split}
\end{equation}
Plugging \eqref{0-6} into \eqref{0-5}, we have
\begin{equation}\label{000-6}
\begin{split}
&\|\nabla u^\epsilon(t)\|_{L^2}^2+\int_0^t\|\Delta u^\epsilon(r)\|_{L^2}^2\mathrm{d}r\\
&\lesssim_{\mathbf{N}}\|u^\epsilon(0)\|_{H^1}^2+\int_0^t(1+\|\nabla\phi\|_{L^{\infty}}^2\|n^{\epsilon}(r)\|_{L^2}^2+\|u^{\epsilon}(r)\|_{H^1}^2)\mathrm{d}r\\
&\quad+\sum_{j\geq1}\int_0^t(\nabla G _j(u^{\epsilon}(r)),\nabla u^{\epsilon}(r))_{L^2}\mathrm{d}W_j(r).
\end{split}
\end{equation}
By using the BDG inequality and estimate \eqref{0-4}, we infer from \eqref{000-6} that
\begin{equation*}
\begin{split}
&\mathbb{E}\sup_{r\in [0,t]}\|\nabla u^{\epsilon}(r)\|_{L^2}^2+\mathbb{E}\int_0^t\|\Delta u^{\epsilon}(r)\|_{L^2}^2 \mathrm{d}r \lesssim C_{a,\epsilon,\phi,T,\mathbf{y}_0}+  C_{\mathbf{N}}\mathbb{E} \int _0^t \|\nabla u^{\epsilon}(r)\|_{L^2}^2\mathrm{d}r,
\end{split}
\end{equation*}
which combined with Gronwall's lemma means that
\begin{equation}\label{0-7}
\begin{split}
&\mathbb{E}\sup_{t\in [0,T]}\|u^{\epsilon}(t)\|_{H^1}^2+ \mathbb{E}\int_0^T\|u^{\epsilon}(t)\|_{H^2}^2 \mathrm{d}t \lesssim_{a,\epsilon,\phi,\mathbf{N},T,\mathbf{y}_0}1.
\end{split}
\end{equation}
Let us define
$$\tau^{L_1}:= \inf\{t>0;~ \sup_{r\in [0,t]}\| u^\epsilon(r)\|_{H^1}^2>L_1~ \textrm{or}~  \int_0^t \|u^\epsilon(r)\|_{H^2}^2 \mathrm{d}r>L_1\}.
$$
Then by \eqref{0-7}, we have $\tau^{L_1}\rightarrow\infty$ as $L_1\rightarrow\infty$ a.s. By using the embedding $H^2(\mathbb{R}^3)\hookrightarrow L^{\infty}(\mathbb{R}^3)$, and the equation satisfied by $n^\epsilon$, we have
\begin{equation}\label{0-8}
\begin{split}
&\frac{1}{2}\frac{\mathrm{d}}{\mathrm{d}t}\|\nabla n^{\epsilon}\|_{L^2}^2+\|\Delta n^{\epsilon}\|_{L^2}^2+a\|\nabla n^{\epsilon}\|_{L^2}^2+3\|n^{\epsilon}\nabla n^{\epsilon}\|_{L^2}^2\\
&=(u^{\epsilon}\cdot \nabla n^{\epsilon},\Delta n^{\epsilon})_{L^2}+(\nabla\cdot(n^{\epsilon}\nabla (c^{\epsilon}*\rho^\epsilon)),\Delta n^{\epsilon})_{L^2}+2(1+a)(n^{\epsilon}\nabla n^{\epsilon},\nabla n^{\epsilon})_{L^2}\\
&\quad\leq\frac{1}{2}\|\Delta n^{\epsilon}\|_{L^2}^2+C\|u^{\epsilon}\|_{L^{\infty}}^2\|\nabla n^{\epsilon}\|_{L^2}^2+C\|\nabla(c^{\epsilon}*\rho^\epsilon)\|_{L^{\infty}}^2\|\nabla n^{\epsilon}\|_{L^2}^2\\
&\quad+C\|\Delta(c^{\epsilon}*\rho^\epsilon)\|_{L^{\infty}}^2\|n^{\epsilon}\|_{L^2}^2+\|n^{\epsilon}\nabla n^{\epsilon}\|_{L^2}^2+C_a\|\nabla n^{\epsilon}\|_{L^2}^2\\
&\leq\frac{1}{2}\|\Delta n^{\epsilon}\|_{L^2}^2+\|n^{\epsilon}\nabla n^{\epsilon}\|_{L^2}^2+C\|u^{\epsilon}\|_{H^{2}}^2\|\nabla n^{\epsilon}\|_{L^2}^2+C_{\epsilon}\|c^{\epsilon}\|_{L^{2}}^2\|\nabla n^{\epsilon}\|_{L^2}^2\\
&\quad+C_{\epsilon}\|c^{\epsilon}\|_{L^{2}}^2\|n^{\epsilon}\|_{L^2}^2+C_a\|\nabla n^{\epsilon}\|_{L^2}^2,
\end{split}
\end{equation}
which implies that
\begin{equation*}
\begin{split}
&\frac{\mathrm{d}}{\mathrm{d}t}\|\nabla n^{\epsilon}\|_{L^2}^2+\|\Delta n^{\epsilon}\|_{L^2}^2\lesssim_{\epsilon,a}\|c^{\epsilon}\|_{L^{2}}^2\|n^{\epsilon}\|_{L^2}^2+(1+\|u^{\epsilon}\|_{H^{2}}^2+\|c^{\epsilon}\|_{L^{2}}^2)\|\nabla n^{\epsilon}\|_{L^2}^2.
\end{split}
\end{equation*}
Thus by applying Gronwall's lemma and noting the definition of $\tau^{L_i},~i=0,~1$, we derive that
\begin{equation*}
\begin{split}
&\|\nabla n^{\epsilon}(t\wedge\tau^{L_0}\wedge\tau^{L_1})\|_{L^2}^2+\int_0^{t\wedge\tau^{L_0}\wedge\tau^{L_1}}\|\Delta n^{\epsilon}(r)\|_{L^2}^2\mathrm{d}r\\
&\lesssim_{\epsilon,a}e^{C_{\epsilon,a}\int_0^{t\wedge\tau^{L_0}\wedge\tau^{L_1}}(1+\|u^{\epsilon}(r)\|_{H^{2}}^2+\|c^{\epsilon}(r)\|_{L^{2}}^2)\mathrm{d}r}\\
&\quad\times\left(\|n^{\epsilon}(0)\|_{H^1}^2+\int_0^{t\wedge\tau^{L_0}\wedge\tau^{L_1}}\|c^{\epsilon}(r)\|_{L^{2}}^2\|n^{\epsilon}(r)\|_{L^2}^2\mathrm{d}r\right)\\
&\lesssim_{\epsilon,a}e^{C_{\epsilon,a,T,L_0,L_1}}\left(\|n^{\epsilon}(0)\|_{H^1}^2+TL_0^2\right)\\
&\lesssim_{\epsilon,a,T,L_0,L_1,n_0}1,
\end{split}
\end{equation*}
which combined with \eqref{0-4} implies that
\begin{equation}\label{0-9}
\begin{split}
&\mathbb{E}\sup_{t\in [0,T]}\|n^{\epsilon}(t\wedge\tau^{L_0}\wedge\tau^{L_1})\|_{H^1}^2+ \mathbb{E}\int_0^{T\wedge\tau^{L_0}\wedge\tau^{L_1}}\|n^{\epsilon}(t)\|_{H^2}^2 \mathrm{d}t \lesssim_{a,\epsilon,\phi,\mathbf{N},T,L_0,L_1,\mathbf{y}_0}1.
\end{split}
\end{equation}
By applying a similar argument as that for \eqref{0-8}-\eqref{0-9}, it is not difficult to obtain
\begin{equation*}
\begin{split}
&\mathbb{E}\sup_{t\in [0,T]}\|c^{\epsilon}(t\wedge\tau^{L_0}\wedge\tau^{L_1})\|_{H^1}^2+ \mathbb{E}\int_0^{T\wedge\tau^{L_0}\wedge\tau^{L_1}}\|c^{\epsilon}(t)\|_{H^2}^2 \mathrm{d}t \lesssim_{a,\epsilon,\phi,\mathbf{N},T,L_0,L_1,\mathbf{y}_0}1.
\end{split}
\end{equation*}
Therefore, we have
\begin{equation}\label{0-10}
\begin{split}
&\mathbb{E}\sup_{t\in [0,T]}\|(n^\epsilon,c^\epsilon,u^\epsilon)(t\wedge\tau^{L_0}\wedge\tau^{L_1})\|_{\textbf{H}^1}^2 +\mathbb{E} \int_0^{T\wedge\tau^{L_0}\wedge\tau^{L_1}}  \|(n^\epsilon,c^\epsilon,u^\epsilon)(t) \|_{\textbf{H}^2}^2  \mathrm{d} t\\
&\lesssim_{a,\epsilon,\phi,\mathbf{N},T,L_0,L_1,\mathbf{y}_0}1.
\end{split}
\end{equation}
For any $L_2>0$, let us define
$$\tau^{L_2}:= \inf\{t>0;~ \sup_{r\in [0,t]}\|( n^\epsilon, c^\epsilon, u^\epsilon)(r)\|_{\textbf{H}^1}^2>L_2~ \textrm{or}~  \int_0^t \|( n^\epsilon, c^\epsilon, u^\epsilon)(r)\|_{\textbf{H}^2}^2 \mathrm{d}r>L_2\}.
$$
Obviously, by \eqref{0-10}, we have $T\wedge\tau^{L_0}\wedge\tau^{L_1}\wedge\tau^{L_2}\rightarrow T\wedge\tau^{L_0}\wedge\tau^{L_1}$ as $L_2\rightarrow\infty$ a.s. For the sake of convenience, we denote
$$\textbf{t}^{L}_i:=\tau^{L_0}\wedge\cdots\wedge\tau^{L_i}$$
in the following discussion. Now, applying the It\^{o} formula to $\|\Delta u^{\epsilon}\|_{L^2}^2$, we infer from the Sobolev inequality $\|f\|_{L^6}\lesssim\|\nabla f\|_{L^2}$ (cf. \cite{brezis2011functional}), the estimate
\begin{equation*}
\begin{split}
\|\nabla\mathbf{g}(|u|^2)\|_{L^2}^2&\leq\|\nabla(|u|^2)\|_{L^2}^2+\|\nabla\mathbf{g}_1(|u|^2)\|_{L^2}^2\lesssim\||u||\nabla u|\|_{L^2}^2+\|\mathbf{g}_1'(|u|^2)|u||\nabla u|\|_{L^2}^2\\
&\lesssim_{\mathbf{N}}\||u||\nabla u|\|_{L^2}^2+\||\nabla u|\|_{L^2}^2
\end{split}
\end{equation*}
and the Gagliardo-Nirenberg inequality $\|\nabla f\|_{L^4}^4\lesssim\|\Delta f\|_{L^2}^3\|\nabla f\|_{L^2}$ that
\begin{equation}\label{0-11}
\begin{split}
&\|\Delta u^\epsilon(t)\|_{L^2}^2+\int_0^t\|\nabla\Delta u^\epsilon(r)\|_{L^2}^2\mathrm{d}r\lesssim\|\Delta u^\epsilon(0)\|_{L^2}^2\\
&\quad+\int_0^t(\|\nabla(\mathbf{g}(|u^{\epsilon}(r)|^2)u^{\epsilon}(r))\|_{L^2}^2+\|\nabla((u^{\epsilon}(r)\cdot\nabla) u^{\epsilon}(r))\|_{L^2}^2+\|\nabla((n^{\epsilon}(r)\nabla \phi)*\rho^\epsilon)\|_{L^2}^2\\
&\quad+\|\Delta G(u^{\epsilon}(r))\|_{\mathcal{L}_2(\mathbf{U};L^2)}^2)\mathrm{d}r+\sum_{j\geq1}\int_0^t(\Delta G _j(u^{\epsilon}(r)),\Delta u^{\epsilon}(r))_{L^2}\mathrm{d}W_j(r)\\
&\lesssim\|u^\epsilon(0)\|_{H^2}^2+\int_0^t(\|u^{\epsilon}(r)\|_{L^{\infty}}^2\|\nabla(\mathbf{g}(|u^{\epsilon}(r)|^2))\|_{L^2}^2+\|\mathbf{g}(|u^{\epsilon}(r)|^2)\|_{L^3}^2\|\nabla u^{\epsilon}(r)\|_{L^{6}}^2\\
&\quad+\|\nabla u^{\epsilon}(r)\|_{L^{4}}^4+\|u^{\epsilon}(r)\|_{H^{2}}^2\|\Delta u^{\epsilon}(r)\|_{L^2}^2+C_{\epsilon}\|\nabla\phi\|_{L^{\infty}}^2\|n^{\epsilon}(r)\|_{L^2}^2+1+\|u^{\epsilon}(r)\|_{H^2}^2)\mathrm{d}r\\
&\quad+\sum_{j\geq1}\int_0^t(\Delta G _j(u^{\epsilon}(r)),\Delta u^{\epsilon}(r))_{L^2}\mathrm{d}W_j(r)\\
&\lesssim_{\epsilon,\mathbf{N}}\|u^\epsilon(0)\|_{H^2}^2+\int_0^t(1+\|u^{\epsilon}(r)\|_{H^{1}}^4+\|\nabla\phi\|_{L^{\infty}}^2\|n^{\epsilon}(r)\|_{L^2}^2)\mathrm{d}r\\
&\quad+\int_0^t(1+\|u^{\epsilon}(r)\|_{H^{2}}^2+\|u^{\epsilon}(r)\|_{H^{1}}^4+\|u^{\epsilon}(r)\|_{H^{1}}^2\|u^{\epsilon}(r)\|_{H^{2}}^2)\|\Delta u^{\epsilon}(r)\|_{L^{2}}^2\mathrm{d}r\\
&\quad+\sum_{j\geq1}\left|\int_0^t(\Delta G _j(u^{\epsilon}(r)),\Delta u^{\epsilon}(r))_{L^2}\mathrm{d}W_j(r)\right|.
\end{split}
\end{equation}
Applying Gronwall's lemma to \eqref{0-11} and noting the definition of $\textbf{t}^{L}_2$,  we have
\begin{equation}\label{0-12}
\begin{split}
&\|\Delta u^\epsilon(t\wedge\textbf{t}^{L}_2)\|_{L^2}^2+\int_0^{t\wedge\textbf{t}^{L}_2}\|\nabla\Delta u^\epsilon(r)\|_{L^2}^2\mathrm{d}r\\
&\leq e^{C_{\epsilon,\mathbf{N}}\int_0^{t\wedge\textbf{t}^{L}_2}(1+\|u^{\epsilon}(r)\|_{H^{2}}^2+\|u^{\epsilon}(r)\|_{H^{1}}^4+\|u^{\epsilon}(r)\|_{H^{1}}^2\|u^{\epsilon}(r)\|_{H^{2}}^2)\mathrm{d}r}\\
&\quad\times C_{\epsilon,\mathbf{N}}\Bigl(\|u^\epsilon(0)\|_{H^2}^2+\int_0^{t\wedge\textbf{t}^{L}_2}(1+\|u^{\epsilon}(r)\|_{H^{1}}^4+\|\nabla\phi\|_{L^{\infty}}^2\|n^{\epsilon}(r)\|_{L^2}^2)\mathrm{d}r\\
&\quad+\sum_{j\geq1}\Bigl|\int_0^{t\wedge\textbf{t}^{L}_2}(\Delta G _j(u^{\epsilon}(r)),\Delta u^{\epsilon}(r))_{L^2}\mathrm{d}W_j(r)\Bigl|\Bigl)\\
&\leq C_{\epsilon,\mathbf{N},L_0,L_1,L_2,T,\phi,u_0}+ C_{\epsilon,\mathbf{N},L_0,L_1,L_2,T}\sum_{j\geq1}\Bigl|\int_0^{t\wedge\textbf{t}^{L}_2}(\Delta G _j(u^{\epsilon}(r)),\Delta u^{\epsilon}(r))_{L^2}\mathrm{d}W_j(r)\Bigl|.
\end{split}
\end{equation}
By using the BDG inequality, it follows that
\begin{equation}\label{0-13}
\begin{split}
&\mathbb{E}\bigg[C_{\epsilon,\mathbf{N},L_0,L_1,L_2,T}\sum_{j\geq1}\sup_{t\in[0,T]}\bigg|\int_0^{t\wedge\textbf{t}^{L}_2}(\Delta G _j(u^{\epsilon}(r)),\Delta u^{\epsilon}(r))_{L^2}\mathrm{d}W_j(r)\bigg|\bigg]\\
&\leq\frac{1}{2}\mathbb{E}\sup_{t\in[0,T]}\|\Delta u^\epsilon(t\wedge\textbf{t}^{L}_2)\|_{L^2}^2+ C_{\epsilon,\mathbf{N},L_0,L_1,L_2,T}\mathbb{E}\int_0^{T\wedge\textbf{t}^{L}_2}(1+\|u^{\epsilon}(t)\|_{H^2}^2)\mathrm{d}t.
\end{split}
\end{equation}
By combining \eqref{0-12}-\eqref{0-13} and using the Gronwall lemma, we infer from \eqref{0-10} that
\begin{equation}\label{0-14}
\begin{split}
&\mathbb{E}\sup_{t\in[0,T]}\|u^\epsilon(t\wedge\textbf{t}^{L}_2)\|_{H^2}^2+ \mathbb{E}\int_0^{T\wedge\textbf{t}^{L}_2}\|u^{\epsilon}(t)\|_{H^3}^2\mathrm{d}t\lesssim_{a,\epsilon,\mathbf{N},\phi,L_0,L_1,L_2,T,u_0}1.
\end{split}
\end{equation}
Next, we shall establish the $H^2$-estimate of solution $(n^{\epsilon},c^{\epsilon})$. Indeed, according to the $n^\epsilon$-equation, we have
\begin{equation}\label{0-15}
\begin{split}
&\frac{\mathrm{d}}{\mathrm{d}t}\|n^{\epsilon}\|_{H^2}^2+\|\nabla n^{\epsilon}\|_{H^2}^2\\
&\lesssim\|u^{\epsilon}\|_{H^2}^2\|n^{\epsilon}\|_{H^2}^2+C_{\epsilon}\|c^{\epsilon}\|_{L^{2}}^2\|n^{\epsilon}\|_{H^2}^2+\|n^{\epsilon}\|_{H^2}^2+C_a(\|(n^{\epsilon})^2\|_{H^2}^2+\|(n^{\epsilon})^3\|_{H^2}^2)\\
&\lesssim\|u^{\epsilon}\|_{H^2}^2\|n^{\epsilon}\|_{H^2}^2+C_{\epsilon}\|c^{\epsilon}\|_{L^{2}}^2\|n^{\epsilon}\|_{H^2}^2+C_a(1+\|n^{\epsilon}\|_{L^{\infty}}^4)\|n^{\epsilon}\|_{H^2}^2\\
&\lesssim_{a,\epsilon}(1+\|u^{\epsilon}\|_{H^2}^2+\|c^{\epsilon}\|_{L^{2}}^2+\|n^{\epsilon}\|_{H^1}^2\|n^{\epsilon}\|_{H^2}^2)\|n^{\epsilon}\|_{H^2}^2.
\end{split}
\end{equation}
In the proof of the above inequality, we have used the Gagliardo-Nirenberg inequality
\begin{equation*}
\|f\|_{L^{\infty}}\lesssim\|f\|_{L^6}^{\frac{1}{2}}\|f\|_{H^2}^{\frac{1}{2}}\lesssim\|f\|_{H^1}^{\frac{1}{2}}\|f\|_{H^2}^{\frac{1}{2}},\quad f\in H^2(\mathbb{R}^3).
\end{equation*}
Thus by applying the Gronwall lemma to \eqref{0-15} and taking the expectation, we infer from \eqref{0-10} and \eqref{0-14} that
\begin{equation}\label{0-17}
\begin{split}
&\mathbb{E}\sup_{t\in[0,T]}\|n^\epsilon(t\wedge\textbf{t}^{L}_2)\|_{H^2}^2+ \mathbb{E}\int_0^{T\wedge\textbf{t}^{L}_2}\|n^{\epsilon}(t)\|_{H^3}^2\mathrm{d}t\\
&\lesssim_{a,\epsilon,n_0}\mathbb{E}e^{\int_0^{T\wedge\textbf{t}^{L}_2}(1+\|u^{\epsilon}(t)\|_{H^2}^2+\|c^{\epsilon}(t)\|_{L^{2}}^2+\|n^{\epsilon}(t)\|_{H^1}^2\|n^{\epsilon}(t)\|_{H^2}^2)\mathrm{d}t}\\
&\lesssim_{a,\epsilon,\mathbf{N},\phi,L_0,L_1,L_2,T,\mathbf{y}_0}1.
\end{split}
\end{equation}
Similarly, it is not hard to derive that
\begin{equation*}
\mathbb{E}\sup_{t\in[0,T]}\|c^\epsilon(t\wedge\textbf{t}^{L}_2)\|_{H^2}^2+ \mathbb{E}\int_0^{T\wedge\textbf{t}^{L}_2}\|c^{\epsilon}(t)\|_{H^3}^2\mathrm{d}t\lesssim_{a,\epsilon,\mathbf{N},\phi,L_0,L_1,L_2,T,\mathbf{y}_0}1,
\end{equation*}
which combined with \eqref{0-14} and \eqref{0-17} implies that
\begin{equation}\label{0-18}
\mathbb{E}\sup_{t\in [0,T]}\|(n^\epsilon,c^\epsilon,u^\epsilon)(t\wedge\textbf{t}^{L}_2)\|_{\textbf{H}^2}^2 +\mathbb{E} \int_0^{T\wedge\textbf{t}^{L}_2}  \|(n^\epsilon,c^\epsilon,u^\epsilon)(t) \|_{\textbf{H}^3}^2  \mathrm{d} t\lesssim_{a,\epsilon,\mathbf{N},\phi,L_0,L_1,L_2,T,\mathbf{y}_0}1.
\end{equation}
 Next, we shall establish the $\textbf{H}^7$-estimate of solution $(n^{\epsilon},c^{\epsilon},u^{\epsilon})$. Define
$$
\tau^{L_3}:= \inf\{t>0;~ \sup_{r\in [0,t]}\|( n^\epsilon, c^\epsilon, u^\epsilon)(r)\|_{\textbf{H}^2}^2>L_3~ \textrm{or}~  \int_0^t \|( n^\epsilon, c^\epsilon, u^\epsilon)(r)\|_{\textbf{H}^3}^2 \mathrm{d}r>L_3\}.
$$
Then by \eqref{0-18}, we have $T\wedge\textbf{t}^{L}_3\rightarrow T\wedge\textbf{t}^{L}_2$ as $L_3\rightarrow\infty$ a.s. Applying It\^{o}'s formula to $\|\Lambda^7 u^{\epsilon}\|_{L^2}^2$ and using the Moser estimate, we have
\begin{equation}\label{0-19}
\begin{split}
&\|u^\epsilon(t)\|_{H^7}^2+\int_0^t\|\nabla u^\epsilon(r)\|_{H^7}^2\mathrm{d}r\lesssim\|u^\epsilon(0)\|_{H^7}^2\\
&\quad+\int_0^t(\|u^\epsilon(r)\|_{H^7}^2+\|\mathbf{g}(|u^{\epsilon}(r)|^2)u^{\epsilon}(r)\|_{H^7}^2+\|u^{\epsilon}(r)\otimes u^{\epsilon}(r)\|_{H^7}^2+\|(n^{\epsilon}(r)\nabla \phi)*\rho^\epsilon\|_{H^7}^2\\
&\quad+\|G(u^{\epsilon}(r))\|_{\mathcal{L}_2(\mathbf{U};H^7)}^2)\mathrm{d}r+\sum_{j\geq1}\int_0^t(\Lambda^7 G _j(u^{\epsilon}(r)),\Lambda^7  u^{\epsilon}(r))_{L^2}\mathrm{d}W_j(r)\\
&\lesssim\|u^\epsilon(0)\|_{H^7}^2+\int_0^t(\|u^{\epsilon}(r)\|_{L^{\infty}}^2\|\mathbf{g}(|u^{\epsilon}(r)|^2)\|_{H^7}^2+\|\mathbf{g}(|u^{\epsilon}(r)|^2)\|_{L^{\infty}}^2\| u^{\epsilon}(r)\|_{H^{7}}^2\\
&\quad+\|u^{\epsilon}(r)\|_{L^{\infty}}^2\|u^{\epsilon}(r)\|_{H^{7}}^2+C_{\epsilon}\|\nabla\phi\|_{L^{\infty}}^2\|n^{\epsilon}(r)\|_{L^2}^2+1+\|u^{\epsilon}(r)\|_{H^7}^2)\mathrm{d}r\\
&\quad+\sum_{j\geq1}\left|\int_0^t(\Lambda^7 G _j(u^{\epsilon}(r)),\Lambda^7  u^{\epsilon}(r))_{L^2}\mathrm{d}W_j(r)\right|.
\end{split}
\end{equation}
By using the estimate \eqref{Gu1}, recalling the definition of the stopping time $\textbf{t}^{L}_3$ and using the Sobolev embedding $H^2(\mathbb{R}^3)\hookrightarrow L^{\infty}(\mathbb{R}^3)$, it follows that
\begin{equation}\label{0-20}
\begin{split}
\int_0^{t\wedge\textbf{t}^{L}_3}\|u^{\epsilon}(r)\|_{L^{\infty}}^2\|\mathbf{g}(|u^{\epsilon}(r)|^2)\|_{H^7}^2\mathrm{d}r&\lesssim\int_0^{t\wedge\textbf{t}^{L}_3}(1+\|u^{\epsilon}(r)\|_{H^{2}}^{28})\|u^{\epsilon}(r)\|_{H^7}^2\mathrm{d}r\\
&\lesssim_{L_0,L_1,L_2,L_3}\int_0^{t\wedge\textbf{t}^{L}_3}\|u^{\epsilon}(r)\|_{H^7}^2\mathrm{d}r.
\end{split}
\end{equation}
By \eqref{con2}, we have
\begin{equation}\label{0-21}
\begin{split}
\int_0^{t\wedge\textbf{t}^{L}_3}\|\mathbf{g}(|u^{\epsilon}(r)|^2)\|_{L^{\infty}}^2\| u^{\epsilon}(r)\|_{H^{7}}^2\mathrm{d}r&\lesssim\int_0^{t\wedge\textbf{t}^{L}_3}\|u^{\epsilon}(r)\|_{L^{\infty}}^4\| u^{\epsilon}(r)\|_{H^{7}}^2\mathrm{d}r\\
&\lesssim_{L_0,L_1,L_2,L_3}\int_0^{t\wedge\textbf{t}^{L}_3}\|u^{\epsilon}(r)\|_{H^7}^2\mathrm{d}r.
\end{split}
\end{equation}
Plugging the estimates \eqref{0-20}-\eqref{0-21} into \eqref{0-19}, we derive that for all $t\in [0,T]$
\begin{equation*}
\begin{split}
&\|u^\epsilon(t)\|_{H^7}^2+\int_0^{t\wedge\textbf{t}^{L}_3}\|\nabla u^\epsilon(r)\|_{H^7}^2\mathrm{d}r\\
&\lesssim\|u^\epsilon(0)\|_{H^7}^2+C_{L_0,L_1,L_2,L_3,\epsilon,\phi,T}+C_{L_0,L_1,L_2,L_3}\int_0^{t\wedge\textbf{t}^{L}_3}\|u^\epsilon(r)\|_{H^7}^2\mathrm{d}r\\
&\quad+\sum_{j\geq1}\left|\int_0^{t\wedge\textbf{t}^{L}_3}(\Lambda^7 G _j(u^{\epsilon}(r)),\Lambda^7  u^{\epsilon}(r))_{L^2}\mathrm{d}W_j(r)\right|,
\end{split}
\end{equation*}
which combined with the Gronwall lemma implies that
\begin{equation}\label{0-22}
\begin{split}
&\|u^\epsilon(t)\|_{H^7}^2+\int_0^{t\wedge\textbf{t}^{L}_3}\|\nabla u^\epsilon(r)\|_{H^7}^2\mathrm{d}r\leq C_{L_0,L_1,L_2,L_3,T}\\
&\quad\times\Bigg(\|u^\epsilon(0)\|_{H^7}^2+C_{L_0,L_1,L_2,L_3,\epsilon,\phi,T}+\sum_{j\geq1}\Bigg|\int_0^{t\wedge\textbf{t}^{L}_3}(\Lambda^7 G _j(u^{\epsilon}(r)),\Lambda^7  u^{\epsilon}(r))_{L^2}\mathrm{d}W_j(r)\Bigg|\Bigg).
\end{split}
\end{equation}
By using the BDG inequality and Assumption  \ref{as3}, we infer from \eqref{0-22} that
\begin{equation*}
\begin{split}
&\mathbb{E}\sup_{t\in[0,T]}\|u^\epsilon(t\wedge\textbf{t}^{L}_3)\|_{H^7}^2+\mathbb{E}\int_0^{T\wedge\textbf{t}^{L}_3}\|\nabla u^{\epsilon}(t)\|_{H^7}^2\mathrm{d}t\\
&\leq C_{L_0,L_1,L_2,L_3,\epsilon,\phi,T,u_0}+ C_{L_0,L_1,L_2,L_3,T}\mathbb{E}\int_0^{T\wedge\textbf{t}^{L}_3}\|u^{\epsilon}(t)\|_{H^7}^2\mathrm{d}t,
\end{split}
\end{equation*}
which combined with Gronwall's lemma means that
\begin{equation}\label{0-23}
\begin{split}
&\mathbb{E}\sup_{t\in[0,T]}\|u^\epsilon(t\wedge\textbf{t}^{L}_3)\|_{H^7}^2+\mathbb{E}\int_0^{T\wedge\textbf{t}^{L}_3}\|\nabla u^{\epsilon}(t)\|_{H^7}^2\mathrm{d}t\lesssim_{L_0,L_1,L_2,L_3,\epsilon,\phi,T,u_0}1.
\end{split}
\end{equation}
Moreover, by using the $n^\epsilon$-equation, one can verify that
\begin{equation}\label{0-24}
\begin{split}
&\frac{\mathrm{d}}{\mathrm{d}t}\|n^{\epsilon}\|_{H^7}^2+\|\nabla n^{\epsilon}\|_{H^7}^2\\
&\lesssim\|u^{\epsilon}\|_{H^2}^2\|n^{\epsilon}\|_{H^7}^2+\|n^{\epsilon}\|_{H^2}^2\|u^{\epsilon}\|_{H^7}^2+C_{\epsilon}\|c^{\epsilon}\|_{L^{2}}^2\|n^{\epsilon}\|_{H^7}^2+C_a(\|(n^{\epsilon})^2\|_{H^7}^2+\|(n^{\epsilon})^3\|_{H^7}^2)\\
&\lesssim_{a}\|n^{\epsilon}\|_{H^2}^2\|u^{\epsilon}\|_{H^7}^2+(1+\|u^{\epsilon}\|_{H^2}^2+C_{\epsilon}\|c^{\epsilon}\|_{L^{2}}^2+\|n^{\epsilon}\|_{H^2}^4)\|n^{\epsilon}\|_{H^7}^2.
\end{split}
\end{equation}
Applying the Gronwall lemma to \eqref{0-24} and using the estimate \eqref{0-23}, we have
\begin{equation}\label{0-25}
\begin{split}
&\mathbb{E}\sup_{t\in[0,T]}\|n^\epsilon(t\wedge\textbf{t}^{L}_3)\|_{H^7}^2+\mathbb{E}\int_0^{T\wedge\textbf{t}^{L}_3}\|\nabla n^{\epsilon}(t)\|_{H^7}^2\mathrm{d}t\lesssim_{L_0,L_1,L_2,L_3,\epsilon,\phi,a,T,n_0,u_0}1.
\end{split}
\end{equation}
Similarly, we also have for $c^\epsilon$-equation
\begin{equation*}
\mathbb{E}\sup_{t\in[0,T]}\|c^\epsilon(t\wedge\textbf{t}^{L}_3)\|_{H^7}^2+\mathbb{E}\int_0^{T\wedge\textbf{t}^{L}_3}\|\nabla c^{\epsilon}(t)\|_{H^7}^2\mathrm{d}t\lesssim_{L_0,L_1,L_2,L_3,\epsilon,\phi,T,n_0,c_0,u_0}1,
\end{equation*}
which together with \eqref{0-23} and \eqref{0-25} implies that
\begin{equation}\label{0-26}
\mathbb{E}\sup_{t\in [0,T]}\|(n^\epsilon,c^\epsilon,u^\epsilon)(t\wedge\textbf{t}^{L}_3)\|_{\textbf{H}^7}^2 +\mathbb{E} \int_0^{T\wedge\textbf{t}^{L}_3}  \|(n^\epsilon,c^\epsilon,u^\epsilon) (t)\|_{\textbf{H}^8}^2  \mathrm{d} t\lesssim_{a,\epsilon,\phi,L_0,L_1,L_2,L_3,T,\mathbf{y}_0}1.
\end{equation}
Thereby, for any $T>0$ and $L_0,L_1,L_2,L_3>0$,  it follows that
\begin{equation}\label{0-27}
\widetilde{\textbf{t}^\epsilon}\geq T\wedge \textbf{t}^{L}_3 =T \wedge\tau^{L_0}\wedge\tau^{L_1}\wedge\tau^{L_2}\wedge\tau^{L_3},~\mathbb{P}\textrm{-a.s.},
\end{equation}
where $\widetilde{\textbf{t}^\epsilon}$ is the maximal existence time. Let $L_0,L_1,L_2,L_3$ in \eqref{0-27} tend to $+\infty$ one by one and using the fact $T\wedge\textbf{t}^{L}_i\rightarrow T\wedge\textbf{t}^{L}_{i-1}$ as $L_i\rightarrow+\infty$ ($i=1,2,3$), we obtain $\mathbb{P}\{\widetilde{\textbf{t}^\epsilon}=\infty\}=1$, which implies that the solution $(n^\epsilon,c^\epsilon,u^\epsilon)$ exists globally in time.

Finally, due to the fact of $n^\epsilon\in\mathcal {C}([0,T];H^{7}(\mathbb{R}^3))$,  $\mathbb{P}$-a.s., one can verify that $\Delta n ^\epsilon\in  \mathcal {C}([0,T];H^{5}(\mathbb{R}^3))$, $ \nabla\cdot \left(n^\epsilon\nabla (c^\epsilon*\rho^\epsilon)\right)\in\mathcal {C}([0,T];H^{5}(\mathbb{R}^3))$, $L(n^{\epsilon})\in\mathcal {C}([0,T];H^{7}(\mathbb{R}^3))$ and $u^\epsilon\cdot \nabla n^\epsilon\in\mathcal {C}([0,T];H^{6}(\mathbb{R}^3))$, $\mathbb{P}$-a.s. As a result, for almost all $\omega\in \Omega$, we deduce from the $n^\epsilon$-equation in \eqref{Mod-1} that $\partial_tn^{\epsilon}\in\mathcal {C}([0,T];H^{5}(\mathbb{R}^3))$, which shows that $n^{\epsilon}\in\mathcal {C}^1([0,T];H^{5}(\mathbb{R}^3))$, $\mathbb{P}$-a.s. In a similar manner, one can also show that $c^{\epsilon}\in\mathcal {C}^1([0,T];H^{5}(\mathbb{R}^3))$. The proof of Lemma \ref{lem2.5} is thus completed.
\end{proof}

\section{Proof of Theorem \ref{th1}}
In this section, we shall first derive a uniform  entropy-energy estimate and a further higher-order energy estimate (independent of $\epsilon$) for the smooth approximate solutions $\{\textbf{y}^{\epsilon}\}_{\epsilon >0} $. Then, we shall prove that the original system \eqref{CNS} admits a unique strong solution by using stochastic compactness method. Let us begin with the following entropy-energy estimate.

\subsection{A stochastic entropy-energy estimate}\label{sec3-1}

\begin{lemma} \label{lem3}
Let $(n^\epsilon,c^\epsilon,u^\epsilon)$ be the global solution of \eqref{Mod-1} constructed in Lemma \ref{lem2.5} with $\epsilon \in (0,1)$. Then for all $p\geq1$ and $T>0$ we have
\begin{equation}\label{lem3-1-1}
\begin{split}
\mathbb{E} \sup_{t\in [0,T]} \left(\mathcal{F}(n^{\epsilon},c^{\epsilon},u^{\epsilon})(t)\right)^p
+ \mathbb{E}\left(\int_0^{T}\mathcal{G}(n^{\epsilon},c^{\epsilon},u^{\epsilon})(t)\mathrm{d}t\right)^p\lesssim _{\phi,a,\mathbf{N},p,T,n_0,c_0,u_0}  1,
\end{split}
\end{equation}
where
\begin{equation*}
\begin{split}
\mathcal{F}(n^{\epsilon},c^{\epsilon},u^{\epsilon})(t)&:=\|n^{\epsilon}(t)\|_{L^1\cap L\log L}+\|\nabla\sqrt{c^{\epsilon}(t)}\|_{L^2}^2+\|u^{\epsilon}(t)\|_{L^2}^2,\\
\mathcal{G}(n^{\epsilon},c^{\epsilon},u^{\epsilon})(t)&:=\|\nabla\sqrt{n^{\epsilon}(t)+1}\|_{L^2}^2+\|\sqrt{c^{\epsilon}(t)}|D^2 \ln c^{\epsilon}(t)|\|_{L^2}^2+\|n^{\epsilon}(t)*\rho^{\epsilon}|\nabla\sqrt{c^{\epsilon}(t)}|^2\|_{L^1}\\
&\quad+\|\nabla u^{\epsilon}(t)\|_{L^2}^2+\|u^{\epsilon}(t)\|_{L^4}^4.
\end{split}
\end{equation*}
Here, $L \emph{\textrm{log}} L(\mathbb{R}^3)$ denotes the  Zygmund space which is equipped with the norm
$$
\|f\|_{L \log L}=\inf  \{k>0; \int_{\mathbb{R}^3} Z(f/ k) \mathrm{d}x \leq 1 \}
$$
with respect to $Z(t) =t\ln^+ t$ if $t\geq1$ and $Z(t) =0$ otherwise.
\end{lemma}

\begin{proof}[\emph{\textbf{Proof}}]
Recalling that $n^{\epsilon},c^{\epsilon}\in \mathcal {C}^1([0,T];H^{5}(\mathbb{R}^3))$ and $u^{\epsilon}\in\mathcal {C}([0,T];\mathbb{H}^{7})$ (cf. Lemma \ref{lem2.5}), it follows from the Sobolev embedding $H^5(\mathbb{R}^3)\hookrightarrow \mathcal{C}^3_b(\mathbb{R}^3)$ that, $\mathbb{P}$-a.s., $(n^{\epsilon}, c^{\epsilon},u^{\epsilon})$ satisfies the first two equations of \eqref{Mod-1} in the classical sense. By using the $L^2$-energy estimates for $n^{\epsilon}_-=\max\{-n^{\epsilon},0\}$ and $c^{\epsilon}_-=\max\{-c^{\epsilon},0\}$, it is standard to show that $n^{\epsilon}(t,x)\geq0$ and $c^{\epsilon}(t,x)\geq0$ (cf. \cite{hausenblas2024existence,nie2020global}), $\mathbb{P}$-a.s.

Let us claim that $\mathbb{P}$-a.s.
$$
n^{\epsilon}(t,x),~c^{\epsilon}(t,x)>0,\quad \textrm{for all}~ (0,T]\times\mathbb{R}^3.
$$
Otherwise, assume that $c^{\epsilon}(t^*,x^*)=0$ at some $(t^*,x^*)\in (0,T]\times \mathbb{R}^3$. Consider a domain $\mathcal{O}_R:=\{x\in\mathbb{R}^3;|x-x^*|<R\}$. Due to the parabolic structure of the equation $c^{\epsilon}_t+u^{\epsilon}\cdot\nabla c^{\epsilon}-\Delta c^{\epsilon}+n^{\epsilon}c^{\epsilon}=0$ and the boundedness of coefficients $n^{\epsilon}$ and $u^{\epsilon}$, it follows from the strong maximum principle (cf. \cite[Theorem 2.7]{lieberman1996second}) that $c^{\epsilon}(t,x)\equiv0$ in $(0,t^*)\times\mathcal{O}_R$, which implies that $c^{\epsilon}(0,x)=0$ in $\mathcal{O}_R$, contradicting the Assumption \ref{as2}. The positivity of $n^{\epsilon}$ follows from the similar argument.

Now let us derive the a priori uniform bound \eqref{lem3-1-1}. First, based on the structure of equation $c^{\epsilon}$, it is straightforward to verify that $\mathbb{P}$-a.s.
\begin{equation}\label{ad1}
\begin{split}
\|c^{\epsilon}(t)\|_{L^1\cap L^{\infty}}\leq \|c_0\|_{L^1\cap L^{\infty}}.
\end{split}
\end{equation}
Integrating the first equation of \eqref{Mod-1} with respect to the spatial variable $x\in \mathbb{R}^3$ and using Young's inequality, it follows that
\begin{equation*}
\begin{split}
\|n^{\epsilon}(t)\|_{L^1}+\int_0^t(a\|n^{\epsilon}(r)\|_{L^1}^1+\|n^{\epsilon}(r)\|_{L^3}^3)\mathrm{d}r&=\|n_0^{\epsilon}\|_{L^1}+(1+a)\int_0^t\|n^{\epsilon}(r)\|_{L^2}^2\mathrm{d}r\\
&\leq \frac{1}{2}\int_0^t\|n^{\epsilon}(r)\|_{L^3}^3\mathrm{d}r+C_a\int_0^t\|n^{\epsilon}(r)\|_{L^1}^1\mathrm{d}r,
\end{split}
\end{equation*}
which combined with Gronwall's lemma implies that for all $t\in[0,T]$ and $\mathbb{P}$-a.s.
\begin{equation}\label{ad2}
\begin{split}
\|n^{\epsilon}(t)\|_{L^1}+\int_0^t(\|n^{\epsilon}(r)\|_{L^2}^2+\|n^{\epsilon}(r)\|_{L^3}^3)\mathrm{d}r\lesssim_{a,T,n_0}1.
\end{split}
\end{equation}
By using the chain rule to $(n^\epsilon+1) \ln (n^\epsilon+1 )$  and integrating by parts,  we have
\begin{equation} \begin{split}\label{lem3-1}
&\frac{\mathrm{d}}{\mathrm{d}t}\|(n^\epsilon+1)\ln(n^\epsilon+1)\|_{L^1}+4\|\nabla\sqrt{n^\epsilon+1}\|^2_{L^2}\\
&\quad+a\|n^\epsilon(1+\ln(n^\epsilon+1))\|_{L^1}+\|(n^\epsilon)^3(1+\ln(n^\epsilon+1))\|_{L^1}\\
&=\int_{\mathbb{R}^3}\nabla n^\epsilon(x)\cdot\nabla (c^{\epsilon}*\rho^{\epsilon}(x))\mathrm{d}x+\int_{\mathbb{R}^3}\Delta ( c^{\epsilon}*\rho^{\epsilon}(x))\ln(n^\epsilon(x)+1)\mathrm{d}x\\
&\quad+(1+a)\int_{\mathbb{R}^3}(n^\epsilon(x))^2(1+\ln(n^\epsilon(x)+1))\mathrm{d}x.
\end{split} \end{equation}
Note that $\Delta c^{\epsilon}=2|\nabla\sqrt{c^{\epsilon}}|^2+2\sqrt{c^{\epsilon}}\Delta\sqrt{c^{\epsilon}}$, the second equation of \eqref{Mod-1} can be written as
\begin{equation*} \begin{split}
\frac{\mathrm{d}}{\mathrm{d}t}\sqrt{c^{\epsilon}}+u^{\epsilon}\cdot\nabla\sqrt{c^{\epsilon}}=(\sqrt{c^{\epsilon}})^{-1}|\nabla\sqrt{c^{\epsilon}}|^2+\Delta\sqrt{c^{\epsilon}}-\frac{1}{2}(n^{\epsilon}*\rho^{\epsilon})\sqrt{c^{\epsilon}}.
\end{split} \end{equation*}
Multiplying both sides of the above equation by $-4\Delta\sqrt{c^{\epsilon}}$ and integrating the resulted equality with respect to $x$ over $\mathbb{R}^3$, we deduce that
\begin{equation} \begin{split}\label{lem3-2}
&2 \frac{\mathrm{d}}{\mathrm{d}t}\|\nabla\sqrt{c^{\epsilon}}\|_{L^2}^2\\
&=\frac{1}{2}\int_{\mathbb{R}^3}(c^{\epsilon}(x))^{-2}|\nabla c^{\epsilon}(x)|^2\Delta c^{\epsilon}(x)\mathrm{d}x-\int_{\mathbb{R}^3}(c^{\epsilon}(x))^{-1}|\Delta c^{\epsilon}(x)|^2\mathrm{d}x\\
&\quad-\frac{1}{2}\int_{\mathbb{R}^3}(c^{\epsilon}(x))^{-1}n^{\epsilon}*\rho^{\epsilon}(x)|\nabla c^{\epsilon}(x)|^2\mathrm{d}x-\int_{\mathbb{R}^3}\nabla(n^{\epsilon}*\rho^{\epsilon}(x))\cdot\nabla c^{\epsilon}(x) \mathrm{d}x\\
&\quad+4\int_{\mathbb{R}^3}\Delta \sqrt{c^{\epsilon}(x)}(u^{\epsilon}(x)\cdot\nabla \sqrt{c^{\epsilon}(x)})\mathrm{d}x.
\end{split} \end{equation}
To get some uniform bound for $c^\epsilon$-equation, let us first show that the first two integrals on the R.H.S of \eqref{lem3-2} can be formulated as
\begin{equation} \begin{split}\label{www}
&\frac{1}{2}\int_{\mathbb{R}^3}(c^{\epsilon}(x))^{-2}|\nabla c^{\epsilon}(x)|^2\Delta c^{\epsilon}(x)\mathrm{d}x-\int_{\mathbb{R}^3}(c^{\epsilon}(x))^{-1}|\Delta c^{\epsilon}(x)|^2\mathrm{d}x\\
&=-\int_{\mathbb{R}^3}c^{\epsilon}(x)|D^2 \ln c^{\epsilon}(x)|^2\mathrm{d}x.
\end{split} \end{equation}
Indeed, by applying integration by parts, it follows that
\begin{equation} \begin{split}\label{ww1}
&\int_{\mathbb{R}^3}(c^{\epsilon}(x))^{-1}\nabla c^{\epsilon}(x)\cdot\nabla\Delta c^{\epsilon}(x)\mathrm{d}x\\
&=\int_{\mathbb{R}^3}(c^{\epsilon}(x))^{-2}|\nabla c^{\epsilon}(x)|^2\Delta c^{\epsilon}(x)\mathrm{d}x-\int_{\mathbb{R}^3}(c^{\epsilon}(x))^{-1}|\Delta c^{\epsilon}(x)|^2\mathrm{d}x.
\end{split} \end{equation}
and
\begin{equation} \begin{split}\label{ww2}
&2\int_{\mathbb{R}^3}(c^{\epsilon}(x))^{-1}\nabla c^{\epsilon}(x)\cdot\nabla\Delta c^{\epsilon}(x)\mathrm{d}x=2\int_{\mathbb{R}^3}(c^{\epsilon}(x))^{-1}\nabla c^{\epsilon}(x)\cdot \text{div}(D^2 c^{\epsilon}(x))\mathrm{d}x\\
&=-2\int_{\mathbb{R}^3}(c^{\epsilon}(x))^{-1}|D^2 c^{\epsilon}(x)|^2\mathrm{d}x+2\int_{\mathbb{R}^3}(c^{\epsilon}(x))^{-2}(D^2 c^{\epsilon}(x)\cdot\nabla c^{\epsilon}(x))\cdot\nabla c^{\epsilon}(x)\mathrm{d}x,
\end{split} \end{equation}
where $D^2 c^\epsilon$ denotes the Hessian matrix of $c^\epsilon$. Moreover, direct calculation shows that
\begin{equation} \begin{split}\label{ww3}
&\int_{\mathbb{R}^3}(c^{\epsilon}(x))^{-2}|\nabla c^{\epsilon}(x)|^2\Delta c^{\epsilon}(x)\mathrm{d}x=-\int_{\mathbb{R}^3}\nabla((c^{\epsilon}(x))^{-2}|\nabla c^{\epsilon}(x)|^2)\nabla c^{\epsilon}(x)\mathrm{d}x\\
&=2\int_{\mathbb{R}^3}(c^{\epsilon}(x))^{-3}|\nabla c^{\epsilon}(x)|^4\mathrm{d}x-2\int_{\mathbb{R}^3}(c^{\epsilon}(x))^{-2}(D^2 c^{\epsilon}(x)\cdot\nabla c^{\epsilon}(x))\cdot\nabla c^{\epsilon}(x)\mathrm{d}x.
\end{split} \end{equation}
Through the algebraic operations \eqref{ww2}+\eqref{ww3}-$2\times$\eqref{ww1}, we obtain the following expression
\begin{equation} \begin{split}\label{ww4}
&\int_{\mathbb{R}^3}(c^{\epsilon}(x))^{-2}|\nabla c^{\epsilon}(x)|^2\Delta c^{\epsilon}(x)\mathrm{d}x\\
&=-\frac{2}{3}\int_{\mathbb{R}^3}(c^{\epsilon}(x))^{-1}|D^2 c^{\epsilon}(x)|^2\mathrm{d}x+\frac{2}{3}\int_{\mathbb{R}^3}(c^{\epsilon}(x))^{-3}|\nabla c^{\epsilon}(x)|^4\mathrm{d}x+\frac{2}{3}\int_{\mathbb{R}^3}(c^{\epsilon}(x))^{-1}|\Delta c^{\epsilon}(x)|^2\mathrm{d}x.
\end{split} \end{equation}
On the other hand, a direct computation shows that
\begin{equation*} \begin{split}
&\int_{\mathbb{R}^3}c^{\epsilon}(x)|D^2 \ln c^{\epsilon}(x)|^2\mathrm{d}x\\
&=\int_{\mathbb{R}^3}(c^{\epsilon}(x))^{-1}|D^2 c^{\epsilon}(x)|^2\mathrm{d}x-2\int_{\mathbb{R}^3}(c^{\epsilon}(x))^{-2}(D^2 c^{\epsilon}(x)\cdot\nabla c^{\epsilon}(x))\cdot\nabla c^{\epsilon}(x)\mathrm{d}x\\
&\quad+\int_{\mathbb{R}^3}(c^{\epsilon}(x))^{-3}|\nabla c^{\epsilon}(x)|^4\mathrm{d}x,
\end{split} \end{equation*}
which combined with \eqref{ww3} yields that
\begin{equation} \begin{split}\label{ww5}
&\int_{\mathbb{R}^3}c^{\epsilon}(x)|D^2 \ln c^{\epsilon}(x)|^2\mathrm{d}x\\
&=\int_{\mathbb{R}^3}(c^{\epsilon}(x))^{-1}|D^2 c^{\epsilon}(x)|^2\mathrm{d}x-\int_{\mathbb{R}^3}(c^{\epsilon}(x))^{-3}|\nabla c^{\epsilon}(x)|^4\mathrm{d}x+\int_{\mathbb{R}^3}(c^{\epsilon}(x))^{-2}|\nabla c^{\epsilon}(x)|^2\Delta c^{\epsilon}(x)\mathrm{d}x.
\end{split} \end{equation}
The identity \eqref{www} follows by directly computing $\frac{3}{2}\times$\eqref{ww4}+\eqref{ww5}. As a result, the equation \eqref{lem3-2} can now be formulated as
\begin{equation} \begin{split}\label{ww6}
&2 \frac{\mathrm{d}}{\mathrm{d}t}\|\nabla\sqrt{c^{\epsilon}}\|_{L^2}^2+\int_{\mathbb{R}^3}c^{\epsilon}(x)|D^2 \ln c^{\epsilon}(x)|^2\mathrm{d}x\\
&=-\frac{1}{2}\int_{\mathbb{R}^3}(c^{\epsilon}(x))^{-1}n^{\epsilon}*\rho^{\epsilon}(x)|\nabla c^{\epsilon}(x)|^2\mathrm{d}x-\int_{\mathbb{R}^3}\nabla(n^{\epsilon}*\rho^{\epsilon}(x))\cdot\nabla c^{\epsilon}(x) \mathrm{d}x\\
&\quad+4\int_{\mathbb{R}^3}\Delta \sqrt{c^{\epsilon}(x)}(u^{\epsilon}(x)\cdot\nabla \sqrt{c^{\epsilon}(x)})\mathrm{d}x.
\end{split} \end{equation}
Adding $\eqref{ww6}$ to \eqref{lem3-1} and noting the fact of $(\nabla (n^{\epsilon}*\rho^{\epsilon}),\nabla c^{\epsilon})_{L^2}=(\nabla n^{\epsilon},\nabla (c^{\epsilon}*\rho^{\epsilon}))_{L^2}$, we have
\begin{equation} \begin{split}\label{lem3-6}
& \frac{\mathrm{d}}{\mathrm{d}t}\|(n^\epsilon+1)\ln(n^\epsilon+1)\|_{L^1}+2 \frac{\mathrm{d}}{\mathrm{d}t}\|\nabla\sqrt{c^{\epsilon}}\|_{L^2}^2+4\|\nabla\sqrt{n^{\epsilon}+1}\|^2_{L^2}\\
&\quad+\int_{\mathbb{R}^3}c^{\epsilon}(x)|D^2 \ln c^{\epsilon}(x)|^2\mathrm{d}x+\frac{1}{2}\int_{\mathbb{R}^3}(c^{\epsilon}(x))^{-1}n^{\epsilon}*\rho^{\epsilon}(x)|\nabla c^{\epsilon}(x)|^2\mathrm{d}x\\
&\quad+a\|n^\epsilon(1+\ln(n^\epsilon+1))\|_{L^1}+\|(n^\epsilon)^3(1+\ln(n^\epsilon+1))\|_{L^1}\\
&=4\int_{\mathbb{R}^3}\Delta \sqrt{c^{\epsilon}(x)}(u^{\epsilon}(x)\cdot\nabla \sqrt{c^{\epsilon}(x)})\mathrm{d}x+\int_{\mathbb{R}^3}\Delta (c^{\epsilon}*\rho^{\epsilon}(x))\ln(n^{\epsilon}(x)+1)\mathrm{d}x\\
&\quad+(1+a)\int_{\mathbb{R}^3}(n^{\epsilon}(x))^2(1+\ln(n^{\epsilon}(x)+1))\mathrm{d}x\\
&:=A_1+A_2+A_3.
\end{split} \end{equation}
In order to estimate the terms $A_1$ and $A_2$, we claim that
\begin{equation} \begin{split}\label{ww7}
&\int_{\mathbb{R}^3}(c^{\epsilon}(x))^{-3}|\nabla c^{\epsilon}(x)|^4\mathrm{d}x+\int_{\mathbb{R}^3}(c^{\epsilon}(x))^{-1}|D^2 c^{\epsilon}(x)|^2\mathrm{d}x\lesssim\int_{\mathbb{R}^3}c^{\epsilon}(x)|D^2 \ln c^{\epsilon}(x)|^2\mathrm{d}x.
\end{split} \end{equation}
Indeed, by integration by parts  and H\"{o}lder's inequality, it follows that
\begin{equation*} \begin{split}
&\int_{\mathbb{R}^3}(c^{\epsilon}(x))^{-3}|\nabla c^{\epsilon}(x)|^4\mathrm{d}x=\int_{\mathbb{R}^3}|\nabla \ln c^{\epsilon}(x)|^2\nabla \ln c^{\epsilon}(x)\cdot \nabla c^{\epsilon}(x)\mathrm{d}x\\
&=-\int_{\mathbb{R}^3}c^{\epsilon}(x)\nabla|\nabla \ln c^{\epsilon}(x)|^2\cdot\nabla \ln c^{\epsilon}(x) \mathrm{d}x-\int_{\mathbb{R}^3}c^{\epsilon}(x)|\nabla \ln c^{\epsilon}(x)|^2\Delta \ln c^{\epsilon}(x) \mathrm{d}x\\
&=-2\int_{\mathbb{R}^3}(c^{\epsilon}(x))^{-1}(D^2\ln c^{\epsilon}(x)\cdot\nabla c^{\epsilon}(x))\cdot\nabla c^{\epsilon}(x) \mathrm{d}x-\int_{\mathbb{R}^3}(c^{\epsilon}(x))^{-1}|\nabla c^{\epsilon}(x)|^2\Delta \ln c^{\epsilon}(x) \mathrm{d}x\\
&\leq2\|(c^{\epsilon})^{\frac{1}{2}}|D^2\ln c^{\epsilon}|\|_{L^2}\|(c^{\epsilon})^{-\frac{3}{2}}|\nabla c^{\epsilon}|^2\|_{L^2}+\|(c^{\epsilon})^{\frac{1}{2}}|\Delta\ln c^{\epsilon}|\|_{L^2}\|(c^{\epsilon})^{-\frac{3}{2}}|\nabla c^{\epsilon}|^2\|_{L^2}\\
&\leq5\|(c^{\epsilon})^{\frac{1}{2}}|D^2\ln c^{\epsilon}|\|_{L^2}\|(c^{\epsilon})^{-\frac{3}{2}}|\nabla c^{\epsilon}|^2\|_{L^2},
\end{split} \end{equation*}
which implies that
\begin{equation} \begin{split}\label{ww8}
&\int_{\mathbb{R}^3}(c^{\epsilon}(x))^{-3}|\nabla c^{\epsilon}(x)|^4\mathrm{d}x
\leq25\|(c^{\epsilon})^{\frac{1}{2}}|D^2\ln c^{\epsilon}|\|_{L^2}^2.
\end{split} \end{equation}
Moreover, by using the basic inequality $(f+g)^2\geq\frac{1}{2}f^2-g^2$, it follows that
\begin{equation*} \begin{split}
c^{\epsilon}|\partial_{ij}\ln c^{\epsilon}|^2&=c^{\epsilon}|(c^{\epsilon})^{-1}\partial_{ij} c^{\epsilon}-(c^{\epsilon})^{-2}\partial_{i} c^{\epsilon}\partial_{j} c^{\epsilon}|^2\\
&\geq\frac{1}{2}c^{\epsilon}|(c^{\epsilon})^{-1}\partial_{ij} c^{\epsilon}|^2-c^{\epsilon}|(c^{\epsilon})^{-2}\partial_{i} c^{\epsilon}\partial_{j} c^{\epsilon}|^2\\
&=\frac{1}{2}(c^{\epsilon})^{-1}|\partial_{ij} c^{\epsilon}|^2-(c^{\epsilon})^{-3}|\partial_{i} c^{\epsilon}\partial_{j} c^{\epsilon}|^2,\quad 1\leq i,~j\leq3,
\end{split} \end{equation*}
which combined with \eqref{ww8} leads to
\begin{equation*} \begin{split}
\int_{\mathbb{R}^3}(c^{\epsilon}(x))^{-1}|D^2 c^{\epsilon}(x)|^2\mathrm{d}x&\lesssim\int_{\mathbb{R}^3}(c^{\epsilon}(x))^{-3}|\nabla c^{\epsilon}(x)|^4\mathrm{d}x+\int_{\mathbb{R}^3}c^{\epsilon}(x)|D^2 \ln c^{\epsilon}(x)|^2\mathrm{d}x\\
&\lesssim\int_{\mathbb{R}^3}c^{\epsilon}(x)|D^2 \ln c^{\epsilon}(x)|^2\mathrm{d}x.
\end{split} \end{equation*}
This prove the inequality \eqref{ww7}. Now, let us proceed to estimate the terms $A_1$-$A_3$ of \eqref{lem3-6}. By applying the divergence-free condition $\nabla\cdot u^{\epsilon}=0$ and Young's inequality, it follows from \eqref{ww7} that for any $\eta>0$
\begin{equation} \begin{split}\label{ww9}
A_1&=-4\int_{\mathbb{R}^3}(\nabla\sqrt{c^{\epsilon}(x)}\cdot\nabla)u^{\epsilon}(x)\cdot\nabla \sqrt{c^{\epsilon}(x)}\mathrm{d}x\\
&\leq4\int_{\mathbb{R}^3}\sqrt{c^{\epsilon}(x)}|\nabla u^{\epsilon}(x)|(c^{\epsilon}(x))^{-\frac{1}{2}}|\nabla\sqrt{c^{\epsilon}(x)}|^2\mathrm{d}x\\
&\leq\eta\int_{\mathbb{R}^3}(c^{\epsilon}(x))^{-3}|\nabla c^{\epsilon}(x)|^4\mathrm{d}x+C_{\eta}\|c_0\|_{L^{\infty}}\|\nabla u^{\epsilon}\|_{L^2}^2\\
&\leq C_1\eta\int_{\mathbb{R}^3}c^{\epsilon}(x)|D^2 \ln c^{\epsilon}(x)|^2\mathrm{d}x+C_{\eta,\|c_0\|_{L^{\infty}}}\|\nabla u^{\epsilon}\|_{L^2}^2.
\end{split} \end{equation}
Noting that
\begin{equation*} \begin{split}
\|\Delta (c^{\epsilon}*\rho^{\epsilon})\|_{L^2}\leq \|\Delta c^{\epsilon}\|_{L^2},~\Delta c^{\epsilon}=2|\nabla\sqrt{c^{\epsilon}}|^2+2\sqrt{c^{\epsilon}}\Delta\sqrt{c^{\epsilon}},
\end{split} \end{equation*}
it follows from the Young inequality and \eqref{ww7} that
\begin{equation} \begin{split}\label{lem3-8}
A_2&\leq\frac{\eta}{\|c_0\|_{L^{\infty}}}\|\Delta c^{\epsilon}\|_{L^2}^2+C_{\eta,\|c_0\|_{L^{\infty}}}\|\ln(n^{\epsilon}+1)\|_{L^2}^2\\
&\leq\frac{8\eta}{\|c_0\|_{L^{\infty}}}\int_{\mathbb{R}^3}(|\nabla\sqrt{c^{\epsilon}(x)}|^4+c^{\epsilon}(x)|\Delta\sqrt{c^{\epsilon}(x)}|^2)\mathrm{d}x+C_{\eta,\|c_0\|_{L^{\infty}}}\int_{\mathbb{R}^3}(n^{\epsilon}(x)+1)\ln(n^{\epsilon}(x)+1)\mathrm{d}x\\
&\leq8\eta\|(\sqrt{c^{\epsilon}})^{-1}|\nabla\sqrt{c^{\epsilon}}|^2\|_{L^2}^2
+8\eta\|\Delta\sqrt{c^{\epsilon}}\|_{L^2}^2+C_{\eta,\|c_0\|_{L^{\infty}}}\|(n^{\epsilon}+1)\ln(n^{\epsilon}+1)\|_{L^1}\\
&\leq C\eta\|(c^{\epsilon})^{-\frac{3}{2}}|\nabla c^{\epsilon}|^2\|_{L^2}^2+C\eta\|(c^{\epsilon})^{-\frac{1}{2}}|D^2 c^{\epsilon}|\|_{L^2}^2+C_{\eta,\|c_0\|_{L^{\infty}}}\|(n^{\epsilon}+1)\ln(n^{\epsilon}+1)\|_{L^1}\\
&\leq C_2\eta\int_{\mathbb{R}^3}c^{\epsilon}(x)|D^2 \ln c^{\epsilon}(x)|^2\mathrm{d}x+C_{\eta,\|c_0\|_{L^{\infty}}}\|(n^{\epsilon}+1)\ln(n^{\epsilon}+1)\|_{L^1}.
\end{split} \end{equation}
For $A_3$, we have
\begin{equation} \begin{split}\label{bb3}
A_3&=(1+a)\int_{\{x\in\mathbb{R}^3;n^{\epsilon} \in[0,2+2a)\}}(n^{\epsilon}(x))^2(1+\ln(n^{\epsilon}(x)+1))\mathrm{d}x\\
&\quad+(1+a)\int_{\{x\in\mathbb{R}^3;n^{\epsilon} \in [2+2a,\infty)\}}\frac{(n^{\epsilon}(x))^3(1+\ln(n^{\epsilon}(x)+1))}{n^{\epsilon}(x)}\mathrm{d}x\\
&\leq C_a\|n^{\epsilon}\|_{L^1}+\frac{1}{2}\|(n^{\epsilon})^3(1+\ln(n^{\epsilon}+1))\|_{L^1}.
\end{split} \end{equation}
Plugging the estimates \eqref{ww9}, \eqref{lem3-8} and \eqref{bb3} into \eqref{lem3-6} and choosing $\eta >0$ small enough such that $\eta\leq\frac{1}{2(C_1+C_2)}$, we have
\begin{equation} \begin{split}\label{lem3-9}
& \frac{\mathrm{d}}{\mathrm{d}t}\|(n^\epsilon+1)\ln(n^\epsilon+1)\|_{L^1}+2 \frac{\mathrm{d}}{\mathrm{d}t}\|\nabla\sqrt{c^{\epsilon}}\|_{L^2}^2+4\|\nabla\sqrt{n^{\epsilon}+1}\|^2_{L^2}\\
&\quad+\frac{1}{2}\int_{\mathbb{R}^3}c^{\epsilon}(x)|D^2 \ln c^{\epsilon}(x)|^2\mathrm{d}x+\frac{1}{2}\int_{\mathbb{R}^3}(c^{\epsilon}(x))^{-1}n^{\epsilon}*\rho^{\epsilon}(x)|\nabla c^{\epsilon}(x)|^2\mathrm{d}x\\
&\quad+a\|n^\epsilon(1+\ln(n^\epsilon+1))\|_{L^1}+\frac{1}{2}\|(n^\epsilon)^3(1+\ln(n^\epsilon+1))\|_{L^1}\\
&\leq C_{\|c_0\|_{L^{\infty}}}\|\nabla u^{\epsilon}\|_{L^2}^2+C_{\|c_0\|_{L^{\infty}}}\|(n^{\epsilon}+1)\ln(n^{\epsilon}+1)\|_{L^1}+C_a\|n^{\epsilon}\|_{L^1}.
\end{split} \end{equation}
Moreover, by applying the It\^{o} formula to $\|u^{\epsilon}\|_{L^2}^2$, we infer that
\begin{equation} \begin{split}\label{lem3-10}
&\|u^{\epsilon}(t)\|_{L^2}^2+2\int_0^t\|\nabla u^{\epsilon}(r)\|_{L^2}^2\mathrm{d}r\\
&=\|u^{\epsilon}(0)\|_{L^2}^2+2\int_0^t(n^{\epsilon}(r)\nabla\phi,u^{\epsilon}(r)*\rho^{\epsilon})_{L^2}-(\mathbf{g}(|u^{\epsilon}(r)|^2)u^{\epsilon}(r),u^{\epsilon}(r))_{L^2}\mathrm{d}r\\
&\quad+\int_0^t\|\textbf{P}  G(u^{\epsilon}(r))\|_{\mathcal{L}_2(\mathbf{U};L^2)}^2\mathrm{d}r+2\sum_{j\geq 1}\int_0^t(G_j(u^\epsilon(r)),u^\epsilon(r))_{L^2}\mathrm{d} W_j(r).
\end{split} \end{equation}
Thanks to the assumption \eqref{con1} for the tamed term, it follows that
\begin{equation*} \begin{split}
-(\mathbf{g}(|u^{\epsilon}|^2)u^{\epsilon},u^{\epsilon})_{L^2}\leq -\|u^{\epsilon}\|_{L^4}^4+C_{\mathbf{N}}\|u^{\epsilon}\|_{L^2}^2,
 \end{split} \end{equation*}
which together with \eqref{lem3-10} and Assumption \ref{as3} yields that
\begin{equation} \begin{split}\label{lem3-12}
&\|u^{\epsilon}(t)\|_{L^2}^2+\int_0^t(\|\nabla u^{\epsilon}(r)\|_{L^2}^2+\|u^{\epsilon}(r)\|_{L^4}^4)\mathrm{d}r\\
&\leq\|u^{\epsilon}(0)\|_{L^2}^2+\int_0^t(\|\nabla\phi\|_{L^{\infty}}^2\|n^{\epsilon}(r)\|_{L^2}^2+1+C_{\mathbf{N}}\|u^{\epsilon}(r)\|_{L^2}^2)\mathrm{d}r\\
&\quad+2\sum_{j\geq 1}\int_0^t(G_j(u^\epsilon(r)),u^\epsilon(r))_{L^2}\mathrm{d} W_j(r).
\end{split} \end{equation}
Multiplying the both sides of \eqref{lem3-12} by  $C_{\|c_0\|_{L^{\infty}}} $, we get from \eqref{lem3-9} and the fact of
$
\|(n^\epsilon+1) \ln (n^\epsilon+1)\|_{L^1}\asymp\|n^\epsilon\|_{L^1 \cap L \textrm{log} L}=\|n^\epsilon\|_{L^1 }+\|n^\epsilon\|_{ L \textrm{log} L}
$
that
\begin{equation*} \begin{split}
&\mathcal{F}(n^{\epsilon},c^{\epsilon},u^{\epsilon})(t)+\int_0^t\mathcal{G}(n^{\epsilon},c^{\epsilon},u^{\epsilon})(r)\mathrm{d}r\\
&\lesssim_{c_0,\mathbf{N}}\mathcal{F}(n^{\epsilon},c^{\epsilon},u^{\epsilon})(0)+\int_0^t(1+\|\nabla\phi\|_{L^{\infty}}^2\|n^{\epsilon}(r)\|_{L^2}^2)\mathrm{d}r+\int_0^t\mathcal{F}(n^{\epsilon},c^{\epsilon},u^{\epsilon})(r)\mathrm{d}r\\
&\quad+\sum_{j\geq 1}\int_0^t(G_j(u^\epsilon(r)),u^\epsilon(r))_{L^2}\mathrm{d} W_j(r).
\end{split} \end{equation*}
Noting that, by \eqref{ad2}, the second integral on the R.H.S. of the last inequality is uniformly bounded with respect to $\epsilon$. Therefore, by raising the $p$-th power on both sides of the above inequality  and applying the BDG inequality to the stochastic integral, we infer that
\begin{equation} \begin{split}\label{asdf}
&\mathbb{E} \sup_{t\in [0,T]} \left(\mathcal{F}(n^{\epsilon},c^{\epsilon},u^{\epsilon})(t)\right)^p
+ \mathbb{E}\left(\int_0^{T}\mathcal{G}(n^{\epsilon},c^{\epsilon},u^{\epsilon})(t)\mathrm{d}t\right)^p\\
&\leq C_{\mathbf{N},a,\phi,p,T,n_0,c_0,u_0}+C_{c_0,\mathbf{N},p,T}\mathbb{E}\int_0^T\left(\mathcal{F}(n^{\epsilon},c^{\epsilon},u^{\epsilon})(t)\right)^p\mathrm{d}t\\
&\quad+C_{c_0,\mathbf{N},p}\mathbb{E}\left(\int_0^T\|G(u^\epsilon(t))\|_{\mathcal{L}_2(\mathbf{U};L^2)}^2\|u^\epsilon(t)\|_{L^2}^2\mathrm{d}t\right)^{\frac{p}{2}}\\
&\leq \frac{1}{2}\mathbb{E} \sup_{t\in [0,T]} \left(\mathcal{F}(n^{\epsilon},c^{\epsilon},u^{\epsilon})(t)\right)^p+C_{a,\phi,p,T,n_0,c_0,u_0}
+C_{c_0,\mathbf{N},p,T}\mathbb{E}\int_0^T\left(\mathcal{F}(n^{\epsilon},c^{\epsilon},u^{\epsilon})(t)\right)^p\mathrm{d}t.
\end{split} \end{equation}
The desired estimate \eqref{lem3-1-1} then follows by applying the Gronwall lemma to \eqref{asdf}. The proof is thus completed.
\end{proof}

\subsection{Further energy estimates}\label{sec3-2}
Relying on the uniform bounds established in Lemma \ref{lem3}, we can further develop energy estimates for the approximate solutions $\textbf{y}^{\epsilon}$ to obtain sufficient space-time regularity, which ensures the existence and uniqueness of global solutions for system \eqref{CNS}.
\begin{lemma} \label{lem4}
Let $(n^\epsilon,c^\epsilon,u^\epsilon)$ be the global solution of \eqref{Mod-1} constructed in Lemma \ref{lem2.5} with $\epsilon \in (0,1)$. Then for all $p\geq1$, we have
\begin{align}
&\mathbb{E} \sup_{t\in [0,T]}\|u^{\epsilon}(t)\|_{H^1}^{2p}
+ \mathbb{E}\left(\int_0^{T}\|u^{\epsilon}(t)\|_{H^2}^2\mathrm{d}t\right)^p\lesssim _{\phi,a,\mathbf{N},p,T,n_0,c_0,u_0}  1,\label{lem4-1}\\
&\mathbb{E} \sup_{t\in [0,T]}\|c^{\epsilon}(t)\|_{H^1}^{2p}
+ \mathbb{E}\left(\int_0^{T}\|c^{\epsilon}(t)\|_{H^2}^2\mathrm{d}t\right)^p\lesssim _{\phi,a,\mathbf{N},p,T,n_0,c_0,u_0}  1.\label{lem4-2}
\end{align}
Moreover, there exist some positive constants $\tilde{C_1},\cdots,\tilde{C_9}$ independent of $\epsilon$ such that, for all $t\in[0,T]$ and $\mathbb{P}$-a.s.,
\begin{align}
&\|c^{\epsilon}(t)\|_{H^2}^2+\int_0^t\|c^{\epsilon}(r)\|_{H^3}^2\mathrm{d}r\cr
&\leq \tilde{C}_{1}\exp\left(\tilde{C}_2\exp\left(\tilde{C}_{3}\exp\left(\tilde{C}_4(\sup_{r\in[0,t]}\|u^{\epsilon}(r)\|_{H^1}^2+\int_0^t\|u^{\epsilon}(r)\|_{H^{2}}^2\mathrm{d}r)\right)\right)\right),\label{lem4-4}\\
&\|n^{\epsilon}(t)\|_{H^1}^2+\int_0^t\|n^{\epsilon}(r)\|_{H^2}^2\mathrm{d}r\cr
&\leq \tilde{C}_{5}\exp\left(\tilde{C}_6\exp\left(\tilde{C}_7\exp\left(\tilde{C}_{8}\exp\left(\tilde{C}_9(\sup_{r\in[0,t]}\|u^{\epsilon}(r)\|_{H^1}^2+\int_0^t\|u^{\epsilon}(r)\|_{H^{2}}^2\mathrm{d}r)\right)\right)\right)\right).\label{lem4-5}
\end{align}
\end{lemma}

\begin{proof}[\emph{\textbf{Proof}}] Applying It\^{o}'s formula to $\|\nabla u^{\epsilon}\|_{L^2}^2$ and using  from \eqref{0-6}, we infer that
\begin{equation*}
\begin{split}
&\|\nabla u^{\epsilon}(t)\|_{L^2}^2+2\int_0^t\|\Delta u^{\epsilon}(r)\|_{L^2}^2\mathrm{d}r\\
&=\|\nabla u^{\epsilon}(0)\|_{L^2}^2+2\int_0^t[(\mathbf{g}(|u^{\epsilon}(r)|^2)u^{\epsilon}(r),\Delta u^{\epsilon}(r))_{L^2}+((u^{\epsilon}(r)\cdot \nabla) u^{\epsilon}(r),\Delta u^{\epsilon}(r))_{L^2}\\
&\quad-((n^{\epsilon}(r)\nabla \phi)*\rho^\epsilon,\Delta u^{\epsilon}(r))_{L^2}]\mathrm{d}r+\int_0^t\|\nabla\textbf{P}  G(u^{\epsilon}(r))\|_{\mathcal{L}_2(\mathbf{U};L^2)}^2\mathrm{d}r\\
&\quad+2\sum_{j\geq 1}\int_0^t(\nabla G_j(u^\epsilon(r)),\nabla u^\epsilon(r))_{L^2}\mathrm{d} W_j(r)\\
&\leq\|\nabla u^{\epsilon}(0)\|_{L^2}^2+\int_0^t\|\Delta u^{\epsilon}(r)\|_{L^2}^2\mathrm{d}r+C_{\mathbf{N}}\int_0^t(1+\|\nabla\phi\|_{L^{\infty}}^2\|n^{\epsilon}(r)\|_{L^2}^2+\|u^{\epsilon}(r)\|_{L^2}^2+\|\nabla u^{\epsilon}(r)\|_{L^2}^2)\mathrm{d}r\\
&\quad+2\sum_{j\geq 1}\int_0^t(\nabla G_j(u^\epsilon(r)),\nabla u^\epsilon(r))_{L^2}\mathrm{d} W_j(r).
\end{split}
\end{equation*}
Raising the $p$-th power on both sides of the above inequality, and using BDG's inequality, we see from \eqref{lem3-1-1} and \eqref{ad2} that
\begin{equation}\label{00-1}
\begin{split}
&\mathbb{E} \sup_{t\in [0,T]}\|\nabla u^{\epsilon}(t)\|_{L^2}^{2p}
+ \mathbb{E}\left(\int_0^{T}\|\Delta u^{\epsilon}(t)\|_{L^2}^2\mathrm{d}t\right)^p\\
&\leq\mathbb{E}\|\nabla u^{\epsilon}(0)\|_{L^2}^{2p}+C_{\mathbf{N},p}\mathbb{E}\left(\int_0^{T}(1+\|\nabla\phi\|_{L^{\infty}}^2\|n^{\epsilon}(t)\|_{L^2}^2+\|u^{\epsilon}(t)\|_{L^2}^2)\mathrm{d}t\right)^p\\
&\quad+C_{\mathbf{N},p}\mathbb{E}\left(\int_0^{T}\|\nabla u^{\epsilon}(t)\|_{L^2}^2\mathrm{d}t\right)^p+C\sum_{j\geq 1}\mathbb{E} \sup_{t\in [0,T]}\left(\left|\int_0^t(\nabla G_j(u^\epsilon(r)),\nabla u^\epsilon(r))_{L^2}\mathrm{d} W_j(r)\right|\right)^p\\
&\leq C_{\phi,a,\mathbf{N},p,T,n_0,c_0,u_0}+C_{\mathbf{N},p}\mathbb{E}\left(\int_0^{T}\|\nabla u^{\epsilon}(t)\|_{L^2}^{2p}\mathrm{d}t\right)\\
&\quad+C_p\mathbb{E}\left(\int_0^T\|\nabla G(u^\epsilon(t))\|_{\mathcal{L}_2(\mathbf{U};L^2)}^2\|\nabla u^\epsilon(t)\|_{L^2}^2\mathrm{d}t\right)^{\frac{p}{2}}\\
&\leq C_{\phi,a,\mathbf{N},p,T,n_0,c_0,u_0}+\frac{1}{2}\mathbb{E} \sup_{t\in [0,T]}\|\nabla u^{\epsilon}(t)\|_{L^2}^{2p}+C_{\mathbf{N},p}\int_0^{T}\mathbb{E}\sup_{r\in[0,t]}\|\nabla u^{\epsilon}(r)\|_{L^2}^{2p}\mathrm{d}t.
\end{split}
\end{equation}
Thus, by applying Gronwall's lemma to \eqref{00-1} and using the estimate \eqref{lem3-1-1}, we obtain \eqref{lem4-1}.

By using inequality \eqref{ad1} and \eqref{ww7}, it follows that
\begin{equation*} \begin{split}
&\sup_{r\in[0,t]}\|\nabla c^{\epsilon}(r)\|_{L^2}^{2}\lesssim \sup_{r\in[0,t]}\|\sqrt{c^{\epsilon}(r)}\|_{L^{\infty}} ^{2} \sup_{r\in[0,t]}\|\nabla\sqrt{c^{\epsilon}(r)}\|_{L^2}^{2}\lesssim\|c_0\|_{L^{\infty}}\sup_{r\in[0,t]}\|\nabla\sqrt{c^{\epsilon}(r)}\|_{L^2}^{2},
\end{split} \end{equation*}
and
\begin{equation*} \begin{split}
\int_0^{t}\|\Delta c^{\epsilon}(r)\|_{L^2}^2\mathrm{d}r&\lesssim\|c_0\|_{L^{\infty}}\int_0^{t}(\|(\sqrt{c^{\epsilon}(r)})^{-1}|\nabla\sqrt{c^{\epsilon}(r)}|^2\|_{{L^2}}^2+\|\Delta\sqrt{c^{\epsilon}(r)}\|_{L^2}^2)\mathrm{d}r\\
&\lesssim\|c_0\|_{L^{\infty}}\int_0^{t}\|\sqrt{c^{\epsilon}(r)}|D^2 \ln c^{\epsilon}(r)|\|_{L^2}^2\mathrm{d}r.
\end{split} \end{equation*}
By   \eqref{lem3-1-1} and \eqref{ad1}, we get from the above two estimates that
\begin{equation*}
\begin{split}
&\mathbb{E} \sup_{t\in [0,T]}\|c^{\epsilon}(t)\|_{H^1}^{2p}+ \mathbb{E}\left(\int_0^{T}\|c^{\epsilon}(t)\|_{H^2}^2\mathrm{d}t\right)^p\\
&\lesssim_{c_0,p}1+\mathbb{E} \sup_{t\in [0,T]}\|\nabla\sqrt{c^{\epsilon}(t)}\|_{L^2}^{2p}+\mathbb{E}\left(\int_0^{T}\|\sqrt{c^{\epsilon}(t)}|D^2 \ln c^{\epsilon}(t)|\|_{L^2}^2\mathrm{d}t\right)^p\\
&\lesssim _{\phi,a,\mathbf{N},p,T,n_0,c_0,u_0}  1,
\end{split}
\end{equation*}
which implies \eqref{lem4-2}.

Next, we shall deal with the estimate for $\|\nabla c^{\epsilon}\|_{L^3}$. By means of the $c^\epsilon$-equation, we obtain that
\begin{equation*} \begin{split}
&\frac{1}{3}\frac{\mathrm{d}}{\mathrm{d}t}\|\partial_ic^{\epsilon}\|_{L^3}^3+2\int_{\mathbb{R}^3}|\partial_ic^{\epsilon}(x)||\nabla\partial_ic^{\epsilon}(x)|^2\mathrm{d}x\\
&=2\int_{\mathbb{R}^3}(n^{\epsilon}*\rho^{\epsilon})c^{\epsilon}(x)|\partial_ic^{\epsilon}(x)|\partial_{ii}c^{\epsilon}(x)\mathrm{d}x+2\int_{\mathbb{R}^3}(u^{\epsilon}(x)\cdot\nabla c^{\epsilon}(x))|\partial_i c^{\epsilon}(x)|\partial_{ii}c^{\epsilon}(x)\mathrm{d}x\\
&\leq\|\sqrt{|\partial_i c^{\epsilon}|}|\nabla\partial_{i}c^{\epsilon}|\|_{L^2}^2+C\|c^{\epsilon}\|_{L^{\infty}}^2\|(n^{\epsilon}*\rho^{\epsilon})\sqrt{|\partial_i c^{\epsilon}|}\|_{L^2}^2+C\||u^{\epsilon}||\nabla c|\sqrt{|\partial_i c^{\epsilon}|}\|_{L^2}^2\\
&\leq\|\sqrt{|\partial_i c^{\epsilon}|}|\nabla\partial_{i}c^{\epsilon}|\|_{L^2}^2+C_{c_0}\|n^{\epsilon}\|_{L^3}^2\|\partial_i c^{\epsilon}\|_{L^3}+C\|u^{\epsilon}\|_{L^6}^2\|\nabla c^{\epsilon}\|_{L^6}^2\|\partial_i c^{\epsilon}\|_{L^3},
\end{split} \end{equation*}
which together with the embedding $H^1(\mathbb{R}^3)\hookrightarrow L^6(\mathbb{R}^3)$, the Young inequality and the Sobolev inequality implies that
\begin{equation} \begin{split}\label{00-2}
&\frac{\mathrm{d}}{\mathrm{d}t}\|\nabla c^{\epsilon}\|_{L^3}^3+\|\nabla|\nabla c^{\epsilon}|^{\frac{3}{2}}\|_{L^2}^2\\
&\leq C_{c_0}\|n^{\epsilon}\|_{L^3}^2+C\|u^{\epsilon}\|_{H^1}^2\|\Delta c^{\epsilon}\|_{L^2}^2+(C_{c_0}\|n^{\epsilon}\|_{L^3}^2+C\|u^{\epsilon}\|_{H^1}^2\|\Delta c^{\epsilon}\|_{L^2}^2)\|\nabla c^{\epsilon}\|_{L^3}^3\\
&\lesssim_{c_0}1+\|n^{\epsilon}\|_{L^3}^3+\|u^{\epsilon}\|_{H^1}^2\|\Delta c^{\epsilon}\|_{L^2}^2+(1+\|n^{\epsilon}\|_{L^3}^3+\|u^{\epsilon}\|_{H^1}^2\|\Delta c^{\epsilon}\|_{L^2}^2)\|\nabla c^{\epsilon}\|_{L^3}^3.
 \end{split} \end{equation}
Applying the Gronwall lemma to \eqref{00-2} and using the basic inequality $x\leq e^{x}$ for any $x\geq0$, we get
\begin{equation} \begin{split}\label{00-3}
&\|\nabla c^{\epsilon}(t)\|_{L^3}^3+\int_0^t\|\nabla|\nabla c^{\epsilon}(r)|^{\frac{3}{2}}\|_{L^2}^2\mathrm{d}r\\
&\leq C_{c_0}\left(\|\nabla c^{\epsilon}(0)\|_{L^3}^3+\int_0^t(1+\|n^{\epsilon}(r)\|_{L^3}^3+\|u^{\epsilon}(r)\|_{H^1}^2\|\Delta c^{\epsilon}(r)\|_{L^2}^2)\mathrm{d}r\right)\\
&\quad\times\exp\left(C_{c_0}\int_0^t(1+\|n^{\epsilon}(r)\|_{L^3}^3+\|u^{\epsilon}(r)\|_{H^1}^2\|\Delta c^{\epsilon}(r)\|_{L^2}^2)\mathrm{d}r\right)\\
&\leq\exp(C_{c_0}\|\nabla c^{\epsilon}(0)\|_{L^3}^3)\exp\left(C_{c_0}\int_0^t(1+\|n^{\epsilon}(r)\|_{L^3}^3+\|u^{\epsilon}(r)\|_{H^1}^2\|\Delta c^{\epsilon}(r)\|_{L^2}^2)\mathrm{d}r\right)\\
&\leq C_{c_0}\exp\left(C_{c_0}\int_0^T(1+\|n^{\epsilon}(r)\|_{L^3}^3)\mathrm{d}r\right)\exp\left(C_{c_0}\int_0^t\|u^{\epsilon}(r)\|_{H^1}^2\|\Delta c^{\epsilon}(r)\|_{L^2}^2\mathrm{d}r\right)\\
&\leq C_{a,T,n_0,c_0}\exp\left(C_{c_0}\int_0^t\|u^{\epsilon}(r)\|_{H^1}^2\|\Delta c^{\epsilon}(r)\|_{L^2}^2\mathrm{d}r\right),
\end{split} \end{equation}
where the last inequality used the fact of \eqref{ad2}. Similarly, by using the $c^\epsilon$-equation again, we have
\begin{equation*} \begin{split}
\frac{\mathrm{d}}{\mathrm{d}t}\|\nabla c^{\epsilon}\|_{L^2}^2+\|\Delta c^{\epsilon}\|_{L^2}^2&\leq\|u^{\epsilon}\cdot\nabla c^{\epsilon}\|_{L^2}^2+\|c^\epsilon(n^\epsilon*\rho^\epsilon)\|_{L^2}^2\\
&\leq\|u^{\epsilon}\|_{L^{\infty}}^2\|\nabla c^{\epsilon}\|_{L^2}^2+\|c^{\epsilon}\|_{L^{\infty}}^2\|n^{\epsilon}\|_{L^2}^2\leq C\|c_0\|_{L^{\infty}}^2\|n^{\epsilon}\|_{L^2}^2+C\|u^{\epsilon}\|_{H^{2}}^2\|\nabla c^{\epsilon}\|_{L^2}^2,
\end{split} \end{equation*}
which together with Gronwall's lemma means that
\begin{equation} \begin{split}\label{00-4}
&\|\nabla c^{\epsilon}(t)\|_{L^2}^2+\int_0^t\|\Delta c^{\epsilon}(r)\|_{L^2}^2\mathrm{d}r\leq C_{a,T,n_0,c_0}\exp\left(C\int_0^t\|u^{\epsilon}(r)\|_{H^{2}}^2\mathrm{d}r\right).
\end{split} \end{equation}
Thus, we get from \eqref{00-3} and \eqref{00-4} that
\begin{equation} \begin{split}\label{00-5}
&\|\nabla c^{\epsilon}(t)\|_{L^3}^3+\int_0^t\|\nabla|\nabla c^{\epsilon}(r)|^{\frac{3}{2}}\|_{L^2}^2\mathrm{d}r\\
&\leq C_{a,T,n_0,c_0}\exp\left(C_{c_0}\sup_{r\in[0,t]}\|u^{\epsilon}(r)\|_{H^1}^2\int_0^t\|\Delta c^{\epsilon}(r)\|_{L^2}^2\mathrm{d}r\right)\\
&\leq C_{a,T,n_0,c_0}\exp\left(C_{c_0}\sup_{r\in[0,t]}\|u^{\epsilon}(r)\|_{H^1}^2C_{a,T,n_0,c_0}\exp\left(C\int_0^t\|u^{\epsilon}(r)\|_{H^{2}}^2\mathrm{d}r\right)\right)\\
&\leq C_{a,T,n_0,c_0}\exp\left(C_{a,T,n_0,c_0}\exp\left(C(\sup_{r\in[0,t]}\|u^{\epsilon}(r)\|_{H^1}^2+\int_0^t\|u^{\epsilon}(r)\|_{H^{2}}^2\mathrm{d}r)\right)\right).
\end{split} \end{equation}

Next, we shall show the estimate for $\|n^{\epsilon}\|_{L^2}$. Taking the $L^2$-inner product of the first equation of \eqref{Mod-1} with $n^{\epsilon}$ and using the interpolation inequality as well as the Sobolev inequality, we have
\begin{equation*} \begin{split}
&\frac{1}{2}\frac{\mathrm{d}}{\mathrm{d}t}\|n^{\epsilon}\|_{L^2}^2+\|\nabla n^{\epsilon}\|_{L^2}^2+a\|n^{\epsilon}\|_{L^2}^2+\|n^{\epsilon}\|_{L^4}^4\\
&=(n^{\epsilon}\nabla (c^{\epsilon}*\rho^{\epsilon}),\nabla n^{\epsilon})_{L^2}+(a+1)\|n^{\epsilon}\|_{L^3}^3\\
&\leq\|n^{\epsilon}\|_{L^3}\|\nabla c^{\epsilon}\|_{L^6}\|\nabla n^{\epsilon}\|_{L^2}+(a+1)\|n^{\epsilon}\|_{L^3}^3\\
&\leq C\|n^{\epsilon}\|_{L^2}^{\frac{1}{2}}\||\nabla c^{\epsilon}|^{\frac{3}{2}}\|_{L^2}^{\frac{1}{6}}\||\nabla c^{\epsilon}|^{\frac{3}{2}}\|_{L^6}^{\frac{1}{2}}\|\nabla n^{\epsilon}\|_{L^2}^{\frac{3}{2}}+(a+1)\|n^{\epsilon}\|_{L^3}^3\\
&\leq\frac{1}{2}\|\nabla n^{\epsilon}\|_{L^2}^2+C\|\nabla c^{\epsilon}\|_{L^3}\|\nabla|\nabla c^{\epsilon}|^{\frac{3}{2}}\|_{L^2}^2\|n^{\epsilon}\|_{L^2}^2+(a+1)\|n^{\epsilon}\|_{L^3}^3.
\end{split} \end{equation*}
By using \eqref{ad2}, \eqref{00-5} and Gronwall's lemma, it follows that
\begin{equation} \begin{split}\label{00-6}
&\|n^{\epsilon}(t)\|_{L^2}^2+\int_0^t(\|n^{\epsilon}(r)\|_{H^1}^2+\|n^{\epsilon}(r)\|_{L^4}^4)\mathrm{d}r\\
&\leq C_{a,T,n_0}\exp\left(C\int_0^t\|\nabla c^{\epsilon}(r)\|_{L^3}\|\nabla|\nabla c^{\epsilon}(r)|^{\frac{3}{2}}\|_{L^2}^2\mathrm{d}r\right)\\
&\leq C_{a,T,n_0}\exp\left(C(1+\sup_{r\in[0,t]}\|\nabla c^{\epsilon}(r)\|_{L^3}^3)\int_0^t\|\nabla|\nabla c^{\epsilon}(r)|^{\frac{3}{2}}\|_{L^2}^2\mathrm{d}r\right)\\
&\leq C_{a,T,n_0,c_0}\exp\left(C\exp\left(C_{a,T,n_0,c_0}\exp\left(C(\sup_{r\in[0,t]}\|u^{\epsilon}(r)\|_{H^1}^2+\int_0^t\|u^{\epsilon}(r)\|_{H^{2}}^2\mathrm{d}r)\right)\right)\right).
\end{split}\end{equation}

Next, let us deal with the estimate for $\|\Delta c^{\epsilon}\|_{L^2}$. Multiplying the second equation of \eqref{Mod-1} by $-\Delta^2 c^{\epsilon}$ and integrating in space, we have
\begin{equation*} \begin{split}
&\frac{1}{2}\frac{\mathrm{d}}{\mathrm{d}t}\|\Delta c^{\epsilon}\|_{L^2}^2+\|\nabla\Delta c^{\epsilon}\|_{L^2}^2=(\nabla(u^{\epsilon}\cdot\nabla c^{\epsilon}),\nabla\Delta c^{\epsilon})_{L^2}+(\nabla(c^{\epsilon}n^{\epsilon}*\rho^{\epsilon}),\nabla\Delta c^{\epsilon})_{L^2}\\
&\leq\frac{1}{2}\|\nabla\Delta c^{\epsilon}\|_{L^2}^2+\|\nabla u^{\epsilon}\|_{L^3}^2\|\nabla c^{\epsilon}\|_{L^6}^2+\|u^{\epsilon}\|_{L^{\infty}}^2\|D^2 c^{\epsilon}\|_{L^2}^2+\|\nabla c^{\epsilon}\|_{L^6}^2\|n^{\epsilon}\|_{L^3}^2+\|c^{\epsilon}\|_{L^{\infty}}^2\|\nabla n^{\epsilon}\|_{L^2}^2\\
&\leq\frac{1}{2}\|\nabla\Delta c^{\epsilon}\|_{L^2}^2+C_{c_0}\|\nabla n^{\epsilon}\|_{L^2}^2+C(1+\|u^{\epsilon}\|_{H^2}^2+\|n^{\epsilon}\|_{L^3}^3)\|\Delta c^{\epsilon}\|_{L^2}^2.
\end{split} \end{equation*}
In terms of \eqref{ad2} and \eqref{00-6}, by using the Gronwall lemma and the basic inequality $1+e^{f+g}\lesssim e^f+e^{f+g}\lesssim e^{f+g}$ for any $f,~g\geq0$, we further have
\begin{equation*} \begin{split}
&\|\Delta c^{\epsilon}(t)\|_{L^2}^2+\int_0^t\|\nabla\Delta c^{\epsilon}(r)\|_{H^1}^2\mathrm{d}r\\
&\leq C_{c_0}\left(1+\int_0^t\|\nabla n^{\epsilon}(r)\|_{L^2}^2\mathrm{d}r\right)\exp\left(C\int_0^t(1+\|u^{\epsilon}(r)\|_{H^2}^2+\|n^{\epsilon}(r)\|_{L^3}^3)\mathrm{d}r\right)\\
&\leq C_{a,T,n_0,c_0}\exp\left(C\int_0^t\|u^{\epsilon}(r)\|_{H^2}^2\mathrm{d}r\right)\\
&\quad\times\exp\left(C\exp\left(C_{a,T,n_0,c_0}\exp\left(C(\sup_{r\in[0,t]}\|u^{\epsilon}(r)\|_{H^1}^2+\int_0^t\|u^{\epsilon}(r)\|_{H^{2}}^2\mathrm{d}r)\right)\right)\right)\\
&\leq C_{a,T,n_0,c_0}\exp\left(C\exp\left(C_{a,T,n_0,c_0}\exp\left(C(\sup_{r\in[0,t]}\|u^{\epsilon}(r)\|_{H^1}^2+\int_0^t\|u^{\epsilon}(r)\|_{H^{2}}^2\mathrm{d}r)\right)\right)\right),
\end{split} \end{equation*}
which together with \eqref{ad1} and \eqref{00-4} implies \eqref{lem4-4}.

Finally, we shall deal with the estimate for $\|\nabla n^{\epsilon}\|_{L^2}$. Multiplying the first equation of \eqref{Mod-1} by $-\Delta n^{\epsilon}$ and integrating in space, we infer from the Sobolev inequality that
\begin{equation*} \begin{split}
&\frac{1}{2}\frac{\mathrm{d}}{\mathrm{d}t}\|\nabla n^{\epsilon}\|_{L^2}^2+\|\Delta n^{\epsilon}\|_{L^2}^2+a\|\nabla n^{\epsilon}\|_{L^2}^2+3\|n^{\epsilon}|\nabla n^{\epsilon}|\|_{L^2}^2\\
&=(u^{\epsilon}\cdot\nabla n^{\epsilon},\Delta n^{\epsilon})_{L^2}+(\nabla\cdot(n^{\epsilon}\nabla (c^{\epsilon}*\rho^{\epsilon})),\Delta n^{\epsilon})_{L^2}+2(a+1)(n^{\epsilon}\nabla n^{\epsilon},\nabla n^{\epsilon})_{L^2}\\
&\leq\frac{1}{2}\|\Delta n^{\epsilon}\|_{L^2}^2+\|n^{\epsilon}|\nabla n^{\epsilon}|\|_{L^2}^2+C\|u^{\epsilon}\cdot\nabla n^{\epsilon}\|_{L^2}^2+C\|\nabla\cdot(n^{\epsilon}\nabla (c^{\epsilon}*\rho^{\epsilon}))\|_{L^2}^2+C_a\|\nabla n^{\epsilon}\|_{L^2}^2\\
&\leq\frac{1}{2}\|\Delta n^{\epsilon}\|_{L^2}^2+\|n^{\epsilon}|\nabla n^{\epsilon}|\|_{L^2}^2+C_a(1+\|u^{\epsilon}\|_{L^{\infty}}^2+\|\nabla c^{\epsilon}\|_{L^{\infty}}^2)\|\nabla n^{\epsilon}\|_{L^2}^2+C\|n^{\epsilon}\|_{L^6}^2\|\Delta c^{\epsilon}\|_{L^3}^2\\
&\leq\frac{1}{2}\|\Delta n^{\epsilon}\|_{L^2}^2+\|n^{\epsilon}|\nabla n^{\epsilon}|\|_{L^2}^2+C_a(1+\|u^{\epsilon}\|_{H^{2}}^2+\|c^{\epsilon}\|_{H^{3}}^2)\|\nabla n^{\epsilon}\|_{L^2}^2.
\end{split} \end{equation*}
By applying Gronwall's lemma to the above inequality and using the estimate \eqref{lem4-4}, we get
\begin{equation*} \begin{split}
&\|\nabla n^{\epsilon}(t)\|_{L^2}^2+\int_0^t\|\Delta n^{\epsilon}(r)\|_{L^2}^2\mathrm{d}r\\
&\leq C_{n_0}\exp\left(C_a\int_0^t(1+\|u^{\epsilon}(r)\|_{H^2}^2+\|c^{\epsilon}(r)\|_{H^3}^2)\mathrm{d}r\right)\\
&\leq C_{a,T,n_0}\exp\left(C_a\int_0^t\|u^{\epsilon}(r)\|_{H^2}^2\mathrm{d}r\right)\exp\left(C_a\int_0^t\|c^{\epsilon}(r)\|_{H^3}^2\mathrm{d}r\right)\\
&\leq C_{a,T,n_0,c_0}\exp\left(C_a\exp\left(C\exp\left(C_{a,T,n_0,c_0}\exp\left(C(\sup_{r\in[0,t]}\|u^{\epsilon}(r)\|_{H^1}^2+\int_0^t\|u^{\epsilon}(r)\|_{H^{2}}^2\mathrm{d}r)\right)\right)\right)\right),
\end{split} \end{equation*}
which combined with \eqref{00-6} implies \eqref{lem4-5}. The proof is thus completed.
\end{proof}

\begin{corollary} \label{1cor}
There exist positive constants $\tilde{C}_{10},\cdots,\tilde{C}_{19}$ independent of $\epsilon$ such that $\mathbb{P}$-a.s.
\begin{align}
&\|n^{\epsilon}\|_{W^{1,2}(0,T;L^2(\mathbb{R}^3))}^2\cr
&\leq\tilde{C}_{10}\exp\left(\tilde{C}_{11}\exp\left(\tilde{C}_{12}\exp\left(\tilde{C}_{13}\exp\left(\tilde{C}_{14}(\sup_{t\in[0,T]}\|u^{\epsilon}(t)\|_{H^1}^2+\int_0^T\|u^{\epsilon}(t)\|_{H^{2}}^2\mathrm{d}t)\right)\right)\right)\right),\label{cor.1}\\
&\|c^{\epsilon}\|_{W^{1,2}(0,T;H^1(\mathbb{R}^3))}^2\cr
&\leq \tilde{C}_{15}\exp\left(\tilde{C}_{16}\exp\left(\tilde{C}_{17}\exp\left(\tilde{C}_{18}\exp\left(\tilde{C}_{19}(\sup_{t\in[0,T]}\|u^{\epsilon}(t)\|_{H^1}^2+\int_0^T\|u^{\epsilon}(t)\|_{H^{2}}^2\mathrm{d}t)\right)\right)\right)\right).\label{cor.2}
\end{align}
\end{corollary}

\begin{proof}[\emph{\textbf{Proof}}] By direct calculation, we deduce from the estimates \eqref{lem4-4} and \eqref{lem4-5} that
\begin{equation*}
\begin{split}
&\int_0^T\|n^{\epsilon}(r)\|_{L^2}^2+\left\|\partial_tn^{\epsilon}(r)\right\|_{L^{2}}^2\mathrm{d}r\\
&\leq C_a\int_0^T(\|u^\epsilon(t)\|_{L^{\infty}}^2\|\nabla n^{\epsilon}(t)\|_{L^2}^2+\|\Delta n ^\epsilon(t)\|_{L^2}^2+\|\nabla c^\epsilon(t)\|_{L^{\infty}}^2\|\nabla n^{\epsilon}(t)\|_{L^2}^2\\
&\quad+\|n^\epsilon(t)\|_{L^{\infty}}^2\|\Delta c^{\epsilon}(t)\|_{L^2}^2+\|n^{\epsilon}(t)\|_{L^2}^2+\|n^{\epsilon}(t)\|_{L^4}^4+\|n^{\epsilon}(t)\|_{L^6}^6)\mathrm{d}t\\
&\leq C_{a,T}\Bigl(\sup_{t\in[0,T]}\|n^{\epsilon}(t)\|_{H^1}^2\int_0^T\|u^\epsilon(t)\|_{H^{2}}^2\mathrm{d}t+\int_0^T\|n ^\epsilon(t)\|_{H^2}^2\mathrm{d}t+\sup_{t\in[0,T]}\|n^{\epsilon}(t)\|_{H^1}^2\int_0^T\|c^\epsilon(t)\|_{H^{3}}^2\mathrm{d}t\\
&\quad+\sup_{t\in[0,T]}\|c^{\epsilon}(t)\|_{H^2}^2\int_0^T\|n^\epsilon(t)\|_{H^{2}}^2\mathrm{d}t+\sup_{t\in[0,T]}\|n^{\epsilon}(t)\|_{H^1}^6T\Bigl)\\
&\leq \tilde{C}_{10}\exp\left(\tilde{C}_{11}\exp\left(\tilde{C}_{12}\exp\left(\tilde{C}_{13}\exp\left(\tilde{C}_{14}(\sup_{t\in[0,T]}\|u^{\epsilon}(t)\|_{H^1}^2+\int_0^T\|u^{\epsilon}(t)\|_{H^{2}}^2\mathrm{d}t)\right)\right)\right)\right),
\end{split}
\end{equation*}
which implies   \eqref{cor.1}. Similarly, according to the second equation of \eqref{Mod-1}, we derive from \eqref{lem4-4} and \eqref{lem4-5} that
\begin{equation*}
\begin{split}
&\int_0^T\|c^{\epsilon}(r)\|_{H^1}^2+\left\|\partial_t c^{\epsilon} (r)\right\|_{L^{2}}^2+\left\|\nabla\partial_t c^{\epsilon}(r)\right\|_{L^{2}}^2\mathrm{d}r\\
&\leq C\int_0^T(\|u^\epsilon(t)\|_{H^2}^2\|c^{\epsilon}(t)\|_{H^1}^2+\|c^\epsilon(t)\|_{H^2}^2+\|c^\epsilon(t)\|_{L^{\infty}}^2\| n^{\epsilon}(t)\|_{L^2}^2)\mathrm{d}t\\
&\quad+C\int_0^T(\|u^\epsilon(t)\|_{H^1}^2\|c^{\epsilon}(t)\|_{H^3}^2+\|u^\epsilon(t)\|_{H^2}^2\|c^{\epsilon}(t)\|_{H^2}^2+\|c^\epsilon(t)\|_{H^3}^2+\|c^\epsilon(t)\|_{H^3}^2\| n^{\epsilon}(t)\|_{L^2}^2\\
&\quad+\|c^{\epsilon}(t)\|_{L^{\infty}}^2\|n^{\epsilon}(t)\|_{H^1}^2)\mathrm{d}t\\
&\leq C\Bigg(\sup_{t\in[0,T]}\|u^{\epsilon}(t)\|_{H^1}^2\int_0^T\|c^\epsilon(t)\|_{H^{3}}^2\mathrm{d}t+\sup_{t\in[0,T]}\|c^{\epsilon}(t)\|_{H^2}^2\int_0^T\|u^\epsilon(t)\|_{H^{2}}^2\mathrm{d}t\\
&\quad+\sup_{t\in[0,T]}\|c^{\epsilon}(t)\|_{H^2}^2\int_0^T\|n^\epsilon(t)\|_{H^{1}}^2\mathrm{d}t+(1+\sup_{t\in[0,T]}\|n^{\epsilon}(t)\|_{L^2}^2)\int_0^T\|c^\epsilon(t)\|_{H^{3}}^2\mathrm{d}t\Bigg)\\
&\leq\tilde{C}_{15}\exp\left(\tilde{C}_{16}\exp\left(\tilde{C}_{17}\exp\left(\tilde{C}_{18}\exp\left(\tilde{C}_{19}(\sup_{t\in[0,T]}\|u^{\epsilon}(t)\|_{H^1}^2+\int_0^T\|u^{\epsilon}(t)\|_{H^{2}}^2\mathrm{d}t)\right)\right)\right)\right),
\end{split}
\end{equation*}
which implies the estimate \eqref{cor.2}. The proof is thus completed.
\end{proof}

\subsection{Tightness of approximation solutions}\label{sec3.3}
Unlike the deterministic system, the uniform a priori entropy-energy inequality is insufficient for us to take the limit $\epsilon\rightarrow0$ to obtain the global solution directly, due to the loss of the topology structure of the space $\Omega$. To overcome this difficulty, we shall prove the tightness of the probability measures induced by the approximation solutions. Then with the help of the Prohorov theorem \cite{da2014stochastic} and Skorokhod representation theorem, one can find convergent subsequence defined on a new probability space. To begin with, let us introduce the following  phase spaces:
\begin{equation}\label{work}
\begin{split}
\mathcal{Z}_n&:=\mathcal{C}([0,T];U')\cap L^2_w(0,T;H^2(\mathbb{R}^3))\cap L^2(0,T;H^1_{loc}(\mathbb{R}^3))\cap \mathcal{C}([0,T];H^1_w(\mathbb{R}^3)),\\
\mathcal{Z}_c&:=\mathcal{C}([0,T];U')\cap L^2_w(0,T;H^3(\mathbb{R}^3))\cap L^2(0,T;H^2_{loc}(\mathbb{R}^3))\cap \mathcal{C}([0,T];H^2_w(\mathbb{R}^3)),\\
\mathcal{Z}_u&:=\mathcal{C}([0,T];U'_1)\cap L^2_w(0,T;\mathbb{H}^2)\cap L^2(0,T;\mathbb{H}^1_{loc})\cap \mathcal{C}([0,T];\mathbb{H}^1_w),
\end{split}
\end{equation}
where the Hilbert spaces $U$ and $U_1$ satisfy that, for any fixed $s>0$, the embeddings $U\subset H^s(\mathbb{R}^3)$ and $U_1\subset \mathbb{H}^s$ are compact (cf. \cite{11brzezniak2013existence}), and $U'$ and $U_1'$ are the corresponding dual spaces. Let $\mathcal{T}_n$, $\mathcal{T}_c$ and $\mathcal{T}_u$ be the supremum of the corresponding topologies with respect to the phase spaces.

Similar to the proof of \cite[Lemma 5.3]{12brzezniak2017note} and \cite[Lemma 2.7]{mikulevicius2005global}, it is not difficult to obtain the following compactness criteria.
\begin{lemma}\label{3lem} Let $(\mathcal{Z}_n,\mathcal{Z}_c,\mathcal{Z}_u)$ be as defined in \eqref{work}.
\begin{itemize}
\item [(\textbf{\textsf{A$_1$}})] A set $\mathcal{K}\subset\mathcal{Z}_n$ is $\mathcal{T}_n$-relatively compact if

\begin{itemize}
\item [(1)] $\sup_{f\in\mathcal{K}}\sup_{t\in[0,T]}\|f(t)\|_{H^1}^2+\sup_{f\in\mathcal{K}}\int_0^T\|f(t)\|_{H^2}^2\textrm{d}t<\infty$,

\item [(2)] $\exists\alpha>0$ such that $\sup_{f\in\mathcal{K}}\|f\|_{\mathcal{C}^{\alpha}([0,T];L^2(\mathbb{R}^3))}<\infty$.
\end{itemize}

\item [{(\textbf{\textsf{A$_2$}})}] A set $\mathcal{K}\subset\mathcal{Z}_c$ is $\mathcal{T}_c$-relatively compact if

\begin{itemize}
\item [(1)] $\sup_{f\in\mathcal{K}}\sup_{t\in[0,T]}\|f(t)\|_{H^2}^2+\sup_{f\in\mathcal{K}}\int_0^T\|f(t)\|_{H^3}^2\textrm{d}t<\infty$,

\item [(2)] $\exists\beta>0$ such that $\sup_{f\in\mathcal{K}}\|f\|_{\mathcal{C}^{\beta}([0,T];H^{1}(\mathbb{R}^3))}<\infty$.
\end{itemize}

\item [{(\textbf{\textsf{A$_3$}})}] A set $\mathcal{K}\subset\mathcal{Z}_u$ is $\mathcal{T}_u$-relatively compact if

\begin{itemize}
\item [(1)] $\sup_{f\in\mathcal{K}}\sup_{t\in[0,T]}\|f(t)\|_{\mathbb{H}^1}^2+\sup_{f\in\mathcal{K}}\int_0^T\|f(t)\|_{\mathbb{H}^2}^2\textrm{d}t<\infty$,

\item [(2)] $\lim\limits_{\delta\rightarrow0}\sup\limits_{f\in\mathcal{K}}\sup\limits_{s,t\in[0,T],|t-s|\leq\delta}\|f(t)-f(s)\|_{\mathbb{H}^0}=0$.
 \end{itemize}
\end{itemize}
\end{lemma}

The tightness of the laws of the approximation solutions is guaranteed by the following result.

\begin{proposition}\label{3pro} Let $(n^\epsilon,c^\epsilon,u^\epsilon)\in$ be the global solution of \eqref{Mod-1} constructed in Lemma \ref{lem2.5} with $\epsilon \in (0,1)$. Then the set of probability measures $\{(\mathscr{L}(n^{\epsilon}),\mathscr{L}(c^{\epsilon}),\mathscr{L}(u^{\epsilon}))\}_{\epsilon\in(0,1)}$ is tight on $(\mathcal{Z}_n\times\mathcal{Z}_c\times\mathcal{Z}_u,\mathcal{T}_n\times\mathcal{T}_c\times\mathcal{T}_u)$.
\end{proposition}
\begin{proof}[\emph{\textbf{Proof}}]
\underline{\textsf{Tightness of $\mathscr{L}(n^{\epsilon})$}}.   According to $(\textbf{\textsf{A$_1$}})$ in Lemma \ref{3lem}, to prove the tightness of $\mathscr{L}(n^{\epsilon})$, it is sufficient to prove that, for any $h>0$, there exist constants $C_i>0$, $i=1,~2,~3$, such that
\begin{align}
&\sup_{\epsilon\in(0,1)}\mathbb{P}\left\{\|n^{\epsilon}\|_{L^{\infty}(0,T;H^1(\mathbb{R}^3))}^2>C_1\right\}\leq h,\quad\sup_{\epsilon\in(0,1)}\mathbb{P}\left\{\|n^{\epsilon}\|_{L^{2}(0,T;H^2(\mathbb{R}^3))}^2>C_2\right\}\leq h,\label{0.1}\\
&\sup_{\epsilon\in(0,1)}\mathbb{P}\left\{\|n^{\epsilon}\|_{\mathcal{C}^{\frac{1}{2}}([0,T];L^{2}(\mathbb{R}^3))}^2>C_3\right\}\leq h.\label{0.2}
\end{align}

To prove the first one in \eqref{0.1}, let us choose the constant $C_1:=\tilde{C}_5e^{\tilde{C}_6e^{\tilde{C}_7e^{\tilde{C}_8e^{\frac{\tilde{C}_9C_{\phi,a,\mathbf{N},T,n_0,c_0,u_0}}{h}}}}}>0$, where the positive constants $\tilde{C}_5,...,\tilde{C}_9$ are provided in Lemma \ref{3lem}. Clearly, we observe inductively that $\frac{\ln (C_1/\tilde{C}_5)}{\tilde{C}_6}=e^{\tilde{C}_7e^{\tilde{C}_8e^{\frac{\tilde{C}_9C_{\phi,a,\mathbf{N},T,n_0,c_0,u_0}}{h}}}}>1$,...,
$\frac{\ln\Bigl(\frac{\ln\bigl(\frac{\ln(C_1/\tilde{C}_5)}{\tilde{C}_6}\bigl)}
{\tilde{C}_7}\Bigl)}{\tilde{C}_8}=e^{\frac{\tilde{C}_9C_{\phi,a,\mathbf{N},T,n_0,c_0,u_0}}{h}}>1$. By using the estimates \eqref{lem4-1} and \eqref{lem4-5}, we infer from Chebyshev's inequality that
\begin{equation*}
\begin{split}
&\sup_{\epsilon\in(0,1)}\mathbb{P}\left\{\|n^{\epsilon}\|_{L^{\infty}(0,T;H^1(\mathbb{R}^3))}^2>C_1\right\}\\
&\leq\sup_{\epsilon\in(0,1)}\mathbb{P}\left\{\tilde{C}_{5}\exp\left(\tilde{C}_6\exp\left(\tilde{C}_7\exp\left(\tilde{C}_{8}\exp\left(\tilde{C}_9(\sup_{t\in[0,T]}\|u^{\epsilon}(t)\|_{H^1}^2+\int_0^T\|u^{\epsilon}(t)\|_{H^{2}}^2\mathrm{d}t)\right)\right)\right)\right)>C_1\right\}\\
&\leq\sup_{\epsilon\in(0,1)}\mathbb{P}\Biggl\{\sup_{t\in[0,T]}\|u^{\epsilon}(t)\|_{H^1}^2+\int_0^T\|u^{\epsilon}(t)\|_{H^{2}}^2\mathrm{d}t>\frac{\ln\Biggl(\frac{\ln\Bigl(\frac{\ln\bigl(\frac{\ln(C_1/\tilde{C}_5)}{\tilde{C}_6}\bigl)}{\tilde{C}_7}\Bigl)}{\tilde{C}_8}\Biggl)}{\tilde{C}_9}\Biggl\}\\
&\leq\sup_{\epsilon\in(0,1)}\frac{\tilde{C}_9}{\ln\Biggl(\frac{\ln\Bigl(\frac{\ln\bigl(\frac{\ln(C_1/\tilde{C}_5)}{\tilde{C}_6}\bigl)}{\tilde{C}_7}\Bigl)}{\tilde{C}_8}\Biggl)}\left(\mathbb{E} \sup_{t\in [0,T]}\|u^{\epsilon}(t)\|_{H^1}^{2}
+ \mathbb{E} \int_0^{T}\|u^{\epsilon}(t)\|_{H^2}^2\mathrm{d}t \right)\\
&\leq\frac{\tilde{C}_9C_{\phi,a,\mathbf{N},T,n_0,c_0,u_0}}{\ln\Biggl(\frac{\ln\Bigl(\frac{\ln\bigl(\frac{\ln(C_1/\tilde{C}_5)}{\tilde{C}_6}\bigl)}{\tilde{C}_7}\Bigl)}{\tilde{C}_8}\Biggl)}\leq h,
\end{split}
\end{equation*}

The second one in \eqref{0.1} can be treated as the first one by taking $C_2:=C_1$.

To prove \eqref{0.2}, let us choose $C_3:=C_{\textrm{emb}}  \tilde{C}_{10}e^{\tilde{C}_{11}e^{\tilde{C}_{12}e^{\tilde{C}_{13}e^{\frac{\tilde{C}_{14}C_{\phi,a,\mathbf{N},T,n_0,c_0,u_0}}{h}}}}}>0$, where $C_{\textrm{emb}} >0$ denotes the Sobolev embedding constant satisfying $\|\cdot\|_{\mathcal{C}^{\frac{1}{2}}([0,T];X)}^2\leq C_{\textrm{emb}}  \|\cdot\|_{W^{1,2}(0,T;X)}^2$. Then, by using \eqref{lem4-1}, \eqref{cor.1} and Chebyshev's inequality, we get
\begin{equation*}
\begin{split}
\sup_{\epsilon\in(0,1)}\mathbb{P}\left\{\|n^{\epsilon}\|_{\mathcal{C}^{\frac{1}{2}}([0,T];L^{2}(\mathbb{R}^3))}^2>C_3\right\}&\leq\sup_{\epsilon\in(0,1)}\mathbb{P}\left\{\|n^{\epsilon}\|_{W^{1,2}(0,T;L^2(\mathbb{R}^3))}^2>\frac{C_3}{C_{\textrm{emb}} }\right\}\\
&\leq\frac{\tilde{C}_{14}C_{\phi,a,\mathbf{N},T,n_0,c_0,u_0}}{\ln\Biggl(\frac{\ln\Bigl(\frac{\ln\bigl
(\frac{\ln(C_3/C_{\textrm{emb}} \tilde{C}_{10})}{\tilde{C}_{11}}\bigl)}
{\tilde{C}_{12}}\Bigl)}{\tilde{C}_{13}}\Biggl)}\leq h.
\end{split}
\end{equation*}
Note that, similar to the constant $C_1$ as before, the structure of $C_3$ makes the logarithmic term in the last inequality  well-defined.

\underline{\textsf{Tightness of $\mathscr{L}(c^{\epsilon})$}}. By $(\textbf{\textsf{A$_2$}})$ in Lemma \ref{3lem}, to prove the tightness of $\mathscr{L}(c^{\epsilon})$, it is sufficient to show that, for any $h>0$, there exist constants $C_i>0$, $i=4,~5,~6$, such that
\begin{equation}\label{0.4}
\begin{split}
&\sup_{\epsilon\in(0,1)}\mathbb{P}\left\{\|c^{\epsilon}\|_{L^{\infty}(0,T;H^2(\mathbb{R}^3))}^2>C_4\right\}\leq h,\quad\sup_{\epsilon\in(0,1)}\mathbb{P}\left\{\|c^{\epsilon}\|_{L^{2}(0,T;H^3(\mathbb{R}^3))}^2>C_5\right\}\leq h,\\
&\sup_{\epsilon\in(0,1)}\mathbb{P}\left\{\|c^{\epsilon}\|_{\mathcal{C}^{\frac{1}{2}}([0,T];H^{1}(\mathbb{R}^3))}^2>C_6\right\}\leq h.
\end{split}
\end{equation}
Indeed, by setting $C_4:=\tilde{C}_{1}e^{\tilde{C}_{2}e^{\tilde{C}_{3}e^{\frac{\tilde{C}_{4}C_{\phi,a,\mathbf{N},T,n_0,c_0,u_0}}{h}}}}>0$, and using the inequality \eqref{lem4-1}, we derive from Chebyshev's inequality that
\begin{equation*}
\begin{split}
&\sup_{\epsilon\in(0,1)}\mathbb{P}\left\{\|c^{\epsilon}\|_{L^{\infty}(0,T;H^2(\mathbb{R}^3))}^2>C_4\right\}\\
&\leq\sup_{\epsilon\in(0,1)}\mathbb{P}\left\{\tilde{C}_{1}\exp\left(\tilde{C}_2\exp\left(\tilde{C}_{3}\exp\left(\tilde{C}_4(\sup_{t\in[0,T]}\|u^{\epsilon}(t)\|_{H^1}^2+\int_0^T\|u^{\epsilon}(t)\|_{H^{2}}^2\mathrm{d}t)\right)\right)\right)>C_4\right\}\\
&\leq\sup_{\epsilon\in(0,1)}\frac{\tilde{C}_4}{\ln\Bigl(\frac{\ln\bigl(\frac{\ln(\frac{C_4}{\tilde{C}_1})}{\tilde{C}_2}\bigl)}{\tilde{C}_3}\Bigl)}\left(\mathbb{E} \sup_{t\in [0,T]}\|u^{\epsilon}(t)\|_{H^1}^{2}
+ \mathbb{E}\left(\int_0^{T}\|u^{\epsilon}(t)\|_{H^2}^2\mathrm{d}t\right)\right)\\
&\leq\frac{\tilde{C}_4C_{\phi,a,\mathbf{N},T,n_0,c_0,u_0}}{\ln\Bigl(\frac{\ln\bigl(\frac{\ln(\frac{C_4}{\tilde{C}_1})}{\tilde{C}_2}\bigl)}{\tilde{C}_3}\Bigl)}\leq h.
\end{split}
\end{equation*}
In a similar manner, by \eqref{lem4-1}, one can prove that
$
\sup_{\epsilon\in(0,1)}\mathbb{P}\{\|c^{\epsilon}\|_{L^{2}(0,T;H^3(\mathbb{R}^3))}^2>C_4\}\leq h.
$

For the last one in \eqref{0.4}, by considering $C_6:=C_{\textrm{emb}}  \tilde{C}_{15}e^{\tilde{C}_{16}e^{\tilde{C}_{17}e^{\tilde{C}_{18}e^{\frac{\tilde{C}_{19}C_{\phi,a,\mathbf{N},T,n_0,c_0,u_0}}{h}}}}}>0$, we deduce from \eqref{cor.2} that
\begin{equation*}
\begin{split}
\sup_{\epsilon\in(0,1)}\mathbb{P}\left\{\|c^{\epsilon}\|_{\mathcal{C}^{\frac{1}{2}}([0,T];H^{1}(\mathbb{R}^3))}^2>C_6\right\}&\leq\sup_{\epsilon\in(0,1)}\mathbb{P}\left\{\|c^{\epsilon}\|_{W^{1,2}(0,T;H^1(\mathbb{R}^3))}^2>\frac{C_6}{C_{\textrm{emb}} }\right\}\\
&\leq\frac{\tilde{C}_{19}C_{\phi,a,\mathbf{N},T,n_0,c_0,u_0}}{\ln\Biggl(\frac{\ln\Bigl(\frac{\ln\bigl(\frac{\ln(\frac{C_6}{C_{\textrm{emb}} \tilde{C}_{15}})}{\tilde{C}_{16}}\bigl)}{\tilde{C}_{17}}\Bigl)}{\tilde{C}_{18}}\Biggl)}\leq h.
\end{split}
\end{equation*}

\underline{\textsf{Tightness of $\mathscr{L}(u^{\epsilon})$}}. According to $(\textbf{\textsf{A$_3$}})$ in Lemma \ref{3lem} and the uniform bound \eqref{lem4-1}, proving the tightness of $\mathscr{L}(u^{\epsilon})$ can be reduced to show that $\{u^{\epsilon}\}_{\epsilon\in(0,1)}$  satisfies the Aldous condition (cf. \cite{aldous1978stopping,aldous1989stopping}) in $\mathbb{H}^0$. To this end, let $\{\tau^{\epsilon}\}_{\epsilon\in(0,1)}$ be a sequence of stopping times such that $0\leq\tau^{\epsilon}\leq T$. By virtue of the third equation of \eqref{Mod-1}, we have for any $\theta>0$
\begin{equation*}
\begin{split}
&u^{\epsilon}(\tau^{\epsilon}+\theta)-u^{\epsilon}(\tau^{\epsilon})=\int_{\tau^{\epsilon}}^{\tau^{\epsilon}+\theta}-A u^\epsilon(r)\mathrm{d}r+\int_{\tau^{\epsilon}}^{\tau^{\epsilon}+\theta}-\textbf{P}[\mathbf{g}(|u^{\epsilon}(r)|^2)u^{\epsilon}(r)]\mathrm{d}r\\
&\quad+\int_{\tau^{\epsilon}}^{\tau^{\epsilon}+\theta} \textbf{P}[ (n^\epsilon(r)\nabla \phi)*\rho^\epsilon]\mathrm{d}r+\int_{\tau^{\epsilon}}^{\tau^{\epsilon}+\theta}-\textbf{P} (u^\epsilon(r)\cdot \nabla) u^\epsilon(r)\mathrm{d}r+\int_{\tau^{\epsilon}}^{\tau^{\epsilon}+\theta}\textbf{P}G(u^\epsilon(r)) \mathrm{d} W(r)\\
&:=J^{\epsilon}_1+\cdots+J^{\epsilon}_5.
\end{split}
\end{equation*}
By using H\"{o}lder's inequality, it follows from \eqref{lem4-1} that
\begin{equation*}
\begin{split}
\mathbb{E}\|J^{\epsilon}_1\|_{L^2}&\lesssim \mathbb{E}\int_{\tau^{\epsilon}}^{\tau^{\epsilon}+\theta}\|u^\epsilon(r)\|_{H^2}\mathrm{d}r\lesssim\theta^{\frac{1}{2}}\mathbb{E}\left(\int_{0}^{T}\|u^\epsilon(t)\|_{H^2}^2\mathrm{d}t\right)^{\frac{1}{2}}\\
&\lesssim_{\phi,a,\mathbf{N},T,n_0,c_0,u_0}\theta^{\frac{1}{2}}.
\end{split}
\end{equation*}
By using the Sobolev embedding $H^1(\mathbb{R}^3)\hookrightarrow L^6(\mathbb{R}^3)$, we have
\begin{equation*}
\begin{split}
\mathbb{E}\|J^{\epsilon}_2\|_{L^2}&\lesssim \mathbb{E}\int_{\tau^{\epsilon}}^{\tau^{\epsilon}+\theta}\|u^\epsilon(r)\|_{L^6}^3\mathrm{d}r\lesssim_{T}\theta^{\frac{1}{2}}\mathbb{E}\sup_{t\in[0,T]}\|u^\epsilon(t)\|_{H^1}^3\mathrm{d}t\\
&\lesssim_{\phi,a,\mathbf{N},T,n_0,c_0,u_0}\theta^{\frac{1}{2}}.
\end{split}
\end{equation*}
Similarly, we derive from \eqref{ad2} that
\begin{equation*}
\begin{split}
\mathbb{E}\|J^{\epsilon}_3\|_{L^2}+\mathbb{E}\|J^{\epsilon}_4\|_{L^2}&\lesssim \|\nabla\phi\|_{L^{\infty}}\mathbb{E}\int_{\tau^{\epsilon}}^{\tau^{\epsilon}+\theta}\|n^\epsilon(r)\|_{L^2}\mathrm{d}r+\mathbb{E}\int_{\tau^{\epsilon}}^{\tau^{\epsilon}+\theta}\|u^\epsilon(r)\|_{H^2}\|u^\epsilon(r)\|_{H^1}\mathrm{d}r\\
&\lesssim_{\phi,a,\mathbf{N},T,n_0,c_0,u_0}\theta^{\frac{1}{2}}.
\end{split}
\end{equation*}
Moreover, we infer from Assumption \ref{as3} that
\begin{equation*}
\begin{split}
\mathbb{E}\|J^{\epsilon}_5\|_{L^2}^2&=\mathbb{E}\int_{\tau^{\epsilon}}^{\tau^{\epsilon}+\theta}\|\textbf{P}G(u^\epsilon(r))\|_{\mathcal{L}_2(\mathbf{U};L^2)}^2\mathrm{d}r\lesssim\theta(1+\mathbb{E}\sup_{t\in[0,T]}\|u^\epsilon(t)\|_{L^2}^2)\\
&\lesssim_{\phi,a,\mathbf{N},T,n_0,c_0,u_0}\theta.
\end{split}
\end{equation*}
Thus, by using Chebyshev's inequality, we infer from the above estimates that for any $h,~\eta>0$ there exists $\delta:=\min\left\{\frac{h^2\eta^2}{(40C_{\phi,a,\mathbf{N},T,n_0,c_0,u_0})^2},\frac{h\eta^2}{50C_{\phi,a,\mathbf{N},T,n_0,c_0,u_0}}\right\}$ such that
\begin{equation*} \begin{split}
&\sup_{\epsilon\in(0,1)}\sup_{\theta\in{[0,\delta]}}\mathbb{P}\left\{\|u^{\epsilon}(\tau^{\epsilon}+\theta)-u^{\epsilon}(\tau^{\epsilon})\|_{L^2}\geq\eta\right\}\\
&\leq\sup_{\epsilon\in(0,1)}\sup_{\theta\in{[0,\delta]}}\sum_{i=1}^5\mathbb{P}\left\{\|J^{\epsilon}_i\|_{L^2}\geq\frac{\eta}{5}\right\}\leq\sup_{\epsilon\in(0,1)}\sup_{\theta\in{[0,\delta]}}\sum_{i=1}^4\frac{5}{\eta}\mathbb{E}\|J^{\epsilon}_i\|_{L^2}+\frac{25}{\eta^2}\mathbb{E}\|J^{\epsilon}_5\|_{L^2}^2\\
&\leq \frac{20C_{\phi,a,\mathbf{N},T,n_0,c_0,u_0}\theta^{\frac{1}{2}}}{\eta}+\frac{25C_{\phi,a,\mathbf{N},T,n_0,c_0,u_0}\theta}{\eta^2}\leq \frac{h}{2}+\frac{h}{2}\leq h,
\end{split} \end{equation*}
which implies that the sequence $\{u^{\epsilon}\}_{\epsilon\in(0,1)}$ satisfies the Aldous condition in $\mathbb{H}^0$. The proof of Proposition \ref{3pro} is thus completed.
\end{proof}

\subsection{Global strong solutions for STCNS system}\label{sec3.4}
Theorem \ref{th1} can now be proved based on the Proposition \ref{3pro}.

\vspace{3mm}\noindent
\textbf{Existence of a global martingale solution.}  The proof  is divided into two steps.

\textsf{Step 1 (Convergence in new space).} According to Proposition \ref{3pro}, the family of induced measures
$\{ (\mathscr{L}(n^{\epsilon}), \mathscr{L}(c^{\epsilon}), \mathscr{L}(u^{\epsilon})) \}_{\epsilon \in (0,1)}$
is tight on $\mathcal{Z}_n \times \mathcal{Z}_c \times \mathcal{Z}_u$. Moreover, the sequence of probability measures defined by
$\mathscr{L}(W^\epsilon)(\cdot) = \mathscr{L}(W)(\cdot) := \mathbb{P}[W^\epsilon \in \cdot]$
is tight on $\mathcal{C}([0,T]; \mathbf{U}_0)$. Therefore, by applying the Jakubowski-Skorokhod theorem (cf.~\cite{jakubowski1998almost}),
there exists a subsequence $(\epsilon_k)_{k \in \mathbb{N}}$ and
$\mathcal{Z}_n \times \mathcal{Z}_c \times \mathcal{Z}_u \times \mathcal{C}([0,T]; \mathbf{U}_0)$-valued random variables
$(n_*, c_*, u_*, W_*)$, $(\bar{n}^k, \bar{c}^k, \bar{u}^k, \bar{W}^k)_{k \in \mathbb{N}}$
defined on a new probability space $(\bar{\Omega}, \bar{\mathcal{F}}, \bar{\mathbb{P}})$, such that:
\begin{itemize}
  \item[(1)] $\mathscr{L}(\bar{n}^k, \bar{c}^k, \bar{u}^k, \bar{W}^k) = \mathscr{L}(n^{\epsilon_k}, c^{\epsilon_k}, u^{\epsilon_k}, W^{\epsilon_k})$ for all $k \in \mathbb{N}$;
  \item[(2)] $(\bar{n}^k, \bar{c}^k, \bar{u}^k, \bar{W}^k) \to (n_*, c_*, u_*, W_*)$ in
  $\mathcal{Z}_n \times \mathcal{Z}_c \times \mathcal{Z}_u \times \mathcal{C}([0,T]; \mathbf{U}_0)$
  on $(\bar{\Omega}, \bar{\mathcal{F}}, \bar{\mathbb{P}})$ as $k \to \infty$, $\bar{\mathbb{P}}$-a.s.
\end{itemize}
More precisely, the conclusion of (2) indicates the following  properties:
\begin{equation}\label{5-1}
\begin{split}
&\bar{n}^k\rightarrow n_*~\textrm{in}~\mathcal{C}([0,T];U')\cap L^2_w(0,T;H^2(\mathbb{R}^3))\cap L^2(0,T;H^1_{loc}(\mathbb{R}^3))\cap \mathcal{C}([0,T];H^1_w(\mathbb{R}^3)),\quad \bar{\mathbb{P}}\textrm{-a.s.,}\\
&\bar{c}^k\rightarrow c_*~\textrm{in}~\mathcal{C}([0,T];U')\cap L^2_w(0,T;H^3(\mathbb{R}^3))\cap L^2(0,T;H^2_{loc}(\mathbb{R}^3))\cap \mathcal{C}([0,T];H^2_w(\mathbb{R}^3)),\quad \bar{\mathbb{P}}\textrm{-a.s.,}\\
&\bar{u}^k\rightarrow u_*~\textrm{in}~\mathcal{C}([0,T];U'_1)\cap L^2_w(0,T;\mathbb{H}^2)\cap L^2(0,T;\mathbb{H}^1_{loc})\cap \mathcal{C}([0,T];\mathbb{H}^1_w),\quad \bar{\mathbb{P}}\textrm{-a.s.}
\end{split}
\end{equation}
Since the laws of $(n^{\epsilon_k},c^{\epsilon_k},u^{\epsilon_k})$ and $(\bar{n}^k,\bar{c}^k,\bar{u}^k)$ are equal in $\mathcal{Z}_n\times\mathcal{Z}_c\times\mathcal{Z}_u$, we infer from \eqref{ad2} and Lemma \ref{lem4} that for all $p\geq1$, there exists a constant $C>0$ independent of $k$ such that
\begin{equation}\label{5-2}
\begin{split}
&\bar{\mathbb{E}} \left( \int_0^{T}\|\bar{n}^k(t)\|_{L^2}^2\textrm{d}t \right) ^p+
\bar{\mathbb{E}}\sup_{t\in[0,T]}\|\bar{c}^k(t)\|_{H^1}^{2p}+\bar{\mathbb{E}} \left( \int_0^{T}\|\bar{c}^k(t)\|_{H^2}^2\textrm{d}t \right) ^p\\
&\quad+
\bar{\mathbb{E}}\sup_{t\in[0,T]}\|\bar{u}^k(t)\|_{H^1}^{2p}+\bar{\mathbb{E}} \left( \int_0^{T}\|\bar{u}^k(t)\|_{H^2}^2\textrm{d}t \right) ^p\leq C.
\end{split}
\end{equation}
By the Banach-Alaoglu theorem and the uniqueness of weak-star limit, we see that $(n_*,c_*,u_*)\in L^{2p}(\bar{\Omega};L^2(0,T;L^2(\mathbb{R}^3)))\times L^{2p}(\bar{\Omega};L^{\infty}(0,T;H^1(\mathbb{R}^3))\cap L^2(0,T;H^2(\mathbb{R}^3)))\times L^{2p}(\bar{\Omega};L^{\infty}(0,T;H^1(\mathbb{R}^3))\cap L^2(0,T;H^2(\mathbb{R}^3)))$. Moreover, invoking the Fatou lemma, we get from \eqref{5-2} that
\begin{equation}\label{5-22}
\begin{split}
&\bar{\mathbb{E}} \left( \int_0^{T}\|n_*(t)\|_{L^2}^2\textrm{d}t \right) ^p+\bar{\mathbb{E}}\sup_{t\in[0,T]}\|c_*(t)\|_{H^1}^{2p}+
\bar{\mathbb{E}} \left( \int_0^{T}\|c_*(t)\|_{H^2}^2\textrm{d}t \right) ^p\\
&\quad+\bar{\mathbb{E}}\sup_{t\in[0,T]}\|u_*(t)\|_{H^1}^{2p}+
\bar{\mathbb{E}} \left( \int_0^{T}\|u_*(t)\|_{H^2}^2\textrm{d}t \right) ^p\leq C.
\end{split}
\end{equation}

\textsf{Step 2 (Construction of martingale solution).} For any $(\varphi,\psi)\in L^2(\mathbb{R}^3)\times \mathbb{H}^0$, let us define
\begin{equation*}
\begin{split}
\mathbf{A}^k(t)&:=(\bar{n}_0*\rho^{\frac{1}{k}},\varphi)_{L^2}+\int_0^t(\Delta \bar{n}^k(r)-\bar{u}^k(r)\cdot\nabla \bar{n}^k(r)-\nabla\cdot (\bar{n}^k(r)\nabla (\bar{c}^k(r)*\rho^{\frac{1}{k}}))\\
&\quad+L(\bar{n}^k(r)),\varphi)_{L^2}\mathrm{d}r,\\
\mathbf{B}^k(t)&:=(\bar{c}_0*\rho^{\frac{1}{k}},\varphi)_{L^2}+\int_0^t(\Delta \bar{c}^k(r)-\bar{u}^k(r)\cdot\nabla \bar{c}^k(r)-\bar{c}^k(r)(\bar{n}^k(r)*\rho^{\frac{1}{k}}),\varphi)_{L^2}\mathrm{d}r,\\
\mathbf{C}^k(t)&:=(\bar{u}_0*\rho^{\frac{1}{k}},\psi)_{L^2}+\int_0^t(\Delta \bar{u}^k(r)-(\bar{u}^k(r)\cdot\nabla)\bar{u}^k(r)-\mathbf{g}(|\bar{u}^k(r)|^2)\bar{u}^k(r)\\
&\quad+(\bar{n}^k(r) \nabla\phi)*\rho^{\frac{1}{k}},\psi)_{L^2}\mathrm{d}r+\int_0^t (G(\bar{u}^k(r))\mathrm{d}\bar{W}^k(r),\psi)_{L^2},
 \end{split}
\end{equation*}
for every integer $k\geq 1$, and
\begin{equation*}
\begin{split}
\mathbf{A}(t)&:=(n_*(0),\varphi)_{L^2}+\int_0^t(\Delta n_*(r)-u_*(r)\cdot\nabla n_*(r)-\nabla\cdot (n_*(r)\nabla c_*(r)))+L(n_*(r)),\varphi)_{L^2}\mathrm{d}r,\\
\mathbf{B}(t)&:=(c_*(0),\varphi)_{L^2}+\int_0^t(\Delta c_*(r)-u_*(r)\cdot\nabla c_*(r)-c_*(r)n_*(r),\varphi)_{L^2}\mathrm{d}r,\\
\mathbf{C}(t)&:=(u_*(0),\psi)_{L^2}+\int_0^t(\Delta u_*(r)-(u_*(r)\cdot\nabla)u_*(r)-\mathbf{g}(|u_*(r)|^2)u_*(r)+n_* (r) \nabla\phi,\psi)_{L^2}\mathrm{d}r\\
&\quad+\int_0^t (G(u_*(r))\mathrm{d}W_*(r),\psi)_{L^2}.
 \end{split}
\end{equation*}
To finish the proof of existence part, it remains to prove that the terms involved in the terms $\mathbf{A}^k(t),\mathbf{B}^k(t)$ and $\mathbf{C}^k(t)$ converge to the terms involved in $\mathbf{A}(t),\mathbf{B}(t)$ and $\mathbf{C}(t)$, respectively.

\underline{\textsf{Convergence for $n$, $c-$equations}}.  Based on \eqref{5-1}, for each term in $\mathbf{A}^k(t)$ and $\mathbf{B}^k(t)$, it is not difficult to obtain the corresponding convergence result. As an example, we shall only provide a verification for the cross-diffusion term in the $n$-equation, i.e.,
\begin{equation}\label{ddd}
\begin{split}
&\lim_{k\rightarrow\infty}\int_0^t(\nabla\cdot (\bar{n}^k(r)\nabla (\bar{c}^k(r)*\rho^{\frac{1}{k}})),\varphi)_{L^2}\mathrm{d}r\\
&=\int_0^t(\nabla\cdot (n_*(r)\nabla c_*(r)),\varphi)_{L^2}\mathrm{d}r,\quad\forall \varphi\in L^2(\mathbb{R}^3),\quad\bar{\mathbb{P}}\textrm{-a.s.}
 \end{split}
\end{equation}
Indeed, for every $\varphi\in L^2(\mathbb{R}^3)$ and $h>0$, there exists a function $\varphi_{h}\in \mathcal{C}_0^{\infty}(\mathbb{R}^3)$ such that $\|\varphi-\varphi_h\|_{L^2}\leq h$.
Moreover, for such a  $\varphi_h\in \mathcal{C}_0^{\infty}(\mathbb{R}^3)$, one can find a $d>0$ large enough such that supp$~\varphi_h \subset \mathcal{O}_d$, where $\{\mathcal{O}_d\}_{d\in\mathbb{N}}$ is a sequence of bounded open subsets of $\mathbb{R}^3$ defined in Subsection \ref{nnnn}.  Thus by the triangle inequality, we have
\begin{equation}\label{d.1}
\begin{split}
&\left|(\nabla\cdot (\bar{n}^k\nabla (\bar{c}^k*\rho^{\frac{1}{k}}))-\nabla\cdot (n_*\nabla c_*),\varphi)_{L^2}\right|\\
&\leq\left|(\nabla\cdot (\bar{n}^k\nabla (\bar{c}^k*\rho^{\frac{1}{k}}))-\nabla\cdot (n_*\nabla c_*),\varphi-\varphi_h)_{L^2}\right|\\
&\quad+\left|(\nabla\cdot (\bar{n}^k\nabla (\bar{c}^k*\rho^{\frac{1}{k}}))-\nabla\cdot (n_*\nabla c_*),\varphi_h)_{L^2}\right|\\
&\leq \left(\|\nabla\cdot (\bar{n}^k\nabla (\bar{c}^k*\rho^{\frac{1}{k}}))\|_{L^2}+\|\nabla\cdot (n_*\nabla c_*)\|_{L^2}\right)\|\varphi-\varphi_h\|_{L^2}\\
&\quad+\left|(\nabla\cdot ((\bar{n}^k-n_*)\nabla (\bar{c}^k*\rho^{\frac{1}{k}})),\varphi_h)_{L^2}\right|+\left|(\nabla\cdot (n_*\nabla (\bar{c}^k*\rho^{\frac{1}{k}}-c_*)),\varphi_h)_{L^2}\right|.
\end{split}
\end{equation}
By using the Sobolev embedding $H^2(\mathbb{R}^3)\hookrightarrow L^{\infty}(\mathbb{R}^3)$ and $H^1(\mathbb{R}^3)\hookrightarrow L^{p}(\mathbb{R}^3)$ for $p\in[2,6]$, the first term on the R.H.S. of \eqref{d.1} can be estimated as
\begin{equation*}
\begin{split}
&\left(\|\nabla\cdot (\bar{n}^k\nabla (\bar{c}^k*\rho^{\frac{1}{k}}))\|_{L^2}+\|\nabla\cdot (n_*\nabla c_*)\|_{L^2}\right)\|\varphi-\varphi_h\|_{L^2}\\
&\lesssim \left(\|\bar{n}^k\|_{H^1}\|\nabla\bar{c}^k\|_{L^{\infty}}+\|\bar{n}^k\|_{L^6}\|\Delta\bar{c}^k\|_{L^{3}}+\|n_*\|_{H^1}\|\nabla c_*\|_{L^{\infty}}+\|n_*\|_{L^6}\|\Delta c_*\|_{L^{3}}\right)h\\
&\lesssim \left(\|\bar{n}^k\|_{H^1}\|\bar{c}^k\|_{H^3}+\|n_*\|_{H^1}\|c_*\|_{H^3}\right)h.
\end{split}
\end{equation*}
Similarly, the second and third terms can be treated by
\begin{equation*}
\begin{split}
&\left|(\nabla\cdot ((\bar{n}^k-n_*)\nabla (\bar{c}^k*\rho^{\frac{1}{k}})),\varphi_h)_{L^2}\right|\lesssim\|\bar{n}^k-n_*\|_{H^1(\mathcal{O}_d)}\|\bar{c}^k\|_{H^3(\mathcal{O}_d)}\|\varphi_h\|_{L^2(\mathcal{O}_d)},
\end{split}
\end{equation*}
and
\begin{equation*}
\begin{split}
&\left|(\nabla\cdot (n_*\nabla (\bar{c}^k*\rho^{\frac{1}{k}}-c_*)),\varphi_h)_{L^2}\right|\\
&\lesssim\left(\|\nabla n_*\|_{L^3(\mathcal{O}_d)}\|\nabla(\bar{c}^k*\rho^{\frac{1}{k}}-c_*)\|_{L^6(\mathcal{O}_d)}+\|n_*\|_{L^{\infty}(\mathcal{O}_d)}\|\Delta(\bar{c}^k*\rho^{\frac{1}{k}}-c_*)\|_{L^2(\mathcal{O}_d)}\right)\|\varphi_h\|_{L^2(\mathcal{O}_d)}\\
&\lesssim\|n_*\|_{H^{2}(\mathcal{O}_d)}\|\bar{c}^k*\rho^{\frac{1}{k}}-c_*\|_{H^2(\mathcal{O}_d)}\|\varphi_h\|_{L^2(\mathcal{O}_d)}
.\end{split}
\end{equation*}
Plugging the last three estimates into \eqref{d.1}, we get for all $t\in[0,T]$ and $\bar{\mathbb{P}}$-a.s.
\begin{equation}\label{d.2}
\begin{split}
&\left|\int_0^t(\nabla\cdot (\bar{n}^k(r)\nabla (\bar{c}^k(r)*\rho^{\frac{1}{k}}))-\nabla\cdot (n_*(r)\nabla c_*(r)),\varphi)_{L^2}\mathrm{d}r\right|\\
&\lesssim h\int_0^T(\|\bar{n}^k(t)\|_{H^1}^2+\|\bar{c}^k(t)\|_{H^3}^2+\|n_*(t)\|_{H^1}^2+\|c_*(t)\|_{H^3}^2)\mathrm{d}t\\
&\quad+\|\varphi_h\|_{L^2(\mathcal{O}_d)}\|\bar{c}^k\|_{L^2(0,T;H^3(\mathcal{O}_d))}\|\bar{n}^k-n_*\|_{L^2(0,T;H^1(\mathcal{O}_d))}\\
&\quad+\|\varphi_h\|_{L^2(\mathcal{O}_d)}\|n_*\|_{L^2(0,T;H^2(\mathcal{O}_d))}\|\bar{c}^k*\rho^{\frac{1}{k}}-c_*\|_{L^2(0,T;H^2(\mathcal{O}_d))}.
\end{split}
\end{equation}
Taking the upper limit and using the result \eqref{5-1}, we infer from \eqref{d.2} that
\begin{equation*}
\begin{split}
&\limsup_{k\rightarrow\infty}\left|\int_0^t(\nabla\cdot (\bar{n}^k(r)\nabla (\bar{c}^k(r)*\rho^{\frac{1}{k}}))-\nabla\cdot (n_*(r)\nabla c_*(r)),\varphi)_{L^2}\mathrm{d}r\right|\leq Ch, \quad \bar{\mathbb{P}}\textrm{-a.s.,}
\end{split}
\end{equation*}
which combined with the arbitrariness of $h$ implies that \eqref{ddd} holds. As a result, we have for all $t\in[0,T]$
\begin{equation}\label{d.3}
\begin{split}
&\lim_{k\rightarrow\infty}\mathbf{A}^k(t)=\mathbf{A}(t),\quad\lim_{k\rightarrow\infty}\mathbf{B}^k(t)=\mathbf{B}(t), \quad \bar{\mathbb{P}}\textrm{-a.s.}
\end{split}
\end{equation}

\underline{\textsf{Convergence for $u-$equations}}. We shall prove that for all $t \in [0,T]$
\begin{equation}\label{d.889}
\begin{split}
&\lim_{k\rightarrow\infty}\mathbf{C}^k(t)=\mathbf{C}(t), \quad \bar{\mathbb{P}}\textrm{-a.s.}
\end{split}
\end{equation}
Note that due to the almost surely convergence result in \eqref{5-1}, the convergence of the linear part is obvious, and for the nonlinear part we also have
\begin{equation*}
\begin{split}
\lim_{k\rightarrow\infty}\left|\int_0^t(\mathbf{g}(|\bar{u}^k(r)|^2)
\bar{u}^k(r)-\mathbf{g}(|u_*(r)|^2)u_*(r),\psi)_{L^2}\mathrm{d}r\right|=0, \quad \bar{\mathbb{P}}\textrm{-a.s.,}
\end{split}
\end{equation*}
and
\begin{equation*}
\begin{split}
\lim_{k\rightarrow\infty}\left|\int_0^t((\bar{u}^k(r)\cdot\nabla)\bar{u}^k(r)
-(u_*(r)\cdot\nabla)u_*(r),\psi)_{L^2}\mathrm{d}r\right|=0, \quad \bar{\mathbb{P}}\textrm{-a.s.}
\end{split}
\end{equation*}
We would like to omit the details of the verification for the above two identities that follows from simple calculations, and mainly focus on the treatment of stochastic integral. Precisely, we shall demonstrate that for all $\psi \in \mathbb{H}^0$ and $t \in [0,T]$,
\begin{equation}\label{d.8}
\lim_{k \to \infty} \int_0^t (G(\bar{u}^k(r))\mathrm{d} \bar{W}^k(r), \psi)_{L^2} = \int_0^t (G(u_*(r))\mathrm{d} W_*(r), \psi)_{L^2}, \quad \bar{\mathbb{P}}\text{-a.s.}
\end{equation}
Let $h > 0$ be arbitrary. Since $\mathcal{V}$ is dense in $\mathbb{H}^0$, there exists $\psi_h \in \mathcal{V}$ such that $\|\psi - \psi_h\|_{L^2} \leq h$. Then we decompose:
\begin{equation*}
\begin{split}
\int_0^t (G(\bar{u}^k(r))\mathrm{d}\bar{W}^k(r) - G(u_*(r))\mathrm{d} W_*(r), \psi)_{L^2} &=\int_0^t (G(\bar{u}^k(r))\mathrm{d}\bar{W}^k(r) - G(u_*(r))\mathrm{d} W_*(r), \psi_h)_{L^2}\\
&\quad+\int_0^t (G(\bar{u}^k(r))\mathrm{d} \bar{W}^k(r)-G(u_*(r))\mathrm{d} W_*(r), \psi - \psi_h)_{L^2}.
\end{split}
\end{equation*}
Applying the BDG inequality, Assumption \ref{as3} and the estimates \eqref{5-2}-\eqref{5-22}, we have for all $t\in[0,T]$
\begin{equation}\label{abc}
\begin{split}
&\bar{\mathbb{E}}\left|\int_0^t (G(\bar{u}^k(r))\mathrm{d}\bar{W}^k(r) - G(u_*(r))\mathrm{d} W_*(r), \psi)_{L^2}\right|^2\\
&\lesssim\bar{\mathbb{E}}\left|\int_0^t (G(\bar{u}^k(r))\mathrm{d}\bar{W}^k(r) - G(u_*(r))\mathrm{d} W_*(r), \psi_h)_{L^2}\right|^2\\
&\quad+\bar{\mathbb{E}}\sup_{t\in[0,T]}\left|\int_0^t (G(\bar{u}^k(r))\mathrm{d} \bar{W}^k(r), \psi - \psi_h)_{L^2}\right|^2\\
&\quad+\bar{\mathbb{E}}\sup_{t\in[0,T]}\left|\int_0^t (G(u_*(r))\mathrm{d} W_*(r), \psi - \psi_h)_{L^2}\right|^2\\
&\lesssim\bar{\mathbb{E}}\left|\int_0^t (G(\bar{u}^k(r))\mathrm{d}\bar{W}^k(r) - G(u_*(r))\mathrm{d} W_*(r), \psi_h)_{L^2}\right|^2\\
&\quad+\|\psi - \psi_h\|_{L^2}^2\left(\bar{\mathbb{E}} \int_0^T \|G(\bar{u}^k(t))\|_{\mathcal{L}_2(\mathbf{U}; L^2)}^2 \mathrm{d}t+\bar{\mathbb{E}} \int_0^T \|G(u_*(t))\|_{\mathcal{L}_2(\mathbf{U}; L^2)}^2 \mathrm{d}t\right)\\
&\leq Ch^2 +C\bar{\mathbb{E}}\left|\int_0^t (G(\bar{u}^k(r))\mathrm{d}\bar{W}^k(r) - G(u_*(r))\mathrm{d} W_*(r), \psi_h)_{L^2}\right|^2.
\end{split}
\end{equation}
To finish the proof of \eqref{d.8}, it is sufficient to prove that the second term on the R.H.S. of \eqref{abc} converges to zero along some subsequence of $\{k\}$. Indeed, since $\bar{u}^k \to u_*$ in $L^2(0,T; \mathbb{H}^1_{\mathrm{loc}})$ $\bar{\mathbb{P}}$-a.s., we deduce from the Lipschitz condition in Assumption \ref{as3} that $G(\bar{u}^k) \to G(u_*)$ in $L^2(0,T;\mathcal{L}_2(\mathbf{U};\mathbb{H}^1_{\mathrm{loc}}))$ $\bar{\mathbb{P}}$-a.s. Moreover, as $\bar{W}^k \to W_*$ in $\mathcal{C}([0,T]; \mathbf{U}_0)$ $\bar{\mathbb{P}}$-a.s., it follows from Lemma 2.4.35 in \cite{mensah2019stochastic} (see also \cite[Lemma 2.6.6]{Breit}) that there exists a subsequence (still denoted by itself) such that for all $t\in[0,T]$, $\int_0^t(G(\bar{u}^k(r))\mathrm{d} \bar{W}^k(r), \psi_h)_{L^2}\rightarrow\int_0^t(G(u_*(r))\mathrm{d}\bar{W}_*(r), \psi_h)_{L^2}$ in probability. By the Borel-Cantelli lemma, one can choose a subsequence (still denoted by itself) such that for all $t\in[0,T]$,
\begin{equation}\label{abc1}
\begin{split}
\lim_{k\rightarrow\infty}\int_0^t (G(\bar{u}^k(r))\mathrm{d} \bar{W}^k(r), \psi_h)_{L^2} = \int_0^t (G(u_*(r))\mathrm{d} W_*(r), \psi_h)_{L^2}, \quad \bar{\mathbb{P}}\text{-a.s.}
\end{split}
\end{equation}
Moreover, we get by \eqref{5-2} and the BDG inequality that
\begin{equation}\label{abc2}
\begin{split}
\bar{\mathbb{E}}\sup_{t\in[0,T]}\left|\int_0^t (G(\bar{u}^k(r))\mathrm{d}\bar{W}^k(r), \psi_h)_{L^2}\right|^{2p}&\lesssim_{T,p}\|\psi_h\|_{L^2}^{2p}\left(1+\bar{\mathbb{E}}\sup_{t\in[0,T]}\|\bar{u}^k(t)\|_{L^2}^{2p}\right)\\
&\lesssim_{T,p}\|\psi_h\|_{L^2}^{2p}<\infty,\quad\forall p\geq1.
\end{split}
\end{equation}
Thanks to the Vitali convergence theorem, we infer from \eqref{abc1}-\eqref{abc2} that for all $t\in[0,T]$
\begin{equation}\label{bbbb}
\begin{split}
\lim_{k\rightarrow\infty}\bar{\mathbb{E}}\left|\int_0^t (G(\bar{u}^k(r))\mathrm{d}\bar{W}^k(r) - G(u_*(r))\mathrm{d} W_*(r), \psi_h)_{L^2}\right|^2=0.
\end{split}
\end{equation}
Since $h > 0$ is arbitrary, it follows from \eqref{abc} and \eqref{bbbb} that
\begin{equation*}
\begin{split}
\lim_{k \to \infty} \bar{\mathbb{E}}\left| \int_0^t (G(\bar{u}^k(r))\mathrm{d} \bar{W}^k(r) - G(u_*(r))\mathrm{d} W_*(r), \psi)_{L^2}\right|^2=0,
\end{split}
\end{equation*}
for all $t\in[0,T]$ and $\psi \in \mathbb{H}^0$, which implies \eqref{d.8} and hence \eqref{d.889}.

Since $(n^k,c^k,u^k)$ is a solution of system \eqref{Mod-1} and $\mathscr{L}(n^k,c^k,u^k,W^k)=\mathscr{L}(\bar{n}^{k},\bar{c}^{k},\bar{u}^{k},\bar{W}^{k})$, we deduce that for all $\varphi\in L^2(\mathbb{R}^3)$, $\psi\in\mathbb{H}^0$, $t\in[0,T]$ and $\bar{\mathbb{P}}$-a.s., the following holds:
\begin{equation}\label{e.1}
\begin{split}
(\bar{n}^k(t),\varphi)_{L^2}=\mathbf{A}^k(t),\quad(\bar{c}^k(t),\varphi)_{L^2}=\mathbf{B}^k(t),\quad(\bar{u}^k(t),\psi)_{L^2}=\mathbf{C}^k(t).
\end{split}
\end{equation}
Let $U$ be a dense subset of $L^2(\mathbb{R}^3)$ and $U_1$ be a dense subset of $\mathbb{H}^0$. Since $\bar{n}^k\rightarrow n_*$, $\bar{c}^k\rightarrow c_*$ in $\mathcal{C}([0,T];U')$ and $\bar{u}^k\rightarrow u_*$ in $\mathcal{C}([0,T];U'_1)$ (cf. \eqref{5-1}), it follows that for all $t \in [0,T]$ and $\bar{\mathbb{P}}$-a.s., we have:
\begin{equation}\label{e.3}
\begin{split}
&\lim_{k\rightarrow\infty}(\bar{n}^k(t),\varphi)_{L^2}=(n_*(t),\varphi)_{L^2},\quad\lim_{k\rightarrow\infty}(\bar{c}^k(t),\varphi)_{L^2}=(c_*(t),\varphi)_{L^2},\\
&\lim_{k\rightarrow\infty}(\bar{u}^k(t),\psi)_{L^2}=(u_*(t),\psi)_{L^2}.
\end{split}
\end{equation}
By combining \eqref{d.3}, \eqref{d.889}, \eqref{e.1} and \eqref{e.3}, we conclude that for all $\varphi\in L^2(\mathbb{R}^3)$, $\psi\in\mathbb{H}^0$ and $t\in[0,T]$
\begin{equation*}
\begin{split}
(n_*(t),\varphi)_{L^2}&=(n_*(0),\varphi)_{L^2}+\int_0^t(\Delta n_*(r)-u_*(r)\cdot\nabla n_*(r)-\nabla\cdot (n_*(r)\nabla c_*(r))\\
&\quad+L(n_*(r)),\varphi)_{L^2}\mathrm{d}r,\\
(c_*(t),\varphi)_{L^2}&=(c_*(0),\varphi)_{L^2}+\int_0^t(\Delta c_*(r)-u_*(r)\cdot\nabla c_*(r)-n_*(r)c_*(r),\varphi)_{L^2}\mathrm{d}r\\
(u_*(t),\psi)_{L^2}&=(u_*(0),\psi)_{L^2}+\int_0^t(\Delta u_*(r)-(u_*(r)\cdot\nabla)u_*(r)-\mathbf{g}(|u_*(r)|^2)u_*(r)\\
&\quad+n_*(r) \nabla\phi,\psi)_{L^2}\mathrm{d}r+\int_0^t (G(u_*(r))\mathrm{d}W_*(r),\psi)_{L^2},\quad \bar{\mathbb{P}}\textrm{-a.s.}
\end{split}
\end{equation*}
Furthermore, by \eqref{5-1} and applying the Lions-Magenes lemma \cite{ad-2lions1963problemes} as well as Lemma 4.1 in \cite{pardoux1975equations}, one can easily verify that the triple  $(n_*,c_*,u_*)$ satisfies the space-time regularity stated in Theorem \ref{th1}. As a result, we conclude that the tuple
$
((\bar{\Omega},\bar{\mathcal {F}},\{\bar{\mathcal {F}}_t\}_{t>0},\bar{\mathbb{P}}),n_*,c_*,u_*,W_*)
$
constitutes a global martingale solution to system \eqref{CNS}. This completes the proof of the existence part. \qed

\vspace{2mm}\noindent
\textbf{Uniqueness of the global strong solution.} 
Let $(n_1,c_1,u_1)$ and $(n_2,c_2,u_2)$ be two global strong solutions to the system \eqref{CNS} with the same initial data $(n_0,c_0,u_0)$. For simplicity, we set
\begin{equation*}
\begin{split}
n^*=n_1-n_2,~c^*=c_1-c_2,~u^*=u_1-u_2, ~ P^*=P_1-P_2.
\end{split}
\end{equation*}
Then for $t\in(0,T]$, the triple $(n^*,c^*,u^*)$ satisfies the following system:
\begin{equation}\label{fin2}
\left\{
\begin{aligned}
&\mathrm{d}n^*+(u_1\cdot\nabla n_1-u_2\cdot\nabla n_2)\mathrm{d}t=[\Delta n^*-\nabla\cdot(n_1\nabla c_1)+\nabla\cdot(n_2\nabla c_2)+L(n_1)-L(n_2)]\mathrm{d}t,\\
&\mathrm{d}c^*+(u_1\cdot\nabla c_1-u_2\cdot\nabla c_2)\mathrm{d}t=(\Delta c^*-n_1c_1+n_2c_2)\mathrm{d}t,\\
&\mathrm{d} u^*+\textbf{P}[(u_1\cdot \nabla) u_1-(u_2\cdot \nabla) u_2]\mathrm{d}t =\mathbf{P}[\Delta u^*-\mathbf{g}(|u_1|^2)u_1+\mathbf{g}(|u_2|^2)u_2+  n^*\nabla\phi]\mathrm{d}t\\
&+\mathbf{P}[G(u_1)-G(u_2)]\mathrm{d}W(t),\\
&\nabla\cdot u^*=0,\\
&(n_0^*,c_0^*,u_0^*)=0.
\end{aligned}
\right.
\end{equation}
Since $n^*\in L^2(0,T;H^1(\mathbb{R}^3))$ and $\partial_tn^*\in L^2(0,T;L^2(\mathbb{R}^3))$, $\bar{\mathbb{P}}$-a.s., it follows that $\frac{1}{2}\frac{\mathrm{d}}{\mathrm{d}t}\|n^*\|_{L^2}^2=(\partial_tn^*,n^*)_{L^2}$. From the first equation of \eqref{fin2}, we obtain that $\bar{\mathbb{P}}$-a.s.
\begin{equation} \begin{split}\label{un-1}
&\frac{1}{2}\frac{\mathrm{d}}{\mathrm{d}t}\|n^*\|_{L^2}^2+\|\nabla n^*\|_{L^2}^2\\
&=(u^*n_1,\nabla n^*)_{L^2}+(n^*\nabla c_1,\nabla n^*)_{L^2}+(n_2\nabla c^*,\nabla n^*)_{L^2}+(L(n_1)-L(n_2),n^*)_{L^2}\\
&:=I_1+I_2+I_3+I_4.
 \end{split} \end{equation}
By using the H\"{o}lder inequality and the Sobolev embedding, we have
\begin{equation*} \begin{split}
|I_1|\leq\|u^*\|_{L^6}\|n_1\|_{L^3}\|\nabla n^*\|_{L^2}\leq\epsilon\|\nabla n^*\|_{L^2}^2+C_{\epsilon}\|n_1\|_{L^3}^2\|\nabla u^*\|_{L^2}^2,
 \end{split} \end{equation*}
and
\begin{equation*} \begin{split}
|I_2|\leq\epsilon\|\nabla n^*\|_{L^2}^2+C_{\epsilon}\|c_1\|_{H^3}^{2}\|n^*\|_{L^2}^2.
 \end{split} \end{equation*}
Thanks to the Sobolev embedding $H^1(\mathbb{R}^3)\hookrightarrow L^p(\mathbb{R}^3)$ for $p\in[2,6]$, we have
\begin{equation*} \begin{split}
|I_3|&\leq\epsilon\|\nabla n^*\|_{L^2}^2+C_{\epsilon}\|n_2\nabla c^*\|_{L^2}^2\leq\epsilon\|\nabla n^*\|_{L^2}^2+C_{\epsilon}\|n_2\|_{L^6}^2\|\nabla c^*\|_{L^3}^2\\
&\leq\epsilon\|\nabla n^*\|_{L^2}^2+C_{\epsilon}\|n_2\|_{H^1}^2\|c^*\|_{H^2}^2.
 \end{split} \end{equation*}
Moreover, there holds
\begin{equation*} \begin{split}
I_4&=-a\|n^*\|_{L^2}^2-\int_{\mathbb{R}^3}(n_1^2(x)+n_1(x)n_2(x)+n_2^2(x))(n^*(x))^2\mathrm{d}x\\
&\quad+(a+1)\int_{\mathbb{R}^3}(n_1(x)+n_2(x))(n^*(x))^2\mathrm{d}x\\
&\leq(a+1)\int_{\mathbb{R}^3}(n_1(x)+n_2(x))(n^*(x))^2\mathrm{d}x\leq C_a(\|n_1\|_{L^3}+\|n_2\|_{L^3})\|n^*\|_{L^6}\|n^*\|_{L^2}\\
&\leq\epsilon\|\nabla n^*\|_{L^2}^2+C_{\epsilon,a}(\|n_1\|_{L^3}^2+\|n_2\|_{L^3}^2)\|n^*\|_{L^2}^2.
 \end{split} \end{equation*}
Substituting the above estimates for $I_1$-$I_4$ into \eqref{un-1} and choosing $\epsilon>0$ small enough, we get
\begin{equation} \begin{split}\label{un-2}
&\frac{\mathrm{d}}{\mathrm{d}t}\|n^*\|_{L^2}^2+\|\nabla n^*\|_{L^2}^2\\
&\leq C_a(\|c_1\|_{H^3}^{2}+\|n_1\|_{L^3}^2+\|n_2\|_{L^3}^2)\|n^*\|_{L^2}^2+C\|n_1\|_{L^3}^2\|\nabla u^*\|_{L^2}^2+C\|n_2\|_{H^1}^2\|c^*\|_{H^2}^2.
\end{split} \end{equation}
In a similar manner, since $\frac{1}{2}\frac{\mathrm{d}}{\mathrm{d}t}\|\nabla n^*\|_{L^2}^2=\langle\nabla\partial_tn^*,\nabla n^*\rangle_{H^{-1},H^1}$, we deduce that $\bar{\mathbb{P}}$-a.s.
\begin{equation} \begin{split}\label{un-3}
&\frac{1}{2}\frac{\mathrm{d}}{\mathrm{d}t}\|\nabla n^*\|_{L^2}^2+\|\Delta n^*\|_{L^2}^2\\
&\leq\frac{1}{2}\|\Delta n^*\|_{L^2}^2+C\|\nabla n_1\|_{L^{3}}^2\|u^*\|_{L^6}^2+C\|u_2\|_{L^{\infty}}^2\|\nabla n^*\|_{L^2}^2+C\|\nabla c_1\|_{L^{\infty}}^2\|\nabla n^*\|_{L^2}^2\\
&\quad+C\|\Delta c_1\|_{L^{3}}^2\|n^*\|_{L^6}^2+C\|\nabla n_2\|_{L^{3}}^2\|\nabla c^*\|_{L^6}^2+C\|n_2\|_{L^{\infty}}^2\|\Delta c^*\|_{L^2}^2\\
&\quad+C_a(1+\|n_1\|_{L^{\infty}}^2+\|n_2\|_{L^{\infty}}^2)\|n^*\|_{L^2}^2+C(\|n_1\|_{L^{6}}^4+\|n_2\|_{L^{6}}^4)\|n^*\|_{L^6}^2\\
&\leq\frac{1}{2}\|\Delta n^*\|_{L^2}^2+C\|\nabla n_1\|_{L^{3}}^2\|\nabla u^*\|_{L^2}^2+C(\|u_2\|_{H^{2}}^2+\|c_1\|_{H^{3}}^2)\|\nabla n^*\|_{L^2}^2+C\|n_2\|_{H^2}^2\|\Delta c^*\|_{L^2}^2\\
&\quad+C_a(1+\|n_1\|_{H^{2}}^2+\|n_2\|_{H^{2}}^2)\|n^*\|_{L^2}^2+C(\|n_1\|_{H^{1}}^4+\|n_2\|_{H^{1}}^4)\|\nabla n^*\|_{L^2}^2.
 \end{split} \end{equation}

For the second equation of \eqref{fin2}, noting that $c^*\in L^2(0,T;H^3(\mathbb{R}^3))$ and $\partial_tc^*\in L^2(0,T;H^1(\mathbb{R}^3))$, $\bar{\mathbb{P}}$-a.s., it follows that $\frac{1}{2}\frac{\mathrm{d}}{\mathrm{d}t}\|c^*\|_{L^2}^2=(\partial_tc^*,c^*)_{L^2}$ and $\frac{1}{2}\frac{\mathrm{d}}{\mathrm{d}t}\|\Delta c^*\|_{L^2}^2=\langle\Delta\partial_tc^*,\Delta c^*\rangle_{H^{-1},H^1}$. By using the Sobolev inequality, one can deduce that $\bar{\mathbb{P}}$-a.s.
\begin{equation} \begin{split}\label{un-4}
&\frac{1}{2}\frac{\mathrm{d}}{\mathrm{d}t}\|c^*\|_{L^2}^2+\|\nabla c^*\|_{L^2}^2=(u^*c_1,\nabla c^*)_{L^2}-(n^*c_1,c^*)_{L^2}-(n_2c^*,c^*)_{L^2}\\
&\leq\|c_1\|_{L^{\infty}}\|u^*\|_{L^2}\|\nabla c^*\|_{L^2}+\|c_1\|_{L^{\infty}}\|n^*\|_{L^2}\|c^*\|_{L^2}+\|n_2\|_{L^3}\|c^*\|_{L^2}\|c^*\|_{L^6}\\
&\leq\frac{1}{2}\|\nabla c^*\|_{L^2}^2+C\|c_1\|_{L^{\infty}}^2\|u^*\|_{L^2}^2+C\|n^*\|_{L^2}^2+C(\|c_1\|_{L^{\infty}}^2+\|n_2\|_{L^3}^2)\|c^*\|_{L^2}^2,
\end{split} \end{equation}
and
\begin{equation} \begin{split}\label{un-5}
&\frac{1}{2}\frac{\mathrm{d}}{\mathrm{d}t}\|\Delta c^*\|_{L^2}^2+\|\nabla\Delta c^*\|_{L^2}^2\\
&\leq\frac{1}{2}\|\nabla\Delta c^*\|_{L^2}^2+C\|\nabla c_1\|_{L^{\infty}}^2\|\nabla u^*\|_{L^2}^2+C\|u^*\|_{L^6}^2\|D^2 c_1\|_{L^3}^2+C\|\nabla u_2\|_{L^{3}}^2\|\nabla c^*\|_{L^6}^2\\
&\quad+C\|u_2\|_{L^{\infty}}^2\|D^2 c^*\|_{L^2}^2+C\|c_1\|_{L^{\infty}}^2\|\nabla n^*\|_{L^2}^2+C\|\nabla c_1\|_{L^{\infty}}^2\|n^*\|_{L^2}^2+C\|\nabla n_2\|_{L^{2}}^2\|c^*\|_{L^{\infty}}^2\\
&\quad+C\|n_2\|_{L^{3}}^2\|\nabla c^*\|_{L^6}^2\\
&\leq\frac{1}{2}\|\nabla\Delta c^*\|_{L^2}^2+C\|c_1\|_{L^{\infty}}^2\|\nabla n^*\|_{L^2}^2+C\|c_1\|_{H^3}^2\|\nabla u^*\|_{L^2}^2+C\|c_1\|_{H^{3}}^2\|n^*\|_{L^2}^2\\
&\quad+C(\|u_2\|_{H^2}^2+\|\nabla n_2\|_{L^{2}}^2+\|n_2\|_{L^{3}}^2)\|c^*\|_{H^2}^2.
 \end{split} \end{equation}

For the fluid equation, by use a version of the It\^{o} formula proved in \cite[Lemma 1.4]{pardouxt1980stochastic}, we get
\begin{equation} \begin{split}\label{un-6}
&\frac{1}{2}\mathrm{d}\|u^*(t)\|_{L^2}^2+\|\nabla u^*(t)\|_{L^2}^2\mathrm{d}t=\Bigl[(u^*\otimes u_1,\nabla u^*)_{L^2}-(\mathbf{g}(|u_1|^2)u_1-\mathbf{g}(|u_2|^2)u_2,u^*)_{L^2}\\
&\quad+(n^*\nabla\phi,u^*)_{L^2}+\frac{1}{2}\|\mathbf{P}[G(u_1)-G(u_2)]\|_{\mathcal{L}_2(\mathbf{U};L^2)}^2\Bigl]\mathrm{d}t+([G(u_1)-G(u_2)]\mathrm{d}W(t),u^*)_{L^2}\\
&\leq\Bigl[\frac{1}{2}\|\nabla u^*\|_{L^2}^2+C(1+\|u_1\|_{L^{\infty}}^2)\|u^*\|_{L^2}^2+C\|\nabla\phi\|_{L^{\infty}}^2\|n^*\|_{L^2}^2\\
&\quad+|(\mathbf{g}(|u_1|^2)u_1-\mathbf{g}(|u_1|^2)u_2,u^*)_{L^2}|+|(\mathbf{g}(|u_1|^2)u_2-\mathbf{g}(|u_2|^2)u_2,u^*)_{L^2}|\Bigl]\mathrm{d}t\\
&\quad+\frac{1}{2}\|\mathbf{P}[G(u_1)-G(u_2)]\|_{\mathcal{L}_2(\mathbf{U};L^2)}^2\mathrm{d}t+([G(u_1)-G(u_2)]\mathrm{d}W(t),u^*)_{L^2}\\
&\leq\Bigl[\frac{1}{2}\|\nabla u^*\|_{L^2}^2+C(1+\|u_1\|_{L^{\infty}}^2+\|u_2\|_{L^{\infty}}^2)\|u^*\|_{L^2}^2+C\|\nabla\phi\|_{L^{\infty}}^2\|n^*\|_{L^2}^2\Bigl]\mathrm{d}t\\
&\quad+\frac{1}{2}\|\mathbf{P}[G(u_1)-G(u_2)]\|_{\mathcal{L}_2(\mathbf{U};L^2)}^2\mathrm{d}t+([G(u_1)-G(u_2)]\mathrm{d}W(t),u^*)_{L^2},\quad \bar{\mathbb{P}}\textrm{-a.s.}
\end{split} \end{equation}
Moreover, an application of the It\^{o} formula to $\|\nabla u^*(t)\|_{L^2}^2$ leads to
\begin{equation*}
\begin{split}
&\|\nabla u(t)\|_{L^2}^2=\|\nabla u_0\|_{L^2}^2\\
&\quad+2\int_0^t\left\langle\nabla(\Delta u(r)-(u(r)\cdot\nabla)u(r)-\mathbf{g}(|u(r)|^2)u(r)+n(r)\nabla\phi),\nabla u(r)\right\rangle_{H^{-1},H^1}\mathrm{d}r\\
&\quad+\int_0^t\|\mathbf{P}\nabla G(u(r))\|_{\mathcal{L}_2(\mathbf{U};L^2)}^2\mathrm{d}r+2\int_0^t (\nabla G(u(r))\mathrm{d}W(r),\nabla u(r))_{L^2},\quad \bar{\mathbb{P}}\textrm{-a.s.}
\end{split}
\end{equation*}
Thanks to the Sobolev inequality, we get from the last identity that
\begin{equation} \begin{split}\label{un-7}
&\frac{1}{2}\mathrm{d}\|\nabla u^*(t)\|_{L^2}^2+\|\Delta u^*(t)\|_{L^2}^2\mathrm{d}t\\
&\leq\Bigl[\frac{1}{2}\|\Delta u^*\|_{L^2}^2+C\|\nabla u_1\|_{L^3}^2\|u^*\|_{L^6}^2+C\|\nabla\phi\|_{L^{\infty}}^2\|n^*\|_{L^2}^2+C\|\mathbf{g}(|u_1|^2)u_1-\mathbf{g}(|u_2|^2)u_2\|_{L^2}^2\Bigl]\mathrm{d}t\\
&\quad+\frac{1}{2}\|\mathbf{P}\nabla[G(u_1)-G(u_2)]\|_{\mathcal{L}_2(\mathbf{U};L^2)}^2\mathrm{d}t+(\nabla[G(u_1)-G(u_2)]\mathrm{d}W(t),\nabla u^*)_{L^2}\\
&\leq\Bigl[\frac{1}{2}\|\Delta u^*\|_{L^2}^2+C\|u_1\|_{H^2}^2\|\nabla u^*\|_{L^2}^2+C\|\nabla\phi\|_{L^{\infty}}^2\|n^*\|_{L^2}^2+C(\|u_1\|_{L^6}^4+\|u_2\|_{L^6}^4)\|u^*\|_{L^6}^2\Bigl]\mathrm{d}t\\
&\quad+\frac{1}{2}\|\mathbf{P}\nabla[G(u_1)-G(u_2)]\|_{\mathcal{L}_2(\mathbf{U};L^2)}^2\mathrm{d}t+(\nabla[G(u_1)-G(u_2)]\mathrm{d}W(t),\nabla u^*)_{L^2}\\
&\leq\Bigl[\frac{1}{2}\|\Delta u^*\|_{L^2}^2+C\|u_1\|_{H^2}^2\|\nabla u^*\|_{L^2}^2+C\|\nabla\phi\|_{L^{\infty}}^2\|n^*\|_{L^2}^2+C(\|u_1\|_{H^1}^4+\|u_2\|_{H^1}^4)\|\nabla u^*\|_{L^2}^2\Bigl]\mathrm{d}t\\
&\quad+\frac{1}{2}\|\mathbf{P}\nabla[G(u_1)-G(u_2)]\|_{\mathcal{L}_2(\mathbf{U};L^2)}^2\mathrm{d}t+(\nabla[G(u_1)-G(u_2)]\mathrm{d}W(t),\nabla u^*)_{L^2},\quad \bar{\mathbb{P}}\textrm{-a.s.}
\end{split} \end{equation}
Putting the estimates \eqref{un-2}-\eqref{un-7} together, we arrive at
\begin{equation} \begin{split}\label{un-8}
&\mathrm{d}\left(\|n^*(t)\|_{H^1}^2+\|c^*(t)\|_{H^2}^2+\|u^*(t)\|_{H^1}^2\right)\\
&\leq C_a\Bigl[\mathbf{N}(t)(\|n^*(t)\|_{H^1}^2+\|c^*(t)\|_{H^2}^2+\|u^*(t)\|_{H^1}^2)\Bigl]\mathrm{d}t+C\|G(u_1)-G(u_2)\|_{\mathcal{L}_2(\mathbf{U};H^1)}^2\mathrm{d}t\\
&\quad+C(G(u_1)-G(u_2)\mathrm{d}W(t),u^*)_{H^1},\quad \bar{\mathbb{P}}\textrm{-a.s.}
\end{split} \end{equation}
where
\begin{equation*} \begin{split}
\mathbf{N}(t):=1+\|\nabla\phi\|_{L^{\infty}}^2+\sum_{i=1}^{2}(\|n_i(t)\|_{H^1}^4+\|n_i(t)\|_{H^2}^2+\|c_i(t)\|_{H^3}^{2}+\|u_i(t)\|_{H^{1}}^4+\|u_i(t)\|_{H^{2}}^2).
\end{split} \end{equation*}
Define $\textbf{t}^R =\textbf{t}^R_1\wedge\textbf{t}^R_2\wedge T$ with
\begin{equation*}
\begin{split}
\textbf{t}^R_i:=&\inf \left\{t>0;~\sup _{r\in[0, t]} \|n_i(r) \|_{H^1}^2 \vee \int_0^t \|n_i(r) \|_{H^2}^2 \mathrm{d}r \vee \sup _{r\in[0, t]} \|c_i(r) \|_{H^2}^2\right.\\
 &\left.\vee \int_0^t\|c_i(r) \|_{H^3}^2\mathrm{d}r \vee \sup _{r\in[0, t]} \|u_i(r) \|_{H^1}^2 \vee \int_0^t \|u_i(r) \|_{H^2}^2\mathrm{d}r \geq R\right\},\quad i=1,~2.
\end{split}
\end{equation*}
Since $ (n_i, c_i, u_i )$, $i=1,~2$ are global solutions of \eqref{CNS}, the continuity-in-time of the solutions implies that $\textbf{t}^R \nearrow T$  as $R\rightarrow\infty$, $\bar{\mathbb{P}}$-a.s. By using Assumption \ref{as3}, we infer from \eqref{un-8} that
\begin{equation} \begin{split}\label{un-9}
&\|n^*(t\wedge\textbf{t}^R)\|_{H^1}^2+\|c^*(t\wedge\textbf{t}^R)\|_{H^2}^2+\|u^*(t\wedge\textbf{t}^R)\|_{H^1}^2\\
&\leq C_{a,R}\int_0^{t\wedge\textbf{t}^R}\mathbf{N}(r)(\|n^*(r)\|_{H^1}^2+\|c^*(r)\|_{H^2}^2+\|u^*(r)\|_{H^1}^2)\mathrm{d}r\\
&\quad+C\left|\int_0^{t\wedge\textbf{t}^R}(G(u_1(r))-G(u_2(r))\mathrm{d}W(r),u^*(r))_{H^1}\right|.
\end{split} \end{equation}
Applying the Gronwall lemma to \eqref{un-9} and noting the fact that $\int_0^{t\wedge\textbf{t}^R}\mathbf{N}(r)\mathrm{d}r\leq C_{R,\phi,T}$, we obtain
\begin{equation} \begin{split}\label{un-10}
&\|n^*(t\wedge\textbf{t}^R)\|_{H^1}^2+\|c^*(t\wedge\textbf{t}^R)\|_{H^2}^2+\|u^*(t\wedge\textbf{t}^R)\|_{H^1}^2\\
&\leq C_{a,R,\phi,T}\left|\int_0^{t\wedge\textbf{t}^R}(G(u_1(r))-G(u_2(r))\mathrm{d}W(r),u^*(r))_{H^1}\right|.
\end{split} \end{equation}
By using the BDG inequality, it follows from \eqref{un-10} that
\begin{equation*} \begin{split}
&\bar{\mathbb{E}}\sup_{r\in[0,t\wedge\textbf{t}^R]}(\|n^*(r)\|_{H^1}^2+\|c^*(r)\|_{H^2}^2+\|u^*(r)\|_{H^1}^2)\\
&\leq C_{a,R,\phi,T}\bar{\mathbb{E}}\left(\int_0^{t\wedge\textbf{t}^R}\|G(u_1(r))-G(u_2(r))\|_{\mathcal{L}_2(\mathbf{U};H^1)}^2\|u^*(r)\|_{H^1}^2\right)^{\frac{1}{2}}\\
&\leq \frac{1}{2}\mathbb{E}\sup_{r\in[0,t\wedge\textbf{t}^R]}\|u^*(r)\|_{H^1}^2+C_{a,R,\phi,T}
\int_0^{t}\bar{\mathbb{E}}\sup_{s\in[0,r\wedge\textbf{t}^R]}\|u^*(s)\|_{H^1}^2\mathrm{d}r,
\end{split} \end{equation*}
which implies $\bar{\mathbb{E}}\sup_{r\in[0,t\wedge\textbf{t}^R]}(\|n^*(r)\|_{H^1}^2+\|c^*(r)\|_{H^2}^2+\|u^*(r)\|_{H^1}^2)=0$. After taking the limit $R\rightarrow\infty$, we conclude that
$
\|n^*(t)\|_{H^1}^2+\|c^*(t)\|_{H^2}^2+\|u^*(t)\|_{H^1}^2=0$, for all $t\in[0,T]$ and $\bar{\mathbb{P}}$-a.s. This completes the proof of the uniqueness part.

\vspace{2mm}\noindent
\textbf{Proof of Theorem \ref{th1}.}  According to the classical Yamada-Watanabe theorem \cite{da2014stochastic,watanabe1971uniqueness,yamada1971uniqueness} and the Gy\"{o}ngy-Krylov characterization of the convergence in probability \cite{gyongy1980stochastic}, one can prove the existence and uniqueness of pathwise solutions, provided the existence of martingale solutions and the pathwise uniqueness as we established before. The proof of Theorem \ref{th1} is thus completed.\qed

\section*{Data availability}

No data was used for the research described in the article.

\section*{Conflict of interest statement}

The authors declared that they have no conflicts of interest to this work.

\section*{Acknowledgements}

This work was partially supported by the National Key Research and Development Program of China (Grant No.  2023YFC2206100), and the National Natural Science Foundation of China (Grant No. 12231008).

\bibliographystyle{amsrefs}
\bibliography{SCNS-3D}

\end{document}